\documentclass[11pt,reqno,UTF8,twoside]{amsart}
\usepackage{mathrsfs}
\usepackage{amsfonts,amssymb,amsmath,amsthm}
\usepackage{cite}
\usepackage[colorlinks=true,citecolor=red]{hyperref}
\usepackage{titletoc}
\usepackage{geometry}
\usepackage{bm}
\usepackage{indentfirst}
\usepackage{graphicx}
\usepackage{float} 
\usepackage{booktabs}
\usepackage{longtable}
\usepackage{subfigure}
\usepackage{stmaryrd}
\usepackage{enumerate}
\usepackage{tikz}
\usepackage{ulem}
\usepackage{color,soul}
\usepackage{setspace}
\usepackage{stmaryrd}
\usepackage[pagewise]{lineno} 
\allowdisplaybreaks[4]
\usepackage{appendix}
\usepackage{cases}
\usepackage{todonotes}
\usepackage{enumitem}

\numberwithin{equation}{section}
\newtheorem{theorem}{Theorem}[section]
\newtheorem{lemma}{Lemma}[section]
\newtheorem{corollary}{Corollary}[section]
\newtheorem{proposition}{Proposition}[section]
\numberwithin{equation}{section}
\newtheorem{remark}{Remark}[section]
\newtheorem{definition}{Definition}[section]
\newtheorem{thm}{Theorem}

\newcommand{\N}{\mathbb{N}}
\newcommand{\R}{\mathbb{R}}
\newcommand{\Z}{\mathbb{Z}}
\newcommand{\E}{\mathbb{E}}
\newcommand{\Pb}{\mathbb{P}}

\hypersetup{linkcolor=blue}

\makeatletter
\@namedef{subjclassname@2020}{\textup{2020} Mathematics Subject
	Classification}
\makeatother

\begin{document}

	\title[ Mixing for random  nonlinear wave equations]
	{{\large E{\MakeLowercase{xponential mixing for random nonlinear wave equations:}}\\ 
 \vspace{1mm}
 {\MakeLowercase{weak dissipation and localized control}}}
 } 

 \author[Z. Liu, D. Wei, S. Xiang, Z. Zhang, J.-C. Zhao]{Z\MakeLowercase{iyu} L\MakeLowercase{iu}, D\MakeLowercase{ongyi} W\MakeLowercase{ei}, S\MakeLowercase{hengquan} X\MakeLowercase{iang}, Z\MakeLowercase{hifei} Z\MakeLowercase{hang}, J\MakeLowercase{ia-}C\MakeLowercase{heng} Z\MakeLowercase{hao}}
		
\address[Ziyu Liu]{School of Mathematical Sciences, Peking University, 100871, Beijing, China.}
\email{liuziyu@math.pku.edu.cn}

\address[Dongyi Wei]{School of Mathematical Sciences, Peking University, 100871, Beijing, China.}
\email{jnwdyi@pku.edu.cn}

\address[Shengquan  Xiang]{School of Mathematical Sciences, Peking University, 100871, Beijing, China.}
\email{shengquan.xiang@math.pku.edu.cn}

\address[Zhifei  Zhang]{School of Mathematical Sciences, Peking University, 100871, Beijing, China.}
\email{zfzhang@math.pku.edu.cn}

\address[Jia-Cheng Zhao]{School of Mathematical Sciences, Peking University, 100871, Beijing, China.}
\email{zhaojiacheng@math.pku.edu.cn}

\subjclass[2020]{
37A25,  
37S15, 
37L50, 
35L71,  
93C20. 
    }
	
\keywords{Exponential mixing; Nonlinear wave equations; Global stability; Asymptotic compactness; Controllability}
	
\begin{abstract}

We establish a new criterion for exponential mixing of random dynamical systems. Our criterion is applicable to a wide range of systems,  including in particular dispersive equations. Its verification is in nature related to several topics, i.e., asymptotic compactness in dynamical systems, global stability of evolution equations, and localized control problems.

As an initial application, we exploit the exponential mixing of random nonlinear wave equations with degenerate damping,  critical nonlinearity, and physically localized  noise. The essential challenge lies in the fact that the weak dissipation and randomness interact in the evolution. 
\end{abstract}

\maketitle
\setcounter{tocdepth}{1}
\tableofcontents

\section{Introduction}
	
\renewcommand{\thethm}{\Alph{thm}}

The ergodic and mixing properties, crucial for the understanding of random systems, have been the focus of research yielding significant advancements in recent decades \cite{HM-06,HM-08,BBPS-22b,KNS-20,Shi-15}.
However, there have been few results achieved for dispersive equations. The analysis in this setting is usually delicate in the absence of smoothing effect; the existing criteria valid for parabolic-type equations are hardly applicable.

\begin{center}
{\it Does the mixing property hold for general dispersive equations?}
\end{center}

\noindent We provide a criterion of exponential mixing for random dynamical systems in general Polish space, i.e. Theorem~\ref{Thm-informal}. This result is an attempt to seek for sharp sufficient conditions for the exponential mixing of dispersive equations, as an affirmative answer to the above question. Especially, the criterion, composed by {\it asymptotic compactness, irreducibility} and {\it coupling condition}, is closely related to dynamical system, dispersive equations and control theory.

\noindent As an initial application of the criterion, we establish the exponential mixing for a general model of nonlinear wave equations in the form
\begin{equation*}
     \boxempty u+a(x)\partial_t u+u^3=\eta,
\end{equation*}
i.e. Theorem~\ref{Thm-informal-model}, where $a(x)$ induces the damping effect, and $\eta$ stands for the random noise.
The generality mentioned encompasses several aspects, including  {\it degenerate/localized damping, critical nonlinearity}\footnote{In the context of $n$-dimensional wave equations, the Sobolev-critical exponent of nonlinearity is $n/(n-2)$ for $n\geq 3$ (see, e.g., \cite{BV-92}), which differs from the energy-critical exponent $(n+2)/(n-2)$ (see, e.g., \cite{BSS-09}). This is justified by the Sobolev embedding $H^1\hookrightarrow L^{2n/(n-2)}$, implying that if a nonlinear function satisfies a polynomial growth with power not exceeding $n/(n-2)$, then its Nemytski\u{\i} operator maps $H^1$ into $L^2$. While we focus on the cubic nonlinearity that is Sobolev-critical, our results and their proofs should be adaptable to the case of super-cubic nonlinearity.}  and {\it  random noise localized in physical space}. In particular, the weak dissipation mechanism induced by the 
localized damping, mingled with the random perturbations, contributes to part of the main challenges  in the research; see Sections \ref{Section-1-2},\ref{Section-1-3} later. We believe that the approach is  general and adaptable to other types of dispersive equations.

\vspace{0.4em}

In the sequel, let us give a sketch of those topics involved in the criterion:

\begin{itemize}[leftmargin=3em]
\item[$(1)$] Asymptotic compactness is a fundamental object in the theory of global attractor for {\it dynamical systems}, motivated by the issues in turbulence \cite{FP-67,Lady-72}. In this topic the dispersive setting is fairly subtle due to the lack of  smoothing effect \cite{BV-92,Hale-88}. In addition, the
localization of damping and randomness lead to extra obstacles in our analysis.
\vspace{0.2em}
\item[$(2)$] The issue of irreducibility will be reduced to a  stability problem, where the latter is a significant topic in the dynamics of {\it dispersive  equations} \cite{Bardos-92,Grillakis-90,LS-23,KT-95,LT-12}.  

\vspace{0.2em}

\item[$(3)$]  The coupling condition corresponds to the stabilization which is one of the central problems in {\it control theory}  \cite{Lions-88,Coron-07}. Our analysis of coupling condition involves various objects, including unique continuation, Carleman estimates, Hilbert uniqueness method and the localized dissipation, constituting a long piece  of section in this paper. 
\end{itemize}

\vspace{0.3em}

Below in Section~\ref{Section-1-1} we give an overview of the abstract criterion (i.e. Theorem~\ref{Thm-informal}), including historical backgrounds and main contributions.
In Section~\ref{Section-1-2} we present the mixing result for the random wave equations (i.e. Theorem~\ref{Thm-informal-model}), and discuss its generality. 
Section~\ref{Section-1-3} outlines the proof of Theorem~\ref{Thm-informal-model}, highlighting the main challenges and our approaches. A brief outline of the rest of the paper is available in  Section~\ref{Section-1-4}.  

\vspace{0.2em}

\subsection{Probabilistic framework}\label{Section-1-1}

In this section we introduce a new criterion for exponential mixing of random dynamical systems. This criterion is a consequence of inspiration from the prior related frameworks and the observation on asymptotic compactness from the dynamical system point of view. It is applicable to a wide class of dispersive equations.

\subsubsection{Historical  backgrounds}

The study of ergodic and mixing properties for randomly forced equations has been a principal motivation  of ergodic theory for Markov processes.  In particular, it has led to significant results for the 2D Navier--Stokes systems; for the early
achievements; see, e.g., \cite{BKL-02,EMS-01,FM-95,MY-02,Hairer-02,Mattingly-02,KS-00,EM-01}. In recent years, Hairer and Mattingly \cite{HM-06,HM-08} introduce the asymptotic strong Feller property to provide a first result for the situation when the noise is white in time and is extremely degenerate in Fourier modes.  More recently, Kuksin--Nersesyan--Shirikyan \cite{KNS-20} propose a controllability property to establish a similar result when the degenerate random forces are coloured in time. The reader is referred to, e.g., \cite{HM-11b,Ner-24,FGRT-15,BGN-23,KNS-20-1} for other contributions in the context of extremely degenerate noise. In \cite{Shi-15,Shi-21}, Shirikyan invokes another  controllability approach to study the case in which the random perturbation is localized in the physical space. In the context of unbounded domains, the recent paper \cite{Ner-22} by Nersesyan derives exponential mixing by developing the controllability approaches
of the papers \cite{KNS-20,Shi-21}.

There have been several general approaches applied to the ergodic and mixing properties for various models.  For instance, Hairer--Mattingly--Scheutzow \cite{HMS-11} formulate a generalized form of Harris theorem \cite{Harris-56} (see also \cite{MT-09,HM-11} for a detailed account), providing a criterion for exponential mixing and applying it to stochastic delay equations. We refer the reader to \cite{BKS-20,HM-08,GM-05} for some applications for stochastic parabolic equations and modifications of the Harris-type results. Another intensively studied approach is the coupling method, developed in \cite{Hairer-02,MY-02,Mattingly-02,KPS-02,KS-01,KS-02}. Based on the coupling method, Kuksin and Shirikyan \cite{KS-12,Shi-08} propose general conditions, i.e., {\it recurrence} and {\it squeezing}, for mixing properties. Some applications and extensions for both ODE and PDE models of such framework can be found in, e.g., \cite{Martirosyan-14, Shi-15,Shi-21,Shi-17}.

\subsubsection{Obstructions for mixing of dispersive equations, an idea from dynamical systems}

In the context of dispersive equations, the main difficulty lies in the non-compactness of the resolving operator, which results from the lack of the smoothing effect. This leads to an aftermath that the aforementioned frameworks for mixing properties seem hardly applicable to the dispersive setting. For instance, the squeezing \cite{KS-12} usually requires  extra regularity of the target trajectory. Analogous obstacles appear to the discussion of the asymptotic strong Feller property \cite{HM-06}, approximate controllability  \cite{CXZ-2023, KNS-20},  etc. Accordingly, our research starts with a question,

\begin{center}
{\it How to compensate for the absent compactness?}
\end{center}

Our answering this question employs the notion of asymptotic compactness from the dynamical system theory. Recall that the mixing property
describes a certain type of limiting behavior that a physical system asymptotically converges  to a statistical equilibrium in the distribution sense. Accordingly, one may relax the  compactness requirement and provide an alternative of a limiting form.  At the same time, the theory of global attractor for infinite-dimensional dynamical system involves a viewpoint of asymptotic compactness, illustrating such limiting-type compactness \cite{BV-92,Hale-88}. These motivate us to build up an explicit relation between the asymptotic compactness for possibly non-compact semiflow and the mixing property.

\subsubsection{A general framework}

Let $\mathcal{X}$ and $\mathcal{Z}$ be Polish spaces, and denote by $d$ the metric on $\mathcal X$. Let $S\colon\mathcal{X} \times\mathcal{Z} \rightarrow \mathcal{X}$ be a continuous mapping, and $\{\xi_n;n\in\N\}$ a  sequence of $\mathcal{Z}$-valued independent identically distributed (i.i.d. for abbreviation) random variables with a common law $\ell$.  We consider a \textit{random dynamical system} defined by 
\begin{equation}\label{RDS}
x_{n+1}=S(x_{n},\xi_{n}),\quad n\in\N,
\end{equation}
 with initial condition
\begin{equation}\label{initial condition}
x_0=x.
\end{equation}
We proceed to describe our abstract result for system $\{x_n;n\in\N\}$, omitting some inessential technical details. Assume first that $\ell$ is compactly supported, and the mapping $S$ is Lipschitz on any bounded set. 
The essential hypotheses are roughly stated as follows:

\begin{enumerate}[leftmargin=2em]
\item[$\mathbf{(H)}$]  {\it  
\begin{enumerate}[leftmargin=1.5em]
\item[$(a)$]{\rm (}Asymptotic compactness{\rm )} There exists a compact subset $\mathcal Y$ of $\mathcal X$  such that $ \{x_n;n\in\N\}$  exponentially converges to $\mathcal Y$ in a pathwise manner. We further denote the attainable set from $\mathcal Y$ by $\mathcal Y_\infty$ {\rm(}see Definition {\rm\ref{Def-attainable})}. 
\medskip
\item[$(b)$]{\rm (}Irreducibility{\rm )} There exists $z\in\mathcal Y$ with the following property: for every $\varepsilon>0$, there is $m\in\N^+$ and $p>0$ such that for any $x\in\mathcal{Y}_\infty$, 
$$\Pb (d(x_m,z)<\varepsilon)\geq p.$$
\item[$(c)$] {\rm (}Coupling condition{\rm )} For every $x,x'\in\mathcal Y_\infty$, the pair $(x_1,x_1')$ admits a coupling $(\mathcal R,\mathcal R')$ satisfying 
$$
\Pb(d(\mathcal{R},\mathcal{R}')>\tfrac{1}{2}d(x,x'))\leq Cd(x,x'),
$$
where $x_1'$ is defined as in {\rm(\ref{RDS}),(\ref{initial condition})} with $x$ replaced by $x'$.
\end{enumerate} 
}
\end{enumerate}

\medskip

It is worth mentioning that the hypotheses of irreducibility and coupling condition are directly inspired by the previous works \cite{HM-11b,EM-01} and \cite{Shi-15,Shi-21}, respectively. See Section \ref{Section-Probalistic part} for more information.

The following result is a simplified version of our criterion for exponential mixing.  See Section~\ref{Section-generalresult} for a rigorous description of this criterion, where the hypotheses are more general to some extent.

\begin{thm}\label{Thm-informal} 
{\it Assume that hypothesis $\mathbf{(H)}$ holds. Then the Markov process $\{x_n;n\in\N\}$, defined by {\rm(\ref{RDS}),(\ref{initial condition})}, has a unique invariant measure $\mu_*$ on $\mathcal X$. Moreover, $\mu_*$ is exponential mixing, i.e., there exists a constant $\beta>0$ such that
\begin{equation*}
\|\mathscr D(x_n)-\mu_*\|_L^*\leq  C(x)e^{-\beta n}
\end{equation*}  
for any $x\in \mathcal{X}$ and $n\in\N$, where $\|\cdot\|^*_L$ denotes the dual-Lipschitz distance on $\mathcal X$ and $\mathscr{D}(x_n)$ stands for the law of $x_n$.}
\end{thm}

The ergodic and mixing properties involved in Theorem \ref{Thm-informal} play a significant role in understanding its asymptotic behavior of random dynamical system, which have been applied to the K41 theory \cite{BBPS-22,GK-22}, stochastic quantization \cite{TW-18}, chaotic behavior \cite{BBPS-22d,BBPS-21}, and among others. Besides, exponential mixing is fundamental to a number of statistical consequences, including the law of large numbers, central limit theorems and large deviations \cite{KW-12,DZ-10}.

\begin{remark}
A main contribution of the present criterion is to reduce explicitly the issue of mixing property to a restricted system on a compact phase space. This reduction provides in particular a solution for the requirement of extra regularity in squeezing/stabilization problems, in the context of dispersive equations.
Another  contribution is to establish a connection between the mixing property and other topics in various research fields, so that the related methodologies are available for the ergodicity problems. 

To be more precise, the verification of asymptotic compactness can be accomplished by invoking the ideas in
the theories of global attractors  {\rm(}see, e.g., {\rm\cite{BV-92,Hale-88})}. Meanwhile, in many circumstances of PDEs, the irreducibility can be proved by means of either the global stability of free dynamics {\rm\cite{HM-06,KNS-20,Shi-15}} or the approximate controllability of   associated system {\rm\cite{KNS-20-1,GHM-18}}. Also, inspired by the parabolic case {\rm(}see, e.g., {\rm\cite{Shi-15,Shi-21})},  a possible approach for verifying the coupling hypothesis includes the arguments from control theory {\rm\cite{Coron-07}}.

Conceivably, the criterion presented here is applicable to a wide range of dissipative  equations, especially, while the aforementioned topics have been well developed for this type of models.
\end{remark}

\subsection{Random wave equations}\label{Section-1-2} 

Let $D$ be a bounded domain in $\R^3$ having smooth boundary $\partial D$. The model under consideration reads 
\begin{equation}\label{Problem-0}
\begin{cases}
\boxempty u+a(x)\partial_t u+u^3=\eta(t,x), \quad x\in D, \\ 
u|_{\partial D}=0, \\
u[0]=(u_0,u_1):=u^0,
\end{cases}
\end{equation}
where the notation $\boxempty:=\partial_{tt}^2-\Delta$ stands for the d'Alembert operator, and  $u[t]:=(u,\partial_t u)(t)$. Our settings for the damping coefficient $a(x)$ and random noise $\eta(t,x)$ are stated in $(\mathbf{S1})$ and $(\mathbf{S2})$ below, respectively.

Let $\{\lambda_j;j\in\N^+\}$ be the eigenvalues of $-\Delta$ with the Dirichlet condition, satisfying $\lambda_{j+1}\geq \lambda_{j}$. The eigenvectors corresponding to $\lambda_j$ are denoted by $e_j$, which form an orthonormal basis of $L^2(D)$. 
We denote by $H^s\ (s>0)$ the domain of fractional power $(-\Delta)^{s/2}$, and write $H=L^2(D)$. Setting $\mathcal H^s=H^{1+s}\times H^s$, the phase space of (\ref{Problem-0}) is specified as $\mathcal H:=\mathcal H^0$. 
We define the energy functional $E:\mathcal H\rightarrow\R^+$ as
\begin{equation}\label{energy-functional}
E(\psi)=
\frac{1}{2}\int_D\left[|\nabla \psi_0(x)|^2+
\psi^2_1(x)+\frac{1}{2} \psi^4_0(x)\right],\quad \psi=(\psi_0,\psi_1).
\end{equation}
The energy for a  solution $u$ is expressed as 
$E_u(t):=E(u[t])$.\\
Let $\{\alpha_k;k\in\N^+\}$ denote a smooth orthonormal basis of $L^2(0,1)$. It induces a sequence of functions $\alpha_k^{\scriptscriptstyle T}(t)=\frac{1}{\sqrt{T}}\alpha_k(\frac{t}{T})$, forming an orthonormal basis of $L^2(0,T)$.

\vspace{0.6em}

In Section~\ref{Section-121} below, we provide a brief statement of our setting and main result.
Further discussions of the result are then contained in Section \ref{Section-122}.

\subsubsection{Main result}\label{Section-121}

We introduce a notion of $\Gamma$-type domain which is initially used by Lions \cite{Lions-88}.  
Such a geometric setting will be  involved both in the degeneracy/localization of $a(x)$ and the structure of $\eta(t,x)$. 

\begin{definition}\label{Def-Gamma}
A $\Gamma$-type domain is a subdomain of $D$ in the form
$$
N_\delta(x_0):=\left\{
x\in D;|x-y|<\delta\ {\it for\ some\ }y\in\Gamma(x_0)
\right\},
$$
where $x_0\in\mathbb R^3\setminus \overline D,\,\delta>0$ and $\Gamma(x_0)=\{x\in\partial D;(x-x_0)\cdot n(x)>0\}.$ 
\end{definition}

\begin{enumerate}[leftmargin=2.5em]
\item[$\mathbf{(S1)}$] ({\bf Localized structure}) {\it The function $a(\cdot)\in C^\infty(\overline{D})$ is non-negative, and there exists a $\Gamma$-type domain $N_\delta(x_0)$ and a constant $a_0>0$ such that
\begin{equation}\label{Damping-localization}
a(x)\geq a_0,\quad \forall\,x\in N_\delta(x_0).
\end{equation}
Meanwhile, let $\chi(\cdot)\in C^\infty(\overline{D})$ satisfy that there exists a $\Gamma$-type domain $N_{\delta'}(x_1)$ and a constant $\chi_0>0$ such that
\begin{equation}\label{Gamma-condition}
\chi(x)\geq \chi_0,\quad\forall\,x\in  N_{\delta'}(x_1).
\end{equation}
}
\end{enumerate}

\vspace{0.6em}

Clearly, setting $(\mathbf{S1})$ covers the case where $a(x)\equiv a_0$ and $\chi(x)\equiv\chi_0$. Moreover, it would determine a quantity $\mathbf{T}= \mathbf{T}(D, a, \chi)>0$, which will be taken as a lower bound for time spread of the random noise $\eta(t,x)$; see Section {\rm\ref{Section-conclusion}} for more information.

\vspace{0.6em}
\begin{enumerate}[leftmargin=2.5em]
\item[$\mathbf{(S2)}$]  
{\it 
Let $\boldsymbol{\rho}=\{\rho_{jk};j,k\in\N^+\}$ be a sequence of probability density functions supported by $[-1,1]$, which is $C^1$ and satisfies  $\rho_{jk}(0)>0$. }
\end{enumerate}

\vspace{0.6em}

Given any $T>0$ and $\{b_{jk};j,k\in\N^+\}$, a sequence of nonnegative numbers,  the random noise $\eta(t,x)$ under consideration is of the form
\begin{equation}\label{structure-noise}
\begin{aligned}
&\eta(t,x)= \eta_n(t-nT,x),\quad t\in [nT,(n+1)T),\,n\in\N,\\ 
&\displaystyle\eta_n(t,x)=\chi(x)\sum_{j,k\in\N^+}b_{jk}\theta^{n}_{jk}\alpha^{\scriptscriptstyle T}_k(t)e_j(x),\quad t\in[0,T),
\end{aligned}
\end{equation}
where $\{\theta_{jk}^{n};n\in\N \}$ is a sequence of i.i.d.~random variables with density $\rho_{jk}$.

\vspace{0.6em}

Consider the deterministic version of (\ref{Problem-0}), reading 
\begin{equation}\label{semilinear-problem}
\begin{cases}
\boxempty u+a(x)\partial_t u+u^3=f(t,x),\quad x\in D,\\
u[0]=(u_0,u_1)=u^0,
\end{cases}
\end{equation}
equipped with  Dirichlet condition as in (\ref{Problem-0})\footnote{All of the wave equations arising in the remainder of this paper, which may be positioned in various settings of stochastic problems, non-autonomous dynamical systems and controlled systems, will be supplemented by the Dirichlet condition, without any explicit mention.},  where $f\colon[0,T]\rightarrow H$ (or $f\colon\R^+\rightarrow H$) denotes a deterministic force. We then define a continuous mapping by 
\begin{equation}\label{definition-maps}
S\colon\mathcal H\times L^2(D_T)\rightarrow \mathcal H,\quad S(u^0,f)=u[T],
\end{equation}
where $u\in C([0,T];H^1)\cap C^1([0,T];H)$ stands for the unique solution of (\ref{semilinear-problem}).
Then, (\ref{Problem-0}) defines a Markov process $\{u^n;n\in\N\}$  with random initial data\footnote{
The use of random  data aims at improving the level of generality for our result on (\ref{Problem-0}), which is more general than the setting involved 
in Theorem \ref{Thm-informal}. Recall that the initial data of $\{x_n;n\in\N\}$ in (\ref{RDS}),(\ref{initial condition}) are deterministic, which makes it more convenient for us to formulate the abstract hypothesis $(\mathbf{H})$.} by
\begin{equation}\label{setting-wave}
\begin{cases}
     u^{n+1}=S(u^n,\eta_n),\quad n\in\N,\\
    u^0 \text{ is an }\mathcal{H}\text{-valued random variable independent of }\{\eta_n;n\in\N\}.
\end{cases}   
\end{equation}

Our result of exponential mixing for (\ref{Problem-0}) is contained in the following. 
 
\begin{thm}\label{Thm-informal-model}
Assume that $a(x),\,\chi(x),\,\boldsymbol{\rho}$ satisfy settings  $(\mathbf{S1})$ and $(\mathbf{S2})$. For every $T>\mathbf{T}$ and ${B}_0>0$, there exists a constant  $N\in\N^+$ such that if the sequence $\{b_{jk};j,k\in\N^+\}$ in {\rm(\ref{structure-noise})} satisfies
\begin{equation}\label{structure-noise-2}
\sum_{j,k\in\N^+}b_{jk}\lambda_j^{2/7}\|\alpha_k\|_{_{L^\infty(0,1)}} \leq {B}_0T^{1/2}\quad\text{and}\quad  b_{jk}\neq0\ \text{ for } 1\leq j,k\leq N,
\end{equation}
then the Markov process $\{u^n;n\in\N\}$ has a unique invariant measure $\mu_*$ on $\mathcal H$ with compact support. Moreover, $\mu_*$ is exponential mixing, i.e., there exist  constants $C, \beta>0$ such that
\begin{equation*}
\|\mathscr{D}(u^n)-\mu_*\|_{L}^*\leq C e^{-\beta n} \Big(1+ \int_{\mathcal{H}}E(v)\nu(dv)\Big)
\end{equation*}
for any random initial data $u^0$ with law $\nu$ and $n\in \N$. 
\end{thm}

See Section \ref{Section-conclusion} for the proof of Theorem \ref{Thm-informal-model}, which will be eventually accomplished after a long series of preparations constituting the bulk of the present paper (see Sections \ref{Section-Probalistic part}-\ref{Section-control}). 

We also mention that recent years have witnessed a considerable interest on random dispersive equations, which involves many topics, such as random data theory \cite{BT-08,BT-08b}, wave turbulence \cite{DH-23,DH-23a,BGHS-21}, Gibbs measure \cite{Bourgain-96,BDNY-24,DNY-appear}, etc. Our result, concerning the exponential mixing for random wave equations, falls into such a category.

To the best of our knowledge, there are few results concerning the ergodicity and mixing for wave equations (and even for other types of dispersive equations). The lack of the smoothing
effect for these equations can partly explain this situation. The existing literature concentrates on the case where the equations are damped-driven on the entire domain and white-forced in time, where the Foia\c{s}--Prodi estimates may be available. See, e.g., \cite{Martirosyan-14,Martirosyan-17} for wave equations and \cite{GHMR-23,DO-05,BFZ-23} for other dispersive equations. 
We also refer the reader to \cite{FT-24,Tolomeo-20,Tolomeo-23} for the results on wave equations in the context of stochastic quantisation.

\begin{remark}\label{remark-damping}
Notably, in Theorem {\rm\ref{Thm-informal-model}} the coefficient $a(x)$ is allowed to vanish outside a subdomain of $D$.
Such degeneracy/localization of damping contributes partly to the novelty of our framework. Roughly speaking,
\begin{itemize}[leftmargin=2em]
\item[$(1)$] the relevant mathematical theories have important application background; 

\item[$(2)$] the presence of localized damping here results from the exploration of sharp sufficient conditions for ergodicity and mixing of wave equations;

\item[$(3)$] the central problem involved is whether the localized dissipation induced by damping can spread to the whole system,  reflected in several essential issues related to dynamical system, global stability and controllability for {\rm(\ref{Problem-0}),(\ref{semilinear-problem})}.
\end{itemize}
Further explanations of these aspects will be found in the remainder of introduction.
\end{remark}

\begin{remark}\label{remark-noise}
More information of the random noise is in the following.
\begin{enumerate}[leftmargin=2em]
\item[$(1)$] The first identity in {\rm(\ref{structure-noise})} indicates that the law of $\eta(t,x)$ is $T$-statistically periodic in time, while
the second is in fact in accordance with 
\begin{equation*}
\eta_n(t,x)=\chi(x)\sum_{j\in\N^+}c_j\theta_j^n(t)e_j(x),\quad t\in [0,T),
\end{equation*}
where $c_j$ are nonnegative numbers, and $\{\theta_j^n;j\in\N^+\}$ stands for a sequence of independent bounded random processes that is not necessarily identically distributed. Moreover, the presence of $\chi(x)$ means that $\eta(t,x)$ possesses the localization feature similarly to $a(x)$.

\item[$(2)$] In view of  {\rm(\ref{structure-noise-2})}, our setting for $\eta(t,x)$ covers both of the cases where it is finite-/infinite-dimensional in time. The former means that  $\eta(t,x)$ is a smooth function of time variable, while the latter implies that it may be rough in time. Another consequence of {\rm(\ref{structure-noise-2})} is that the support of the law of $\eta_n$ is compact in $L^2(D_T)$ and bounded in $L^\infty(0,T;H^{\scriptscriptstyle4/7})$. 

\item[$(3)$] Different from the parabolic cases {\rm(}see, e.g., {\rm\cite{Shi-15,KNS-20,BGN-23,Ner-24})}, 
our result of exponential mixing can not be guaranteed for arbitrary time spread $T>0$. This is essentially because the spectral gap in the high frequency of hyperbolic equations is usually bounded.

\end{enumerate}
\end{remark}

 \subsubsection{Discussion of the result}\label{Section-122} 

The main content of this subsection is to illustrate the level of generality of Theorem \ref{Thm-informal-model}. To this end, we first provide some further comments on our settings for the damping coefficient, nonlinearity and random noise in (\ref{Problem-0}). Another thing involved is to demonstrate that our approach is adaptable to several other types of dispersive equations.

\vspace{0.6em}

\noindent  {\bf Localized damping, critical nonlinearity and multi-featured noise.}

\begin{itemize}[leftmargin=2em]
\item[$(1)$] Our attention on {\it localized damping} is motivated by its mathematical interest and practical significance. While the wave equation is a conservative system, many authors have introduced different types of dissipation mechanisms, especially, damping effect, to stabilize the oscillations. In particular,  the localized damping can be traced to the effort to find the minimal dissipation mechanism. This research field stays active in the recent decades; see, e.g., \cite{Bardos-92, Krieger-Xiang-2022, Coron-Krieger-Xiang-1, 
Lions-88,Zuazua-90,CLT-08,JL-13,DLZ-03, KT-95,LT-93} for some contributions along this line.  The related mathematical theories have also been invoked in physical applications such as thermoelasticity of  composed materials \cite{MMN-02}. See Figure \ref{Figure-Gamma} below for a rough picture of the effective domain of damping, involved in setting $(\mathbf{S1})$.

\begin{figure}[H]
\centering
\includegraphics[width=0.5\textwidth]{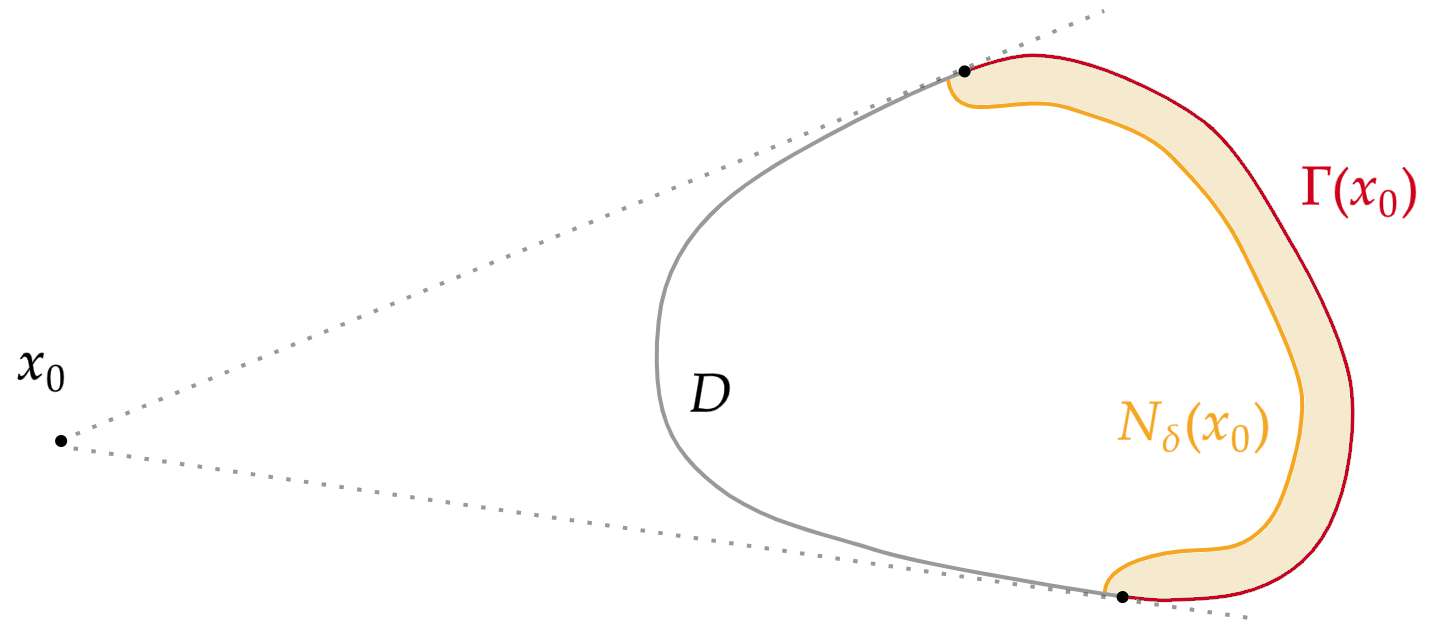}
\caption{$\Gamma$-type domain.}\label{Figure-Gamma}
\end{figure}

On the other hand,  {\it Theorem {\rm\ref{Thm-informal-model}}  is optimal in the sense that the case where the damping vanishes {\rm(}i.e. $a(x)\equiv 0${\rm)} is out of reach.} In fact, the mixing property means in general that the corresponding random dynamical system admits a statistical equilibrium having the global stability, which implies the dissipation of the system. Therefore, 
the damping effect induced by $a(x)$, assuring the dissipation mechanism of (\ref{Problem-0}), seems necessary for  mixing. As a circumstantial evidence, we refer the reader to \cite[Theorem 9.2.3]{DPZ-96} { for a negative result}, concerning a linear wave equation with constant damping and white noise.

\medskip

\item[$(2)$] Considering the {\it subcritical nonlinearity} for wave equations is a previously used approach for addressing the technical issues caused by the lack of the smoothing effect. Under this subcritical setting, the nonlinear term takes values being more regular than the phase space, and such regularity can be  useful in the arguments of ergodicity and mixing; see, e.g., \cite{Martirosyan-14,FT-24,Tolomeo-20,Martirosyan-17}.

In comparison, the availability of  {\it critical nonlinearity} in the present paper is mainly thanks to the general framework described in Section \ref{Section-1-1}, which enables us to employ the asymptotic compactness to reduce the exploration of mixing to the problem restricted on an invariant compactum.

\medskip

\item[$(3)$] 
{As described in Remark \ref{remark-noise}, the random noise $\eta(t,x)$ is {\it localized in physical space and finite-dimensional both in space and time.}} 
Our interest on such type of random noise is inspired directly from the works of \cite{Shi-15,Shi-21} by Shirikyan. 
Another feature of $\eta(t,x)$ is the {\it boundedness in random parameter}, while the statistics associated are essentially different from the white noise. This enables us to invoke the viewpoints coming from  deterministic problems,  compensating for the unavailability of stochastic tools based on It\^{o} calculus; see Section \ref{Section-1-3} for further discussions. We also mention that the bounded noise serves better to build models for some specific physical problems (for instance, in modern meteorology); see, e.g., the monograph \cite{Onofrio-13}.

\end{itemize}

\vspace{0.6em}

\noindent  {\bf Potential future extensions of the approach.}
\medskip

In order to prove Theorem \ref{Thm-informal-model}, it suffices to verify the abstract hypothesis $(\mathbf{H})$, including the asymptotic compactness, irreducibility and coupling condition, so that Theorem \ref{Thm-informal} is applicable to (\ref{Problem-0}). Our approach for this purpose is to invoke, extend and combine the ideas originated in various fields of dynamical system, dispersive equation and control theory:  

\begin{itemize}[leftmargin=2em]
\item[$(1)$] The proof of asymptotic compactness is translated to a similar issue for the non-autonomous dynamical system generated by  (\ref{semilinear-problem}), i.e., whether there exists an $\mathcal H$-compact set attracting exponentially any trajectory of the system. 

\medskip
\item[$(2)$] In the context of PDEs, the irreducibility is typically attributed to a given state that can be reached by the dynamics regardless of initial conditions. Our approach we adopt to verify the irreducibility is based on the global stability\footnote{By ``global'' we mean that the scale of states can be
as large as we want.} of equilibrium for
the unforced problem (i.e. $f(t,x)\equiv 0$ in the context of (\ref{semilinear-problem})), which is in fact one of central issues regarding the dynamics of wave  equations and even other types of dispersive equations.  
\medskip
\item[$(3)$] The verification of coupling hypothesis will be accomplished by analyzing a controlled system associated with (\ref{semilinear-problem}). Our arguments in this part are adaptations and combinations of the underlying ideas in controllability, observability and stabilization, which constitute a major part of control theory. 
\end{itemize}
See Section \ref{Section-1-3} later for relevant discussions of contexts within the prior and present works. 

\vspace{0.6em}

While we focus on model (\ref{Problem-0}) in this paper, we believe that the approach is rather general and it can be adapted with
technical modifications to yield the mixing property for other types of dispersive equations. This is mainly because, as previously mentioned, we translate the issue of mixing property into several specific topics. Meanwhile, there are several results relevant to these topics and available for other dispersive equations, which one may extend further to meet the setting in our framework. The reader is referred to, e.g., \cite{DGL-06,Laurent-10,BBZ-13,Bourgain-14,EKZ-17,ALM-16} for the nonlinear Schr\"{o}dinger equations and \cite{CC-2004, CKN-2024, CRX-2017, LRZ-10,ET-16,EKZ-18} for KdV equations.

\subsection{Overview of the proof }\label{Section-1-3}

\begin{figure}[H]
\centering
\includegraphics[width=0.88\textwidth]{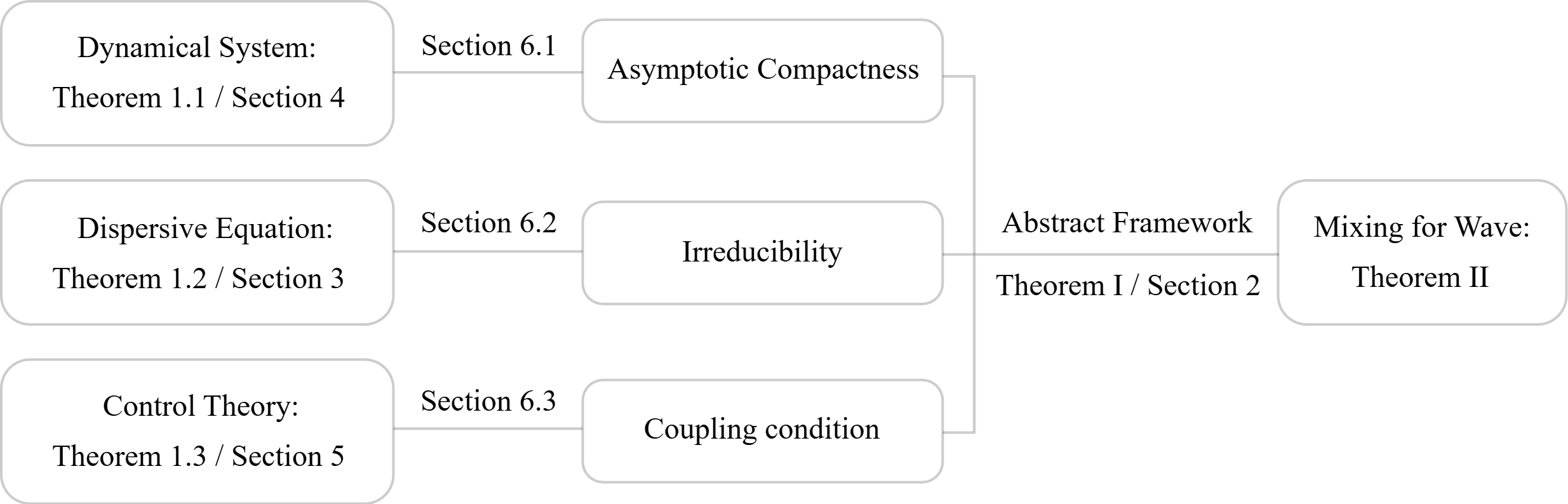}
\caption{Structure of the proof.}\label{Figure-structure}
\end{figure}

As stated in Section \ref{Section-122}, the proof of Theorem \ref{Thm-informal-model} is based on several intermediate results for the deterministic equation (\ref{semilinear-problem}). In what follows, we shall provide brief statements of these results, i.e. Theorems \ref{Thm-informal-AC}-\ref{Thm-informal-SQ} below, and describe their relations to the randomly forced equation (\ref{Problem-0}). See Figure~\ref{Figure-structure} for a rough picture of the proof.

\subsubsection{Asymptotic compactness}

In order to verify hypothesis $(\mathbf{H})$, the initial step is to construct a compact subset of $\mathcal H$, which is exponentially attracting for (\ref{semilinear-problem}). In the construction, one thing to be careful is that the regularity of attracting set should be high enough to carry the  irreducibility and coupling construction. 
Accordingly, we shall prove the existence of an $\mathcal H^{\scriptscriptstyle4/7}$-bounded attracting set for (\ref{semilinear-problem}).
In the language of dynamical system, such property can be described as 
\begin{equation*}
\textrm{ ``$(\mathcal H,\mathcal H^{\scriptscriptstyle4/7})$-asymptotic compactness''. }
\end{equation*}

 The proof of this result constitutes the main content of Section \ref{Section-nonautonomous}.

\begin{theorem}[{\bf Asymptotic compactness}]\label{Thm-informal-AC}
{\it Assume that $a(x)$ satisfies {\rm(\ref{Damping-localization})}. Then for every $R_0>0$, 
there exists a bounded subset $\mathscr B_{\scriptscriptstyle4/7}$ of $\mathcal H^{\scriptscriptstyle4/7}$ and constants $C,\kappa>0$ such that if
$$
\text{the force }f\ \text{belongs to } \overline{B}_{L^\infty(\R^+;H^{\scriptscriptstyle4/7})}(R_0),
$$
then the solution $u$ of {\rm(\ref{semilinear-problem})} satisfies that
\begin{equation*}
{\rm dist}_{\mathcal H}(u[t],\mathscr B_{\scriptscriptstyle4/7})\leq C\left(1+E_u(0)\right)e^{-\kappa t},\quad \forall\, t\geq 0, 
\end{equation*}
where 
${\rm dist}_{\mathcal H}(\cdot,\cdot)$ denotes the Hausdorff pseudo-distance in $\mathcal H$ {\rm(}see {\rm(\ref{pseudo-distence})} later{\rm)}.}
\end{theorem}

{A more general version of Theorem \ref{Thm-informal-AC}, as well as the asymptotic compactness in a ``physical'' space $\mathcal H^1$, is contained in Theorem \ref{Thm-dynamicalsystem}. 
By taking $R_0$ sufficiently large so that $\eta\in\overline{B}_{L^\infty(\R^+;H^{\scriptscriptstyle4/7})}(R_0)$ almost surely (see Remark \ref{remark-noise}), one can check that the attraction of $\mathscr B_{\scriptscriptstyle4/7}$ also works on the solution paths of (\ref{Problem-0}). Hence, the hypothesis of asymptotic compactness in $(\mathbf{H})$ is verified with $\mathcal Y=\overline{\mathscr B_{\scriptscriptstyle4/7}}$; see Section \ref{Section-6-1} for more details.}

\vspace{0.6em}

When $a(x)\equiv a_0>0$, the conclusion of Theorem \ref{Thm-informal-AC} is rather standard; see, e.g.,  \cite{Zelik-04}. On the other hand,
the case of localized damping is much more subtle, which is up to now understood only in the autonomous setting, i.e. $$f(t,x)\equiv f(x).$$ To address the localized damping, one of the approaches is provided in \cite{FZ-93} and consists mainly of the following properties:
\begin{itemize}[leftmargin=1em]
\item The unique continuation for a 
homogeneous equation in the form
\begin{equation}\label{linear-toymodel}
\boxempty v+p(t,x)v=0,\quad t\in[0,T],
\end{equation}
obtained by linearizing the  equation considered there and removing the damping term\footnote{The unique continuation says that if a solution of (\ref{linear-toymodel}) vanishes on the effective domain of damping, then it vanishes on the entire domain; see, e.g., \cite{RZ-98}.}.

\item The monotonicity of the energy, which can be readily derived in the autonomous setting.
\end{itemize}
The combination of them enables one to deduce the global dissipativity (i.e. the existence of an absorbing set) for the equation. As a consequence, the asymptotic compactness  (and hence the existence of global attractor) follows in a fairly standard way.

Another approach is to invoke the unique continuation just mentioned for deriving the gradient structure \cite{Hale-88} for the corresponding dynamical system. This implies the asymptotic compactness without any explicit discussion of dissipativity. See, e.g., \cite{CLT-08,JL-13} for the related literature. 

\vspace{0.6em}

Does the asymptotic compactness hold for (\ref{semilinear-problem}) with a nonzero force $f(t,x)$ depending on $t$? This problem {\it remains open} mainly due to the following {\bf difficulties}:
\begin{itemize}[leftmargin=2em]
\item[$(1)$] The damping coefficient $a(x)$ can be localized in the physical space (see Remark \ref{remark-damping}).

\item[$(2)$] In the presence of $f(t,x)$ the linearized problem of (\ref{semilinear-problem}) is inhomogeneous, and  the unique continuation does not make sense in such situation. As an aftermath, the discussion of gradient structure becomes much more complicated.

\item[$(3)$] The energy function for (\ref{semilinear-problem}) is not necessarily non-increasing in time, which  can be seen from the flux estimate
\begin{equation}\label{flux-0}
E_u(T)-E_u(0)=\int_{D_T} \left[-a(x)|\partial_tu|^2+f\partial_t u\right],\quad \forall\, T\geq 0. 
\end{equation}
\end{itemize}
The main task of Section \ref{Section-nonautonomous} later is to  give an affirmative answer to this question, and then the conclusion as in Theorem \ref{Thm-informal-AC} is obtained.

\vspace{0.6em}

The ideas and methods proposed for overcoming these obstacles contribute to part of novelty of the present paper.
Roughly speaking, we observe that when the energy of a solution is large, it is non-increasing in discrete times (see Lemma \ref{Lemma-nonincreasing}): there exist  constants $T_0,A_0>0$ such that
\begin{equation}\label{Implication-mono}
E_u(0)\geq A_0\quad \Rightarrow\quad E_u(T_0)\leq E_u(0).
\end{equation}
In comparison, it is non-increasing in continuous time when $f(t,x)\equiv 0$. Property (\ref{Implication-mono}) will be obtained by establishing 
\begin{equation*}
\begin{aligned}
\int_0^{T} E_u(t)dt &\lesssim
E_u(T)+\int_{D_T} \left[a(x)|\partial_tu|^2+u^2+|f \partial_t u|+|f|^2\right],
\end{aligned}
\end{equation*}
by means of the multiplier technique, where the related constant is uniform for $T,u,f$. The preceding estimate extracts more information from the flux (in comparison with (\ref{flux-0})), 
illustrating roughly the propagation of localized dissipation to the whole system.

In the sequel, it will be demonstrated that such type of ``discrete monotonicity'' is sufficient for the global dissipativity of (\ref{semilinear-problem}). Based on the dissipativity, we arrive at the $(\mathcal H,\mathcal H^{\scriptscriptstyle4/7})$-asymptotic compactness (in the absence of  gradient structure), as desired, by using some estimations on the basis of  Strichartz estimates (see \cite{BSS-09,BLP-08} and also Proposition \ref{prop-Strichartz} later).

\subsubsection{Irreducibility}

As mentioned in Section \ref{Section-1-2}, we verify the irreducibility hypothesis in $(\mathbf{H})$, by invoking the global stability of an equilibrium for the unforced problem, i.e. (\ref{semilinear-problem}) with $f(t,x)\equiv 0$. To this end, we shall use the following result due to Zuazua \cite{Zuazua-90}.

\begin{theorem}[{\bf Exponential decay;  \cite{Zuazua-90}}]\label{Thm-informal-GS}
{\it Assume that $a(x)$ satisfies {\rm(\ref{Damping-localization})}. Then there exist constants $C,\gamma>0$ such that
\begin{equation}\label{informal-stability}
E_u(t)\leq Ce^{-\gamma t}E_u(0),\quad\forall\, t\geq 0
\end{equation}
for any global solution $u$  of {\rm (\ref{semilinear-problem})} with $f(t,x)\equiv 0$.}
\end{theorem}

{See Proposition \ref{prop-global-stability} for a 
direct consequence of Theorem \ref{Thm-informal-GS}, describing the global stability of zero equilibrium.}
This, combined with  setting $(\mathbf{S2})$, could give rise to the irreducibility for (\ref{Problem-0}); see Section \ref{Section-6-2} for more details. 
Let us mention that the approach of type ``irreducibility via global stability'' has been widely used both in the cases of white noise \cite{HM-06,FGRT-15} and bounded noise \cite{KNS-20,Shi-15}.

\vspace{0.6em}

The stability of the damped wave equations is an active research topic in the recent decades; see, e.g., \cite{JL-13,DLZ-03, Krieger-Xiang-2022, Coron-Krieger-Xiang-1, KT-95, KK-23,LLTT-17}.
In the context of $\Gamma$-type geometric condition (involved in setting $(\mathbf{S1})$), the global stability of type (\ref{informal-stability}) has been fully studied for wave equations with defocusing nonlinearities, which is based on the multiplier technique developed in \cite{Lions-88}.
Another approach to the global stability is within the framework of the microlocal analysis (see, e.g., \cite{BG-97}), where the so-called geometric control condition (GCC for abbreviation) is introduced  \cite{Bardos-92}, and which gives almost sharp stability results.

In particular, we mention here that the GCC-based result in \cite{JL-13} is also sufficient for verifying the irreducibility hypothesis, although it is of local type, i.e., the  constants $C,\gamma$ in (\ref{informal-stability}) depends on the size of initial data. This is mainly because the irreducibility  involved in our criterion is required to work only on a compact set. Therefore, there seems to be some hope in extending the result of Theorem \ref{Thm-informal-model} to the setting of GCC; the key step would be to establish the asymptotic compactness as in Theorem \ref{Thm-informal-AC} for such case.

\subsubsection{Coupling condition}

Inspired by the idea of ``controllability implies mixing'' developed in \cite{Shi-15,Shi-21},
the verification of coupling hypothesis will be based on a squeezing property for the associated controlled system: 
\begin{equation}\label{control-model}
\begin{cases}
    \boxempty u+a(x)\partial_t u+u^3= h(t,x)+\chi\mathscr P^{\scriptscriptstyle T}_{\scriptscriptstyle N}\zeta(t,x), \quad x\in D,\\
    u[0]=(u_0,u_1)=u^0.
    \end{cases}
\end{equation}
Here, $h(t,x)$ is a given external force,  $\zeta(t,x)$ stands for the control, and  $\mathscr P^{\scriptscriptstyle T}_{\scriptscriptstyle N}$ denotes the projection from $L^2(D_T)$ to the finite-dimensional space
\begin{equation*}
\text{span} \{ e_j\alpha^{\scriptscriptstyle T}_k,1\leq j,k\leq N\}.
\end{equation*}
We refer the reader to the monograph \cite{Coron-07} by Coron for comprehensive descriptions of the italic terminology below from the control theory.  Our analysis for the control problem is placed in Section \ref{Section-control}.

The squeezing property for (\ref{control-model}) is collected in the following.

\begin{theorem}[{\bf Squeezing property}]\label{Thm-informal-SQ}
Assume that $a(x),\,\chi(x)$ satisfy setting  $(\mathbf{S1})$. Then for every $T>\mathbf{T}$ and $R_1,R_2>0$, there exist constants $N\in\N^+$ and $d>0$  such that for every  $u^0,\hat u^0\in \overline{B}_{\mathcal H^{\scriptscriptstyle4/7}}(R_1)$ with 
\begin{equation*}
\|u^0-\hat u^0\|_{_{\mathcal H}}\leq d
\end{equation*}
and $h\in\overline{B}_{L^2(0,T;H^{\scriptscriptstyle4/7})}(R_2)$,
 there is a control $\zeta\in L^2(D_T)$ satisfying
\begin{equation}\label{Informal-squeezing}
\|S(\hat u^0,h)-S(u^0, h+\chi\mathscr P^{\scriptscriptstyle T}_{\scriptscriptstyle N}\zeta)\|_{_{\mathcal H}}\leq \frac{1}{4}\|\hat u^0-u^0\|_{_{\mathcal H}},
\end{equation}
where $S$ is defined by {\rm(\ref{definition-maps})}.
\end{theorem} 

{See Theorem \ref{Th1} for a stronger statement of Theorem \ref{Thm-informal-SQ}, where more information on the structure of control, also necessary in dealing with (\ref{Problem-0}), is involved. Denote by $\ell$ the common law of $\eta_n$ in $L^2(D_T)$, and by $\mathcal{E}$ its support. The parameters $R_1,R_2$ can be appropriately chosen so that 
$$
\mathcal Y_\infty\subset\overline{B}_{\mathcal H^{\scriptscriptstyle4/7}}(R_1),\quad\mathcal E\subset\overline{B}_{L^2(0,T;H^{\scriptscriptstyle4/7})}(R_2).
$$
Then, combined with two classical results for optimal couplings and
an estimate for the total variation distance (see Appendix \ref{Appendix-Probability}), the squeezing (\ref{Informal-squeezing}) would imply the coupling condition for (\ref{Problem-0}); see Section \ref{Section-6-3} for more details.}

\vspace{0.6em}

Control problems, including controllability and stabilization\footnote{
In control theory, the controllability means that for any given two states in the phase space, there is a control force driving the system from one state to the other in a finite time. On the other hand, the stabilization problem is whether or not a controlled system can be asymptotically stabilized to a (non-)stationary solution. See \cite{Coron-07} for more information.}, for nonlinear wave equations (and other dispersive equations) with  localized control have attracted much attention in the last few decades; see, e.g., \cite{DL-09,BBZ-13,Bourgain-14,DGL-06,Coron-Krieger-Xiang-1, LRZ-10,BRS-11,ADS-16}. In particular, the literature with low-frequency controls in general concentrates on the stabilization problem, as the controllability properties are usually valid just for the low frequency in the evolution. Such subtlety can be partly explained by a viewpoint of Dehman and Lebeau \cite{DL-09}
that {\it ``the energy of each scale of the control force depends {\rm(}almost{\rm)} only on the energy of the same scale in the states that one wants to control''}.

Since the squeezing property considered here is closely related to the stabilization (see Remark~\ref{remark-stabilization}), the strategy of our proof for Theorem \ref{Thm-informal-SQ} is inspired by the ideas coming from the theories of stabilization, in particular, the prior works \cite{Xiang-23,Xiang-24,ADS-16,KX-23}, with  technical modifications adapted to (\ref{control-model}). The methodology we introduce for proving Theorem \ref{Thm-informal-SQ} is {\it``frequency analysis''}, i.e., 

 \vspace{0.6em}

\begin{center}    
\tikzset{every picture/.style={line width=0.75pt}} 
\begin{tikzpicture}[x=0.75pt,y=0.75pt,yscale=-1,xscale=1]
\draw [color={rgb, 255:red, 0; green, 0; blue, 0 }  ,draw opacity=0.2 ]   (284.1,12.87) -- (309.29,25.5) ;
\draw [color={rgb, 255:red, 0; green, 0; blue, 0 }  ,draw opacity=0.2 ]   (309.29,25.5) -- (284.1,37.51) ; 
\draw [color={rgb, 255:red, 0; green, 0; blue, 0 }  ,draw opacity=0.2 ]   (284.43,37.9) -- (309.62,49.67) ;
\draw [color={rgb, 255:red, 0; green, 0; blue, 0 }  ,draw opacity=0.2 ]   (309.62,49.67) -- (284.43,60.87) ;
\draw [color={rgb, 255:red, 0; green, 0; blue, 0 }  ,draw opacity=0.2 ]   (506,25.85) -- (531.19,39.36) ;
\draw [color={rgb, 255:red, 0; green, 0; blue, 0 }  ,draw opacity=0.2 ]   (531.19,39.36) -- (506,52.2) ;
\draw (155.67,3) node [anchor=north west][inner sep=0.75pt]  [font=\normalsize] [align=left] {{ duality argument}};
\draw (125,26.6) node [anchor=north west][inner sep=0.75pt]  [font=\normalsize] [align=left] {{ observability inequality }};
\draw (162.67,49.8) node [anchor=north west][inner sep=0.75pt]  [font=\normalsize] [align=left] {{ damping effect }};
\draw (311.7,15.27) node [anchor=north west][inner sep=0.75pt]  [font=\normalsize] [align=left] {{ low-frequency controllability }};
\draw (318.21,41.29) node [anchor=north west][inner sep=0.75pt]  [font=\normalsize] [align=left] {{ high-frequency dissipation }};
\draw (532,30) node [anchor=north west][inner sep=0.75pt]  [font=\normalsize] [align=left] {{ squeezing }};
\end{tikzpicture}
\end{center}

\vspace{0.6em}

Below we give a discussion of the main novelty of our approach, and refer to Section \ref{Section-controlresult} later for a technical outline of proof for Theorem \ref{Thm-informal-SQ}.
\begin{itemize}[leftmargin=2em]
\item[$(1)$] We establish a new version of {\it duality  between controllability and observability} in the context of (\ref{control-model}), i.e. Proposition \ref{theorem-duality}, which not only translates the low-frequency controllability problem to the verification of observability inequality 
$$
\int_0^T\|\mathscr P^{\scriptscriptstyle T}_{\scriptscriptstyle N}(\chi\varphi)(t)\|^2_{_{H^{\scriptscriptstyle-s}}}\gtrsim\|\varphi[T]\|^2_{_{\mathcal H^{\scriptscriptstyle-1-s}}}\quad \text{with some }s\in(0,1)
$$
for solutions $\varphi$ of the adjoint system,
but also helps us to improve the regularity of control. The latter plays an important role in deriving the strong dissipation for the high-frequency system. As a by-product, the quantitative controllability can be obtained within our framework and the control is expressed in an explicit form. 

\item[$(2)$] The presence of space-dependent coefficient $a(x)$  leads to various technical complications (see Remark \ref{remark-damping}), so that the arguments used for observability inequality in the prior works, e.g., \cite{ADS-16,Zhang-PRSL,EZ-10,DL-09}, may not be easily applicable in the context involved here. Part of our analyses aim at dealing with such issue, involving  {\it unique continuation}, {\it Carleman estimates} and 
{\it Hilbert uniqueness method} (HUM for abbreviation). As a consequence, the proof of observability constitutes a delicate part of our control analysis. 

\end{itemize}

\vspace{0.6em}

\subsection{Organization of the present paper}\label{Section-1-4} 

In Section~\ref{Section-Probalistic part}, we present a rigorous statement  of our criterion (i.e., Theorem~\ref{Thm-informal}) and its proof. In the sequel, the intermediate results mentioned in Section \ref{Section-1-3} are precisely provided in Sections \ref{Section-PDE}-\ref{Section-control}. 

We in Section \ref{Section-PDE} give a complete statement of global stability result for the unforced version of (\ref{semilinear-problem}), i.e., Theorem \ref{Thm-informal-GS}, as well as some energy and dispersive estimates that will be useful in later sections.
The main content of Section~\ref{Section-nonautonomous} is to prove a stronger version of  Theorem \ref{Thm-informal-AC}, the asymptotic compactness for (\ref{semilinear-problem}). The result therein is obtained by improving the classical arguments in global attractor for dynamical systems and by introducing the notion of discrete monotonicity. We next turn attention to the full statement and proof of squeezing property, i.e., Theorem \ref{Thm-informal-SQ}, in Section \ref{Section-control}. In this part, the ideas and methods in control theory will come into play. 

Finally, putting the above results all together, we conclude in Section~\ref{Section-conclusion} with a rigorous version of Theorem~\ref{Thm-informal-model}, illustrating how our criterion of exponential mixing is applied to the random wave equation (\ref{Problem-0}).

Appendixes \ref{Appendix-Probability} and \ref{Appendix-control} collect some auxiliary results and proofs that are needed in our probabilistic and control analyses of the main text, respectively. In addition, an index of symbols is contained in Appendix \ref{Appendix-index}. 

\medskip

\noindent{\it Note. From now on, the letter $C$ denotes the generic constant which may change from line to line.}

\section{Mixing criterion for random dynamical systems}\label{Section-Probalistic part}

The primary objective of this section is to establish our asymptotic-compactness-based  criterion, i.e. Theorem~\ref{Thm A} below,  as briefly stated in Theorem~\ref{Thm-informal}. It serves as a fundamental instrument to demonstrate exponential mixing for model~\eqref{Problem-0} in Section~\ref{Section-conclusion}.  

\vspace{0.6em}

We begin with some necessary notations and conventions. Let $\mathcal{X}$ and $\mathcal{Z}$ be Polish spaces, and the metric on $\mathcal X$ is denoted by $d$. We write $B_{\mathcal X}(x,r)=\{y\in \mathcal{X};d(x,y)<r\}$ for $x\in \mathcal{X}$ and $r>0$, and $B_{\mathcal X}(r)=B_{\mathcal X}(0,r)$ when $\mathcal{X}$ is a separable Banach space. Let us denote $\overline{B}_{\mathcal X}(r)=\overline{B_{\mathcal X}(r)}$. Define 
    \begin{equation}\label{pseudo-distence}
    \text{dist}_{\mathcal X}(x,A) =\inf_{y\in A}d(x,y),\quad x\in \mathcal X,\,A\subset \mathcal{X}.
    \end{equation}
    If there is no danger of confusion, we shall omit the subscript $\mathcal X$ of the above notations for the sake of simplicity. In addition, let us lay out some collections related to $\mathcal X$: $\mathcal{B}(\mathcal{X})$ denotes its Borel $\sigma$-algebra; $\mathcal{P}(\mathcal{X})$ is the set of Borel probability measures on $\mathcal{X}$; by $B_b(\mathcal{X})$, $C_b(\mathcal{X})$ we denote the set of bounded Borel/continuous functions on $\mathcal{X}$, endowed with the supremum norm $\|\cdot\|_{\infty}$, respectively; $L_b(\mathcal{X})$ stands for the set of bounded Lipschitz functions. For $f\in L_b(\mathcal{X})$, we denote its Lipschitz norm by $$\|f\|_{L}:=\|f\|_{\infty}+ \sup_{x\neq y}\frac{|f(x)-f(y)|}{d(x,y)}.$$ 
    For $ f\in B_b(\mathcal{X})$ and $\mu\in\mathcal{P}(\mathcal{X})$,  we write $\langle f,\mu\rangle=\int_{\mathcal{X}}f(x)\mu(dx)$.
The dual-Lipschitz distance in $\mathcal{P}(\mathcal{X})$ is defined as
    $$
    \|\mu-\nu\|_{L}^*=\sup_{f\in L_b(\mathcal X),\|f\|_L\leq 1}|\langle f,\mu \rangle-\langle f,\nu \rangle|,\quad \mu,\nu\in \mathcal{P}(\mathcal{X}),
    $$
    which metricizes the weak topology; see, e.g., \cite[Section 1.2.3]{KS-12}. 
    
    Recall that for $\mu_1,\mu_2\in\mathcal{P}(\mathcal{X})$, a pair of $\mathcal{X}$-valued random variables $(\xi_1,\xi_2)$ is called a \textit{coupling} for $\mu_1$ and $\mu_2$, if $\mathscr{D}(\xi_i)=\mu_i$, $i=1,2$. We denote by $\mathscr{C}(\mu_1,\mu_2)$ the set of these couplings.

\vspace{0.6em}

The general settings of random dynamical systems and the main theorems are presented in Section~\ref{Section-generalresult}, followed by a brief outline of the proof. The detailed proof is collected in Section~\ref{Sect. 2.2}.

\subsection{Settings and general results}\label{Section-generalresult}

Let us recall that the considered Markov process $\{x_n;n\in\N\}$ is given by \eqref{RDS},\eqref{initial condition},  where $S\colon\mathcal{X} \times\mathcal{Z} \rightarrow \mathcal{X}$ is a locally Lipschitz mapping, and $\{\xi_n;n\in\N\}$ is a  sequence of $\mathcal{Z}$-valued i.i.d.~random variables. The common law of $\xi_n$ is $\ell$, whose support is denoted by $\mathcal{E}$.  In order to indicate the initial condition and the random inputs, we also write 
\begin{equation}\label{rewrite-form}
x_n=S_{n}(x;\xi_0,\cdots,\xi_{n-1})=S_{n}(x;\boldsymbol{\xi}),\quad n\in\N^+
\end{equation}
with $\boldsymbol{\xi}:=\{\xi_n;n\in\N\}$.
Moreover, given a sequence $\boldsymbol{\zeta}=\{\zeta_n;n\in\N\}\in \mathcal{Z}^{\N}$, we denote by $$S_{n}(x ; \zeta_0, \cdots, \zeta_{n-1})=S_{n}(x;\boldsymbol{\zeta})$$ the corresponding deterministic process defined by (\ref{RDS}),(\ref{initial condition}) by replacing $\xi_n$ with $\zeta_n$.

    With the above setting, system (\ref{RDS}),(\ref{initial condition}) defines a Feller family of discrete-time Markov processes in $\mathcal{X}$; see, e.g., \cite[Section 1.3]{KS-12}. We denote by $\{\Pb_x;x\in\mathcal{X}\}$ the corresponding Markov family, by $\E_x$ the corresponding expected values, and by $\{P_n(x,A);x\in\mathcal{X},A\in\mathcal{B}(\mathcal{X}),n\in \N\}$ the corresponding Markov transition functions, i.e., $$P_n(x,A)=\Pb_x(x_n\in A).$$ We use the standard notation for the corresponding Markov semigroup $P_n\colon B_b(\mathcal{X})\rightarrow B_b(\mathcal{X})$ and its dual $P^*_n\colon\mathcal{P}(\mathcal{X})\rightarrow \mathcal{P}(\mathcal{X})$ defined by
	\begin{equation*}
		P_nf(x)= \int_{\mathcal{X}}f(y)P_n(x,dy),\quad\quad P_n^*\mu(A)=\int_{\mathcal{X}}P_n(x,A)\mu(dx)
	\end{equation*}
    for $f\in B_b(\mathcal{X})$, $\mu\in\mathcal{P}(\mathcal{X})$, $x\in\mathcal{X}$ and $A\in\mathcal{B}(\mathcal{X})$. Recall that a probability measure $\mu\in \mathcal{P}(\mathcal{X})$ is called \textit{invariant} for $\{P_n^*;n\in\N\}$  if $P_n^*\mu=\mu$ for any $n\in\N$.  Our goal is to investigate  exponential mixing for the Markov process $\{x_n;n\in\N\}$.

The following notion of attainable set will be used.
	
\begin{definition}\label{Def-attainable}
For every subset $\mathcal{Y}$ of $\mathcal{X}$, the attainable set $\mathcal{Y}_n$ in time $n$ is of the form
\begin{equation*}
\mathcal{Y}_0=\mathcal{Y},\quad \mathcal{Y}_n=\{S_n(x,\zeta_0,\cdots,\zeta_{n-1});x\in\mathcal{Y},\zeta_0,\cdots,\zeta_{n-1}\in \mathcal{E}\},\quad n\in\N^+,
\end{equation*}
and the attainable set $\mathcal{Y}_\infty$ is given by
$$
\mathcal{Y}_\infty=\overline{\bigcup_{n\in\N}\mathcal{Y}_n}.
$$  
\end{definition} 

\medskip

With the preparations above at hand, we list the hypotheses involved in our general criterion:  
	
\begin{itemize}[leftmargin=2.5em]
\item [($\mathbf{AC}$)]   ({\bf  Asymptotic compactness}) There exists a compact subset $\mathcal{Y}$ of $\mathcal{X}$, a constant $\kappa>0$, and a measurable function $V\colon\mathcal{X}\rightarrow\R^+$ which is bounded on bounded sets, such that
\begin{equation}\label{asymptotic compactness}
\text{dist}(S_n(x;\boldsymbol{\zeta}),\mathcal{Y})\leq V(x)e^{-\kappa n}
\end{equation}
for any $x\in\mathcal{X},\,\boldsymbol{\zeta}\in \mathcal{E}^{\N}$ and $n\in\N^+$.
\end{itemize}

Our observation on the asymptotic compactness has been described in Section \ref{Section-1-1}.
In particular, using the compactness of both $\mathcal{Y}$ and $\mathcal{E}$, straightforward compactness arguments imply that the attainable set $\mathcal{Y}_{\infty}$ is compact in $\mathcal{X}$; see Proposition \ref{Y_infty} later.
\begin{itemize}[leftmargin=2.5em]
\item [($\mathbf{I}$)]   ({\bf  Irreducibility on compact set}) There exists a point $z\in\mathcal{Y}$ such that for every $\varepsilon>0$, one can find an integer $m=m(\varepsilon)\in\N^+$  satisfying
\begin{equation}\label{irreducibility}
\inf_{x\in\mathcal{Y}_\infty}P_m(x,B(z,\varepsilon))>0.		
\end{equation}
		
\item [($\mathbf{C}$)]   ({\bf  Coupling condition on compact set}) There exists a constant $r\in[0,1)$  such that for every $x,x'\in\mathcal{Y}_\infty$, there is $(\mathcal{R}(x,x'),\mathcal{R}'(x,x'))\in\mathscr{C}(P_1(x,\cdot),P_1(x',\cdot))$ on a same probability space $(\Omega,\mathcal{F},\Pb)$, satisfying
\begin{equation}\label{squeezing}
\Pb(d(\mathcal{R}(x,x'),\mathcal{R}'(x,x'))>rd(x,x'))\leq g(d(x,x')),
\end{equation}
where $g\colon\R^+\rightarrow\R^+$ is a continuous increasing function with	\begin{equation}\label{g-def}
g(0)=0,\quad\limsup\limits_{n\rightarrow\infty}\frac{1}{n}\ln g(r^n)<0,
\end{equation}
and the mappings $\mathcal{R},\mathcal{R}'\colon\mathcal{Y}_\infty\times \mathcal{Y}_\infty\times \Omega\rightarrow \mathcal{X}$ are measurable.
\end{itemize} 

The last two hypotheses originate from the previous frameworks of ergodicity and mixing. More precisely,
the irreducibility indicates that a common state can be reached by the dynamics regardless of the initial conditions,  previously used to derive the unique ergodicity  {\rm(}see, e.g.,  {\rm\cite{HM-11b,EM-01,KS-12})}. 
On the other hand, 
the coupling condition can be understood as a one-step smoothing effect of the Markov process analogous to the asymptotic strong Feller property {\rm(}but only for regular solutions{\rm)}. It is directly motivated by the work of  {\rm\cite{Shi-15}}, and can be also traced to the earlier literature  {\rm\cite{KS-01,Hairer-02,Mattingly-02,MY-02}}.

As a more precise version of Theorem \ref{Thm-informal}, what follows is one of the main results of this paper, providing a criterion of exponential mixing. Its proof is contained in Section~\ref{Sect. 2.2}. 

\begin{theorem}\label{Thm A}
Assume that the support $\mathcal{E}$ of $\ell$ is compact in $\mathcal{Z}$, and hypotheses $(\mathbf{AC})$, $(\mathbf{I})$ and $(\mathbf{C})$ are satisfied. Then the Markov process $\{x_n;n\in\N\}$ has a unique invariant measure $\mu_*\in\mathcal{P}(\mathcal{X})$ with compact support. Moreover,  there exist constants $C,\beta>0$ such that
\begin{equation}\label{Exponential mixing}
\|P_n^*\nu-\mu_*\|_L^*\leq  Ce^{-\beta n}\left(1+\int_{\mathcal{X}} V(x)\nu(dx)\right)
\end{equation}  
for any $\nu\in\mathcal{P}(\mathcal{X})$ such that $\int_{\mathcal{X}} V(x)\nu(dx)<\infty$ and $n\in\N$.
\end{theorem}

\noindent{\bf Outline of proof for Theorem~\ref{Thm A}.}  We now present a brief overview of the proof for our result, highlighting its main contribution. Our strategy is to first establish mixing on the regular subspace $\mathcal{Y}_\infty$, and then extend to the entire space; see Figure~\ref{Proof of framework} for a rough picture\footnote{This picture is just for illustration, but not rigorous, since neither the attracting set $\mathcal{Y}$ nor the attainable set $\mathcal{Y}_\infty$ can be in a hyperplane in general.}. The proof is divided into three steps:
    \medskip

   \noindent {\bf  Step 1 (Existence of an invariant compactum).} We begin by demonstrating that the natural working space $\mathcal{Y}_\infty$ is compact and invariant due to hypothesis ({\bf AC}); see Proposition~\ref{Y_infty}. This allows a coupling construction to deduce the exponential mixing on $\mathcal{Y}_\infty$ in the next step.

   \medskip

   \noindent {\bf Step 2  (Mixing on the invariant compactum).}  In order to establish the mixing on $\mathcal{Y}_\infty$, we shall invoke  Kuksin--Shirikyan's framework (see \cite{KS-12,Shi-08}), under the hypotheses $(\mathbf{I})$ and $(\mathbf{C})$. More precisely, let us consider a Markov process $\{\vec{\boldsymbol{x}}_n;n\in\N\}$ on the product space $\mathcal{Y}_\infty\times\mathcal{Y}_\infty$ with marginals $P_n(x,\cdot)$ and $P_n(x',\cdot)$, where $x,x'\in \mathcal{Y}_\infty$. The process $\{\vec{\boldsymbol{x}}_n;n\in\N\}$ is called an extension of the process $\{x_n;n\in\N\}$, as detailed in Appendix~\ref{K-S framework}. Hypothesis ({\bf I}) guarantees a recurrence property: the two components of $\vec{\boldsymbol{x}}_n$ can be made to approach each other with arbitrary proximity within a finite time almost surely. Once the two components of $\vec{\boldsymbol{x}}_n$ have become sufficiently close, the coupling condition ({\bf C}) ensures that they will continue to converge with a positive probability; such convergence is referred to as squeezing. Consequently, the Markov property and this loop collectively indicate exponential mixing on $\mathcal{Y}_\infty$. For further details, please refer to Proposition~\ref{local exponential mixing}.

   \medskip

    \noindent {\bf Step 3 (Extending mixing to the original space).} The last step is to extend the $\mathcal{Y}_\infty$-restricted mixing to the entire state space. This is established via the exponential attraction of the invariant compactum $\mathcal{Y}_\infty$ (guaranteed by hypothesis ({\bf AC})) together with a projection procedure; see Proposition~\ref{prop-globalmixing}. 
    
 \begin{figure}[H]   \includegraphics[width=0.4\textwidth]{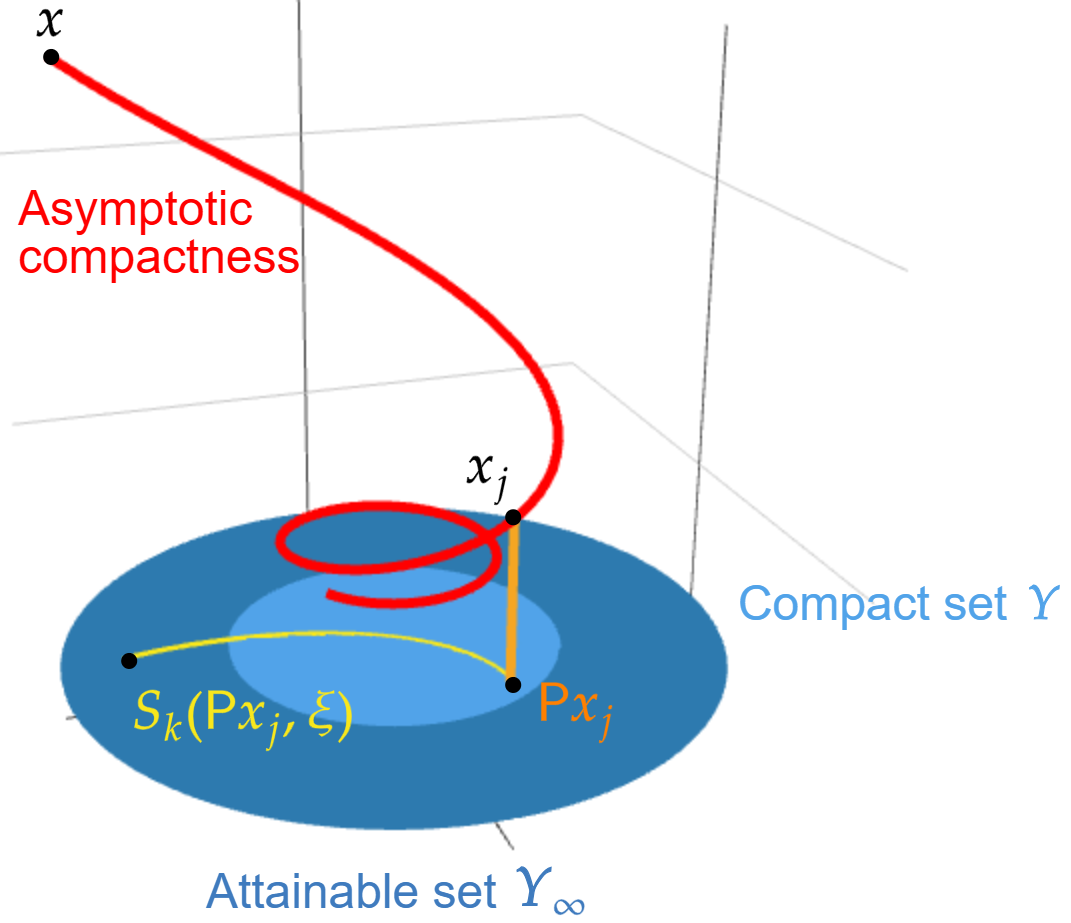}
    \caption{Outline of proof for Theorem~\ref{Thm A}.} \label{Proof of framework}
    \end{figure}

  As straightforward applications of Theorem~\ref{Thm A}, we have the following limit theorems, including the {\it strong law of large numbers} and {\it central limit theorem}  for bounded Lipschitz observables. The proofs are based on standard martingale decomposition procedures and are placed in Appendix~\ref{Proof of Thm limits}.

   \begin{proposition}\label{Limit theorems}
Under the assumptions of Theorem~{\rm\ref{Thm A}}, the following assertions hold:
    \begin{itemize}[leftmargin=2em]
    \item [$(1)$] {\rm(}Strong law of large numbers{\rm)} For every $f\in L_b(\mathcal{X})$ and $x\in\mathcal{X}$,
    \begin{equation*}
       \lim\limits_{n\rightarrow\infty}\frac{1}{n}\sum_{k=0}^{n-1}f(x_k)=\langle f,\mu_*\rangle\quad\text{almost surely}.
    \end{equation*}
    \item  [$(2)$] {\rm(}Central limit theorem{\rm)} For every $f\in L_b(\mathcal{X})$, there exists a constant $\sigma_f\geq 0$ such that
    \begin{equation*}
        \frac{1}{\sqrt{n}}\sum_{k=0}^{n-1}(f(x_k)-\langle f,\mu_*\rangle)\rightarrow \mathcal{N}(0,\sigma_f^2)\quad\text{ as }n\rightarrow\infty
    \end{equation*}
    for any $x\in\mathcal{X}$, where $\mathcal{N}(0,\sigma_f^2)$ denotes a normal random variable with zero mean and variance $\sigma_f^2$, and the convergence is in the sense of distribution.
    \end{itemize}
    \end{proposition}

\subsection{Proof of exponential mixing}\label{Sect. 2.2} As previously mentioned, the proof of Theorem~\ref{Thm A} consists of three steps.

\subsubsection{Existence of an invariant compactum.}  As mentioned in Step 1 of Section \ref{Section-generalresult}, a straightforward consequence of hypothesis $(\mathbf{AC})$ is that $\mathcal{Y}_\infty$ is a compact invariant set. Using (\ref{invariant set}) and the Feller property, a standard Kryloy--Bogolyubov averaging procedure yields that the Markov process $\{x_n;n\in\N\}$ admits an invariant measure. 
\begin{proposition}\label{Y_infty}
Assume that hypothesis $(\mathbf{AC})$ holds and $\mathcal{E}$ is compact in $\mathcal{Z}$. Then $\mathcal{Y}_{\infty}$ is compact in $\mathcal{X}$ and invariant under $S$ in the sense that
\begin{equation}\label{invariant set}
S(\mathcal{Y}_\infty\times\mathcal{E})\subset\mathcal{Y}_\infty.
\end{equation}
\end{proposition}

\begin{proof}[{\bf Proof}] We begin by demonstrating that the set $\mathcal{Y}_\infty$ is compact.  It can be observed that each set $\mathcal{Y}_n$ is compact, given that both $\mathcal{Y}$ and $\mathcal{E}$ are  compact. Let us now consider a sequence $\{y^n;n\in\N\}$  contained in $\bigcup_{l\in\N} \mathcal Y_l$. Then, there exists $l_n\in\N$ and $x^n\in \mathcal{Y}$ such that either $y^n=x^n$, or 
$$
y^n=S_{l_n}(x^n;\zeta_0^n,\cdots, \zeta_{l_n-1}^n)\in \mathcal{Y}_{l_n}
$$
for some $\zeta_j^n\in\mathcal {E},\ j=0,\cdots,l_n-1$. 

If the sequence $\{l_n;n\in\N\}$ is bounded, then taking $m=\max\{l_n;n\in\N\}$, it follows that $\{y^n;n\in\N\}$ is contained in $\bigcup_{0\leq l\leq m}\mathcal{Y}_l$, hence is relatively compact.   In the case where $\{l_n;n\in\N\}$ is unbounded, assume that $l_n\rightarrow\infty$ without loss of generality. By hypothesis ({\bf AC}), it follows that 
$$
{\rm dist}(y^n,\mathcal{Y})\leq V(x^n)e^{-\kappa l_n}
\rightarrow 0
$$
as $n\rightarrow \infty$. Here, we have tacitly used the boundedness of $\{x^n;n\in\N\}$.
Thus, by the compactness of $\mathcal{Y}$, we conclude that the sequence $\{y^n;n\in\N\}$ is relatively compact. Consequently, the compactness of $\mathcal{Y}_{\infty}$ follows immediately. 

It remains to prove that  $\mathcal{Y}_\infty$ is invariant. In view of its compactness, this is a direct consequence of the continuity of $S$. 
\end{proof}

\subsubsection{Mixing on the invariant compactum.} Based on Proposition~\ref{Y_infty}, we shall establish the exponential mixing for $\{x_n;n\in\N\}$ acting on the invariant compactum $\mathcal Y_\infty$. This is presented as the following result. 
	
\begin{proposition}\label{local exponential mixing}
Under the assumptions of Theorem~{\rm\ref{Thm A}}, the Markov process $\{x_n;n\in\N\}$ on $\mathcal{Y}_\infty$ admits a unique invariant measure $\mu_*\in\mathcal{P}(\mathcal Y_\infty)$. Moreover, there exist constants $C_0,\beta_0>0$ such that
\begin{equation}\label{exponential mixing on Y-infty}
    \|P_n^*\nu-\mu_*\|^*_{L}\leq C_0 e^{-\beta_0 n}
\end{equation}
for any $\nu\in\mathcal{P}(\mathcal{Y}_\infty)$ and $n\in\N$. 
\end{proposition}

    In order to demonstrate Proposition~\ref{local exponential mixing}, we employ a coupling construction. In particular, we utilize Theorem \ref{Theorem-KS}, which is a special case of the general result established by Kuksin and Shirikyan \cite{KS-12,Shi-08}. The proof of this proposition is  analogous to that presented in \cite[Theorem 1.1]{Shi-21}, where the approximate controllability and local stabilisability are replaced by irreducibility and coupling condition in the present setting. For the reader's convenience, we provide an outline of the proof below, while the details can be found in Appendix~\ref{Proof of Prop mixing}.

    \begin{proof}[{\bf Sketch of proof}] Following the route described in Step 2 of Section \ref{Section-generalresult}, we shall transform the problem into the verification of the recurrence and squeezing properties for an extension process, which will be appropriately constructed. The proof will be divided into three steps. 
    
    \medskip
    
    \noindent {\bf Step 2.1 (Extension construction).}  Let $ \boldsymbol{Y}_{\infty}=\mathcal{Y}_{\infty}\times\mathcal{Y}_{\infty}$ and constant $\delta\in(0,1)$ be specified later. We introduce the diagonal set in $\boldsymbol{Y}_{\infty}$ by 
    \begin{equation*}
    \boldsymbol{\mathcal{D}}_\delta=\{(x,x')\in\boldsymbol{Y}_{\infty};d(x,x')\leq \delta\}.
    \end{equation*}
    Then, define a coupling operator on $\boldsymbol{Y}_{\infty}$ by the relation
    \begin{equation*}
    \boldsymbol{{R}}(x,x')=\begin{cases}
    (\mathcal{R}(x,x'),\mathcal{R}'(x,x'))\quad&\text{for }(x,x')\in \boldsymbol{\mathcal{D}}_\delta,\\
    (S(x,\xi),S(x',\xi'))&\text{otherwise},
    \end{cases}
    \end{equation*}
    where $\xi$ and $\xi'$ are independent copies of $\xi_0$. Using this  coupling operator $\boldsymbol{{R}}$, we can construct a family of Markov processes $\{\vec{\boldsymbol{x}}_n;n\in\N\}$ on $ \boldsymbol{Y}_\infty$ with the following properties: 

    \begin{itemize}[leftmargin=2em]
        \item [(1)] $\{\vec{\boldsymbol{x}}_n;n\in\N\}$ is an extension of $\{x_n;n\in\N\}$. More precisely, the transition probability $\boldsymbol{P}_n(\vec{\boldsymbol{x}},\cdot)$ of $\vec{\boldsymbol{x}}_n$ is a coupling of $(P_n(x,\cdot),P_n(x',\cdot))$ for $\vec{\boldsymbol{x}}=(x,x')\in \boldsymbol{Y}_\infty$. In what follows, we make a slight abuse of notation and write $\vec{\boldsymbol{x}}_n=(x_n,x'_n)$.
        
        \item [(2)] We shall show that  the extension process $\{\vec{\boldsymbol{x}}_n;n\in\N\}$ verifies the squeezing and recurrence properties on $\boldsymbol{\mathcal{D}}_\delta$ for some $\delta\in(0,1)$ in the following sense: \\
$\centerdot$ (\textit{Squeezing}) There exist constants $C_1,\beta_1>0$ such that the random time 
        $$   \boldsymbol{\sigma}:=\inf\{n\in\N;d(x_n,x_n')> r^n \delta\}
        $$
        satisfies that
        \begin{equation} \label{sketch squeezing}
            \Pb(\boldsymbol{\sigma}=\infty)\geq 1/2,\quad \Pb(\boldsymbol{\sigma}=n)\leq C_1e^{-\beta_1 n} 
        \end{equation}
        for any $\vec{\boldsymbol{x}}\in\boldsymbol{\mathcal{D}}_{\delta}$ and $n\in\N$. Here, the constant $r\in[0,1)$ is established by (\ref{squeezing}).\\
$\centerdot$ (\textit{Recurrence}) There exist constants $C_2,\beta_2>0$ such that the random time 
        $$
        \boldsymbol{\tau}:=\inf\{n\in\N;\vec{\boldsymbol{x}}_n\in\boldsymbol{\mathcal{D}}_\delta\}
        $$
         satisfies that
        \begin{equation}\label{sketch recurrence}
            \Pb(\boldsymbol{\tau}<\infty)=1,\quad  \Pb(\boldsymbol{\tau}=n)\leq C_2e^{-\beta_2 n}
        \end{equation}
          for any $\vec{\boldsymbol{x}}\in\boldsymbol{Y}_{\infty}$ and $n\in\N$.
\end{itemize}
    
Once properties (1) and (2) are established,
we verify the conditions of Theorem~\ref{Theorem-KS}, thereby completing the proof of exponential mixing on $\mathcal{Y}_\infty$.

    \medskip
    
\noindent {\bf Step 2.2 (Verification of squeezing).}
In order to demonstrate the squeezing property, let us fix any $\vec{\boldsymbol{x}}=(x,x')\in\boldsymbol{\mathcal{D}}_\delta$. In view of the definition of $\boldsymbol{R}$ and the coupling hypothesis ({\bf C}), it follows that
\begin{equation}
\Pb(d(x_1,x'_1)\leq rd(x,x'))\leq 1-g(d(x,x')).\label{sketch squeezing 1}
\end{equation}      
This in conjunction with the Markov property allows for the application of standard iteration arguments, which in turn yield the following result: 
\begin{align*}
\displaystyle\Pb(\boldsymbol{\sigma}=\infty)\geq\prod_{n\in\N}(1-g(r^nd(x,x'))).
\end{align*}
Consequently, the first inequality in \eqref{sketch squeezing} is attained by choosing the parameter $\delta\in(0,1)$ sufficiently small and recalling that $g$ satisfies condition \eqref{g-def}.   
        Similarly,  one can further deduce that
        \begin{equation*}
            \Pb(\boldsymbol{\sigma}=n)\leq g(r^n),
        \end{equation*}
       which in turn implies the second inequality in \eqref{sketch squeezing} by taking $\beta_1\in (0,-\limsup\limits_{n\rightarrow\infty}\tfrac{1}{n}\ln g(r^n))$. 
       
       In summary, the squeezing property follows.

    \medskip

   \noindent {\bf Step 2.3 (Verification of recurrence).} 
        It remains to establish the recurrence~\eqref{sketch recurrence}.  Invoking the Markov property and  Borel--Cantelli lemma, it suffices to show that there exists $m\in\N^+$ and $p>0$ such that for every $\vec{\boldsymbol{x}}\in \boldsymbol{Y}_{\infty}$,
        \begin{equation*}
        \Pb(\vec{\boldsymbol{x}}_m\in\boldsymbol{\mathcal{D}}_\delta)\geq p.
        \end{equation*}        
        This can be  achieved through the following two observations.
        \begin{itemize}[leftmargin=2em]
        \item[$\centerdot$] The construction of the extension process allows one to verify that $x_n$ and $x_n'$ are conditionally independent on the set $\{\boldsymbol{\tau}\geq n\}$. In particular, taking hypothesis ({\bf I}) into account,  there exists $m\in\N^+$ such that          
        \begin{equation*}
        \Pb(\vec{\boldsymbol{x}}_m\in\boldsymbol{\mathcal{D}}_\delta|\boldsymbol{\tau}\geq m)\geq (\inf_{x\in\mathcal{Y}_\infty}P_m(x,B(z,\delta/2)))^2>0
        \end{equation*}

         \item[$\centerdot$] On the other hand, let us note that $\vec{\boldsymbol{x}}_{\tau}\in\boldsymbol{\mathcal{D}}_\delta$. Then making use of the strong Markov property and squeezing property, we get
        \begin{equation*}
        \Pb(\vec{\boldsymbol{x}}_m\in\boldsymbol{\mathcal{D}}_\delta|\boldsymbol{\tau}<m)\geq \inf_{\vec{\boldsymbol{x}}\in\boldsymbol{\mathcal{D}}_\delta}\Pb(\boldsymbol{\sigma}=\infty)\geq 1/2.
        \end{equation*}
        \end{itemize}
In combination, these above shall imply the recurrence. The proof of Proposition~\ref{local exponential mixing} is therefore completed.
\end{proof}

\begin{remark}
As a corollary of Proposition~{\rm\ref{local exponential mixing}}, it follows that $\text{supp }\mu_*\subset\mathcal Y_\infty$, which justifies the assertion that $\mu_*$ has compact support.  Indeed, by invoking hypothesis $(${\bf I}$)$, one can further verify that $\text{supp }\mu_*$ is precisely the attainable set from the singleton $z$.
\end{remark}

\subsubsection{Extending mixing to the original space.} It  remains to demonstrate global exponential mixing for $\{x_n;n\in\N\}$, acting on the entire state space $\mathcal X$; see Step 3 of Section \ref{Section-generalresult}. This will be done by combining the $\mathcal Y_\infty$-restricted mixing described in Proposition~\ref{local exponential mixing}, with the exponential attraction of the invariant compactum $\mathcal Y_\infty$ (due to hypothesis~(\textbf{AC})).

\begin{proposition}\label{prop-globalmixing}
Under the assumptions of Theorem~{\rm\ref{Thm A}}, the invariant measure $\mu_*$, established in Proposition~{\rm\ref{local exponential mixing}}, is globally exponential mixing in the sense of {\rm(\ref{Exponential mixing})}.
\end{proposition}

\begin{proof}[{\bf Proof}]
To verify \eqref{Exponential mixing}, it suffices to show that there exist constants $C,\beta>0$ such that
\begin{equation}\label{Exponential mixing2}
    |P_nf(x)-\langle f,\mu_*\rangle|\leq C(1+V(x))e^{-\beta n}
\end{equation}
for any $f\in L_b(\mathcal{X})$ with $\|f\|_{L}\leq 1$, $x\in\mathcal{X}$ and $n\in\N$.

We claim that, in view of the compactness  of $\mathcal{Y}$, there exists a measurable map $\mathsf{P}\colon\mathcal{X}\rightarrow\mathcal{Y}$ such that
\begin{equation}\label{Y projection}
d(x,\mathsf{P}x)\leq 2\text{dist}(x,\mathcal{Y})
\end{equation}
for any $x\in\mathcal{X}$.
Indeed, let $\{y_n;n\in\N\}$ be a dense sequence in $\mathcal{Y}$, and 
$$
\quad A_n=\left\{z\in\mathcal{X};d(z,y_n)<2\text{dist}(z,\mathcal{Y})\right\}\setminus \left(\bigcup_{0\leq j\leq n-1} A_j\right).
$$
Then one can check that $\mathcal{X}=\mathcal{Y}\cup (\bigcup_{n\in\N}A_n)$ and the sets $A_n$, $\mathcal{Y}$ are disjoint. It thus follows that the desired map $\mathsf{P}$ can be taken as 
\begin{equation*}
\mathsf{P}\colon\mathcal{X}\rightarrow\mathcal{Y},\quad\mathsf{P}x=\begin{cases}
y_n\quad &\text{ for }x\in A_n,\\
x\quad &\text{ for } x\in\mathcal{Y}.
\end{cases}
\end{equation*}

Let $x\in\mathcal{X}$ be arbitrarily given and recall the alternative expression (\ref{rewrite-form}) for $\{x_n;n\in\N\}$. 
We also define the shifted sequences by $\boldsymbol{\xi}^j=\{\xi_{n+j};n\in\N\}$ for $j\in\N^+$, which is independent of $x_j$. With these settings, we compute that \begin{equation}\label{P_j+k-mu}
\begin{aligned}
|P_{k+j}f(x)-\langle f,\mu_*\rangle|
&=|\E_{x}f(x_{k+j})-\langle f,\mu_*\rangle|\\
&\leq |\E_{x}f(S_k(\mathsf{P}x_j;\boldsymbol{\xi}^j))-\langle f,\mu_*\rangle|\\
&\quad+|\E_{x}[f(S_{k}(x_j;\boldsymbol{\xi}^j))-f(S_k(\mathsf{P}x_j;\boldsymbol{\xi}^j))]|\\
& =:I_1+I_2
\end{aligned}
\end{equation}
for any $f\in L_b(\mathcal{X})$ with  $\|f\|_{L}\leq 1$ and $j,k\in\N$. In the sequel, we intend to estimate each $I_i$ separately.

From (\ref{exponential mixing on Y-infty}) it follows that 
\begin{equation}\label{I2 bound}
\begin{aligned}
I_1&=|\E_{x}[\E_{x}f(S_k(\mathsf{P}x_j;\boldsymbol{\xi}^j))-\langle f,\mu_*\rangle|\mathcal{F}_j]|\\
&\leq \E_{x}|\E_{y}(f(S_k(y;\boldsymbol{\xi}^j))-\langle f,\mu_*\rangle)|_{y=\mathsf{P}x_j}|\\
&\leq C_0e^{-\beta_0 k},
\end{aligned}
\end{equation}
where $\mathcal{F}_n$ denotes the natural filtration of $\{x_n;n\in\N\}$. In particular, let us mention that the RHS in (\ref{I2 bound}) is independent of $j\in\N$. 

Thus, it suffices to  get control over the size of $I_2$. To this end, we observe,  in view of (\ref{asymptotic compactness}) included in hypothesis ({\bf AC}), that
\begin{equation}\label{xn asymptotic compactness}
\text{dist}(S_n(x;\boldsymbol{\xi}),\mathcal{Y})\leq V(x)e^{-\kappa n}
\end{equation}
almost surely for any $n\in\N$. On the other hand, one can derive, from 
the compactness of $\mathcal{Y}$, that there exists a constant $R>0$ such that 
$$
\{S_n(y;\boldsymbol{\xi});y\in\mathcal{Y},n\in\N\}\subset N_R(\mathcal{Y}):=\{y\in\mathcal{X};\text{dist}(y,\mathcal{Y})<R\},
$$
almost surely.
Then, taking (\ref{xn asymptotic compactness}) into account, one gets that 
\begin{equation*}
\{x_n;n\geq K\}\subset N_R(\mathcal{Y}),
\end{equation*}
almost surely, where $K:=\lceil(\ln (V(x)R^{-1}))/{\kappa}\rceil$.
As a consequence, 
$$
S_{k}(x_j;\boldsymbol{\xi}^j), S_{k}(\mathsf{P}x_j;\boldsymbol{\xi}^j)\in N_R(\mathcal{Y})
$$
for any $k\in\N$ and $j\geq K$.
In view of the Lipschitz continuity of $S$ on $N_R(\mathcal{Y})\times \mathcal E$, there exists a constant $L\geq 1$ such that
\begin{equation}\label{I1 bound}
\begin{aligned}
I_2&\leq \E_{x}d(S_{k}(x_j;\boldsymbol{\xi}^j),S_k(\mathsf{P}x_j;\boldsymbol{\xi}^j))\\
&\leq  L^k\E_{x} d(x_j,\mathsf{P}x_j)\\
&\leq  2V(x)L^ke^{-\kappa j},
\end{aligned}
\end{equation}
where the last inequality follows from (\ref{Y projection}) and (\ref{xn asymptotic compactness}).

We are now prepared to prove (\ref{Exponential mixing2}).  Plugging (\ref{I2 bound}),(\ref{I1 bound}) into (\ref{P_j+k-mu}), it follows that
\begin{equation*}
\begin{aligned}
|P_{n}f(x)-\langle f, \mu_*\rangle|\leq 2V(x)L^ke^{-\kappa j}+ C_0e^{-\beta_0 k}
\end{aligned}
\end{equation*}
for any $n=k+j$ with $k\geq 0$ and $j\geq K$, where we recall that $f\in L_b(\mathcal{X})$ with  $\|f\|_{L}\leq 1$ is arbitrary. For $\varepsilon\in(0,1)$ to be specified below, we set 
$$
k=\lfloor\varepsilon n\rfloor,\quad j=\lceil(1-\varepsilon)n\rceil,
$$
under which it can be derived that
\begin{equation*}
|P_{n}f(x)-\langle f, \mu_*\rangle|\leq 2V(x)e^{(-\kappa +\varepsilon(\kappa+\ln L))n}+ C_0e^{\beta_0}e^{-\beta_0 \varepsilon n}
\end{equation*}
for any $n>K/(1-\varepsilon)$. In conclusion, taking
$$
\begin{cases}
\varepsilon<\frac{\kappa}{\kappa+\ln L},\ 
\beta=\min\left\{\kappa-\varepsilon(\kappa+\ln L),\beta_0\varepsilon,(1-\varepsilon)\kappa
\right\},\\
C=2\max\{
C_0e^{\beta_0},e^{\beta/(1-\varepsilon)}R^{-\beta/((1-\varepsilon)\kappa)}
\},
\end{cases}
$$
we have 
\begin{equation*}
|P_{n}f(x)-\langle f, \mu_*\rangle|\leq C(1+V(x))e^{-\beta n}
\end{equation*}
for any $n>K/(1-\varepsilon)$, while in the case of $n\leq K/(1-\varepsilon)$,
\begin{align*}
|P_{n}f(x)-\langle f, \mu_*\rangle| & \leq Ce^{-\beta /(1-\varepsilon)}R^{\beta/((1-\varepsilon)\kappa)}\\
& \leq Ce^{-\beta /(1-\varepsilon)}R^{\beta/((1-\varepsilon)\kappa)}(1+V(x))V(x)^{-\beta/((1-\varepsilon)\kappa)}\\
& \leq C(1+V(x))e^{-\frac{\beta}{1-\varepsilon}[\frac{\ln(V(x)R^{-1})}{\kappa}+1]}\\
& \leq C(1+V(x))e^{-\beta n},
\end{align*}
where the second inequality is due to $\beta\leq (1-\varepsilon)\kappa$ (and hence $1+s\geq s^{\beta/((1-\varepsilon)\kappa)}$ for any $s\geq 0$).
The proof is then complete.
\end{proof}

Summarizing Propositions \ref{Y_infty}--\ref{prop-globalmixing}, 
the proof of Theorem~\ref{Thm A} is now complete.
	
\section{Global stability and energy profiles of waves}\label{Section-PDE}

In this section, we shall describe a consequence of Theorem \ref{Thm-informal-GS}, i.e., Proposition~\ref{prop-global-stability} below. This proposition ensures the global stability of  zero equilibrium for the unforced problem, i.e., (\ref{semilinear-problem}) with $f(t,x)\equiv 0$.  Such property will play an essential role in the verification of irreducibility (see hypothesis $(\mathbf{H})$ in Section \ref{Section-1-1}) for \eqref{Problem-0}, where the details are contained in Section \ref{Section-6-2}.
In addition, we present some energy characterizations for solutions of linear/nonlinear wave equations, which will be useful in  our analyses of dynamical systems and control problems; see Sections \ref{Section-nonautonomous} and \ref{Section-control}. 

For any two Banach spaces $\mathcal X,\mathcal Y$, the notation $\mathcal L(\mathcal X;\mathcal Y)\ (\mathcal L(\mathcal X)=\mathcal L(\mathcal X;\mathcal X)\text{ for abbreviation})$ stands for the space of bounded linear operator from $\mathcal X$ into $\mathcal Y$, equipped with the usual operator norm. We denote by $\langle\cdot,\cdot\rangle_{\mathcal X,\mathcal X^*}$ the scalar product between $\mathcal X$ and its dual space $\mathcal X^*$. When $\mathcal X$ is also a Hilbert space, $(\cdot,\cdot)_{_{\mathcal X}}$ stands for its inner product.

To continue, we introduce the functional settings for models (\ref{Problem-0}),(\ref{semilinear-problem}). We write $\|\cdot\|=\|\cdot\|_{_{L^2}}$ and  $(\cdot,\cdot)=(\cdot,\cdot)_{_{L^2}}$ for simplicity. 
Recall that $H^s\ (s>0)$ denotes the domain of the fractional power $(-\Delta)^{s/2}$, which can be characterized via
$$
H^s=\{\phi\in H;\sum_{j\in\N^+}\lambda_j^s|(\phi,e_j)|^2<\infty
\}
$$ 
and is equipped with the graph norm
$$
\left\|(-\Delta)^{s/2}\phi\right\|^2=\sum_{j\in\N^+}\lambda_j^s|(\phi,e_j)|^2.
$$
It also follows that $H^s=H^s(D)$ for $0\leq s< 1/2$ and $H^s=\{\phi\in H^s(D);\phi|_{\partial D}=0 \}$ for $1/2< s\leq 2$. The dual space of $H^s$ is denoted by $H^{-s}$, which can be regarded as the completion of $H$ with respect to the norm $\|(-\Delta)^{-s/2}\cdot\|$.
Let us also set
$
\mathcal H^s=H^{1+s}\times H^s$ and $
\mathcal X^s_T=C([0,T];H^{1+s})\cap C^1([0,T];H^s)
$
with $s\in\R
$ and $T>0$.
For simplicity, we write $\mathcal H=\mathcal H^0$ and $\mathcal X_T=\mathcal X^0_T$. 
Denote 
\begin{equation}\label{potential-space}
B_R=B_{C([0,T];H^{\scriptscriptstyle 11/7})}(R)
\end{equation}
with $R>0$. If there is no danger of confusion, we denote $L^q_tL_x^r=L^q(\tau,\tau+T;L^r(D))$ and $L^q_tH_x^s=L^q(\tau,\tau+T;H^s)$, where $\tau\geq0$ and  $q,r\geq 1$.

\subsection{The linear problem}
We in this subsection concentrate on the linear equation
\begin{equation}\label{linear-problem}
\begin{cases}
\boxempty v+b(x)\partial_t v+p(t,x)v=f(t,x),\quad x\in D,\\
v[0]=(v_0,v_1):=v^0,
\end{cases}
\end{equation}
on time interval $[0,T]$, where $b\in C^\infty(\overline{D})$ and $p\in C([0,T];H^{\scriptscriptstyle 11/7})$. We denote by 
$$
v=\mathcal V_{p}(v_0,v_1,f)=\mathcal V_{p}(v^0,f)
$$
the solution of (\ref{linear-problem}). Here, the initial state $v^0$ and the force $f$ will be chosen to be in various spaces, and so is $\mathcal V_{p}(v^0,f)$. These solutions are defined by using the formula of variations of constants, i.e.,
\begin{equation}\label{variation-formula-1}
v[t]=U_b(t)v^0+\int_0^tU_b(t-s)\left(\begin{matrix}
0\\
-p(s)v(s)+f(s)
\end{matrix}
\right)ds,
\end{equation}
where $U_b(t),t\in\R$ stands for the $C_0$-group on $\mathcal H$ associated with the autonomous linear equation $\boxempty v+b(x)\partial_tv=0$. Moreover, $U_b(t)$ is also a $C_0$-group on $\mathcal H^s$ for every $s\in\R$.
	
When the initial condition is replaced with the terminal condition
$
v[T]=(v_0^T,v_1^T):=v^T,
$
the corresponding solution is denoted by 
$$
v=\mathcal V^T_{p}(v_0^T,v_1^T,f)=\mathcal V^T_{p}(v^T,f);
$$
notice that the wave equation (\ref{linear-problem}) is time-reversible. In this situation, the solution is given by the formula of variations of constants of a time-reversible version, i.e.,
\begin{equation}\label{variation-formula-2}
v[t]=U_b(t-T)v^T-\int_t^TU_b(t-s)\left(
\begin{matrix}
0\\
-p(s)v(s)+f(s)
\end{matrix}
\right)ds.
\end{equation}
When $f=0$, let us denote $\mathcal V^T_{p}(v^T)=\mathcal V^T_{p}(v^T,0)$ for the sake of simplicity.
	
Some characterizations of $\mathcal V_{p},\mathcal V_{p}^T$ are collected as the following proposition.
	
\begin{proposition}\label{Energy-estimate}
Let $T,R>0$ and $s\in[0,1/5]$. Then
the following assertions hold. 
\begin{enumerate}[leftmargin=2em]
\item[$(1)$] There exists a constant $C_1>0$ such that
\begin{equation}\label{energy-1}
\sup_{t\in[0,T]}\|v[t]\|^2_{_{\mathcal H^s}}\leq C_1\left[\|v^0\|^2_{_{\mathcal H^s}}+\int_0^T\|f(t)\|^2_{_{H^s}}dt
\right]
\end{equation}
for any $p\in B_R,v^0\in\mathcal H^s$ and $f\in L^2_tH_x^{s}$, where $v=\mathcal  V_{p}(v^0,f)\in \mathcal X_T^s.$ Moreover, the estimate of type {\rm(\ref{energy-1})} also holds with $V_{p}(v^0,f)$ replaced by $\mathcal  V_{p}^T(v^T,f),v^T\in\mathcal{H}^s,f\in L^2_tH_x^{s}$.
			
\item[$(2)$] There exists a constant $C_2>0$ such that
\begin{equation}\label{energy-2}
\|v[t]\|^2_{_{\mathcal H^{-1-s}}}\leq C_2\|v[\tau]\|^2_{_{\mathcal H^{-1-s}}}	
\end{equation}
for any $p\in B_R,v^T\in \mathcal H^{-1-s}$ and $t,\tau\in[0,T]$, where $v=\mathcal V_{p}^T(v^T)\in \mathcal X_T^{-1-s}$.

\item[$(3)$] 
Denoting $v=\mathcal V^T_p(v^T)$ with $v^T\in\mathcal H^{-1-s}$, the mapping 
$$
B_R\ni p\mapsto (v,\partial_tv)\in \mathcal L(\mathcal H^{-1-s};C([0,T];\mathcal H^{-1-s}))
$$
is Lipschitz and continuously differentiable.
\end{enumerate}	
\end{proposition}
	
These conclusions can be proved by means of the formulas (\ref{variation-formula-1}) and (\ref{variation-formula-2}) together with the Gronwall-type inequality. In Proposition \ref{Energy-estimate}, both the regularity assumption on $p$ and the range of values for $s$ correspond to the context of our control arguments in Section \ref{Section-control}. However, these restrictions are in fact not ``optimal'', as our emphasis is not on sharp conditions for the relevant properties.
	
In addition to  inequality (\ref{energy-1}) in Proposition \ref{Energy-estimate}, another useful estimate for $H^1$-solutions of wave equations is the Strichartz inequality; see Proposition \ref{prop-Strichartz} below. This inequality involves the $L^r$-norm (with $r>6$) in space and, in exchange, only the $L^q$-norm (with $q<\infty$) in time.
In comparison, the aforementioned inequality is of $L^\infty$ in time and of $H^1$ in space, while $H^1$ is not included in $L^r$ with $r>6$.

\begin{proposition}\label{prop-Strichartz}
Let $T>0$ and the pair $(q,r)$ satisfy
\begin{equation}\label{Strichartz-index}
\frac{1}{q}+\frac{3}{r}=\frac{1}{2},\quad q\in[7/2,+\infty].
\end{equation}
Then there exists a constant $C=C(T,q)>0$ such that 
$$
\|v\|_{_{L^q_tL^r_x}}\leq C\left[\|v^0\|_{_{\mathcal H}}+\|f\|_{_{L^1_tL^2_x}}
\right]
$$
for any $v^0\in\mathcal H$ and $f\in L^1_tL^2_x$, where $v=\mathcal V_0(v^0,f)\in\mathcal X_T$.
\end{proposition}
	
This can be derived from \cite[Corollary 1.2]{BSS-09} (see also \cite[Theorem 2.1]{JL-13}).

The Strichartz estimates (also called dispersive estimates) is a significant object in the study of wave equations that has attracted the interest of many authors. In particular, this type of estimate has been developed by Burq--Lebeau--Planchon \cite{BLP-08} for $q\geq 5$ and also by Blair--Smith--Sogge \cite{BSS-09} for a wider range of the indices $q,r$, under the setting of smooth bounded domains in Euclidean spaces (or more generally, compact Riemannian manifold with boundary).

In the present paper, the Strichartz estimate in Proposition \ref{prop-Strichartz} will play an important role, when studying the issue of asymptotic compactness for (\ref{semilinear-problem}) (see Theorem \ref{Thm-informal-AC}).
	
\subsection{The nonlinear problem}
	
We proceed to consider the semilinear wave equation (\ref{semilinear-problem}).
In such case, the $C_0$-group generated by the linear part is denoted by $U(t),t\in\R$ (which coincides with $U_b(t)$ for the case of $b=a$). 
	
Similarly to the case of (\ref{linear-problem}), a solution $u\in\mathcal X_T$ of (\ref{semilinear-problem}) is defined to be a solution of the integral equation
\begin{equation}\label{variation-formula-3}
u[t]=U(t)u^0+\int_0^tU(t-s)\left(
\begin{matrix}
0\\
-u^3(s)+f(s)
\end{matrix}
\right)ds.
\end{equation}
	
\begin{proposition}\label{prop-wellposedness}
Let $T>0$ be arbitrarily given. Then the following assertions hold.
\begin{enumerate}[leftmargin=2em]
\item[$(1)$] For every $u^0\in\mathcal H$ and $f\in L^2(D_T)$, there exists a unique solution $u\in \mathcal X_T$ of {\rm(\ref{semilinear-problem})}. Moreover, the mapping 
\begin{equation}\label{solution-mapping}
\mathcal H\times L^2(D_T)\ni (u^0,f)\mapsto u\in \mathcal X_T
\end{equation}
is locally Lipschitz and continuously differentiable. In particular, the Lipschitz constants are of the form $Ce^{CT}$.
			
\item[$(2)$] If also $u^0\in\mathcal H^{\scriptscriptstyle 4/7}$ and $f\in L^2_tH^{\scriptscriptstyle 4/7}_x$, then $u\in \mathcal X_T^{\scriptscriptstyle 4/7}$. Moreover, the solution mapping given in {\rm(\ref{solution-mapping})} is locally Lipschitz and continuously differentiable from $\mathcal H^{\scriptscriptstyle 4/7}\times L^2_tH^{\scriptscriptstyle 4/7}_x$ into $\mathcal X_T^{\scriptscriptstyle 4/7}$.
\end{enumerate}
\end{proposition}

The proof of Proposition \ref{prop-wellposedness} is fairly standard, so we skip it. 
	
In what follows, we introduce the result of global (exponential) stability for the unforced problem, where condition (\ref{Damping-localization}) on the damping coefficient $a(x)$ comes into play.
Let us begin with the exponential decay of the semigroup $U(t)$.
	
\begin{lemma}\label{Linear-semigroup}
Assume that $a(x)$ satisfies {\rm(\ref{Damping-localization})}. Then there exist constants $C,\gamma>0$ such that
\begin{equation}\label{linear-decay}
\|U(t)\|_{_{\mathcal L(\mathcal H^s)}}\leq Ce^{-\gamma t}
\end{equation}
for any $t\geq 0$ and $s\in[0,1]$.
\end{lemma}

This lemma can be found in \cite[Proposition 2.3]{JL-13}, where the author considered a more general setting of geometric control condition.

The global stability of  zero equilibrium for the unforced problem is stated as follows.

\begin{proposition}\label{prop-global-stability}
Assume that $a(x)$ satisfies {\rm(\ref{Damping-localization})} and $f(t,x)\equiv 0$. Then there exist constants $C,\gamma>0$ such that
\begin{equation}\label{global-stability}
\|u[t]\|_{_{\mathcal H}}^2\leq Ce^{-\gamma t}\left(
\|u^0\|^2_{_{\mathcal H}}+\|u_0\|_{_{H^1}}^4
\right)
\end{equation}
for any $u^0=(u_0,u_1)\in\mathcal H$ and $t\geq 0$, where $u\in C(\R^+;H^1)\cap C^1(\R^+;H)$ stands for the solution of {\rm (\ref{semilinear-problem})}.
\end{proposition}

{This proposition is a direct consequence of Theorem \ref{Thm-informal-GS}.}

\section{Asymptotic compactness in non-autonomous dynamics}\label{Section-nonautonomous}

This section is devoted to establishing the $(\mathcal H,\mathcal H^{\scriptscriptstyle4/7})$-asymptotic compactness for the non-autonomous dynamical system generated by (\ref{semilinear-problem}); see Theorem \ref{Thm-dynamicalsystem} later, which is an exact and stronger statement of Theorem \ref{Thm-informal-AC}. {In addition, we consider the 
asymptotic compactness in a ``physical'' space $\mathcal H^1$, for which
more regularity in time and less regularity in space are imposed on the force $f(t,x)$.}

The main theorem and an outline of its proof is placed in Section \ref{Section-DSresult} below, while Sections \ref{Section-DS-1} and \ref{Section-DS-2} contain the  details.

\subsection{Results and outline of proof}\label{Section-DSresult}

Due to the non-autonomy, it is more convenient to consider initial conditions at  general time $\tau\geq 0$:
\begin{equation}\label{semilinear-problem-tau}
\begin{cases}
\boxempty u+a(x)\partial_t u+u^3=f(t,x),\quad x\in D,\\
u[\tau]=(u_0,u_1)=u^\tau.
\end{cases}
\end{equation}
From the viewpoint of dynamical systems, the main characteristics of (\ref{semilinear-problem-tau}) consist of non-autonomous force and weak dissipation. To be precise, the force $f$ is allowed to be time-dependent, while the damping coefficient $a(x)$ is localized in the sense of setting $(\mathbf{S1})$.

In view of the global well-posedness of (\ref{semilinear-problem-tau}) (see Proposition \ref{prop-wellposedness}(1) above), it generates a process on $\mathcal H$ via
\begin{equation*}
\mathcal U^f(t,\tau)u^\tau=u[t],
\end{equation*}
with $f\in L^\infty(\R^+;H)$,
which verifies that $\mathcal U^f(\tau,\tau)=I$ for all $\tau\geq 0$, $\mathcal U^f(t,\tau)=\mathcal U^f(t,s)\circ \mathcal U^f(s,\tau)$ for all $t\geq s\geq \tau$, and the mapping $(t,\tau,u^\tau)\mapsto \mathcal U^f(t,\tau)u^\tau$ is continuous for $t\geq \tau,u^\tau\in \mathcal H$.

Recall that $E_u(t)=E(u[t])$ is the energy function defined via (\ref{energy-functional}). The main theorem of this section is collected in the following.

\begin{theorem}\label{Thm-dynamicalsystem}
Assume that $a(x)$ satisfies {\rm(\ref{Damping-localization})} and let $R_0>0$ be arbitrarily given. Denote $u[\cdot]=\mathcal U^f(\cdot,\tau)u^\tau$ with $u^\tau,f$ to be specified below.
Then the following assertions hold.
\begin{enumerate}
[leftmargin=2em]
\item[$(1)$] There exists a bounded subset $\mathscr B_{\scriptscriptstyle4/7}$ of $\mathcal H^{\scriptscriptstyle4/7}$ and  constants $C,\kappa>0$ 
such that 
\begin{equation*}
{\rm dist}_{_{\mathcal H}}(\mathcal U^f(t,\tau)u^\tau,\mathscr B_{\scriptscriptstyle4/7})\leq C(1+E_u(\tau))e^{-\kappa (t-\tau)}
\end{equation*}
for any $u^\tau\in\mathcal H,f\in\overline{B}_{L^\infty(\R^+;H^{\scriptscriptstyle4/7})}(R_0)$ and $t\geq \tau$.

\item[$(2)$] There exists a bounded subset $\mathscr B_1$ of $\mathcal H^{1}$ and constants $\hat C,\hat\kappa>0$ 
such that 
\begin{equation*}
{\rm dist}_{_{\mathcal H}}(\mathcal U^f(t,\tau)u^\tau,\mathscr B_1)\leq \hat C(1+E_u(\tau))e^{-\hat\kappa (t-\tau)}
\end{equation*}
for any $u^\tau\in\mathcal H,f\in\overline{B}_{F}(R_0)$ and $t\geq \tau$, where $F=W^{1,\infty}(\R^+;H)\cap L^\infty(\R^+;H^{\scriptscriptstyle1/3})$\footnote{Naturally, the norm on the space $F$ is defined as $\|\cdot\|_{_{F}}:=\|\cdot\|_{_{W^{1,\infty}(\R^+;H)}}+\|\cdot\|_{_{L^\infty(\R^+;H^{\scriptscriptstyle1/3})}}$.}.
\end{enumerate}
\end{theorem}

Either of the assertions indicates also that the non-autonomous dynamical system generated by {\rm(\ref{semilinear-problem-tau})} possesses a uniform attractor {\rm(}see, e.g., {\rm\cite[Part 2]{CV-02})}. 

The proof of main theorem can be divided into three steps: 

\medskip

\noindent {\bf Step 1 ($\mathcal H$-dissipativity).} We first establish the $\mathcal H$-dissipativity for the process $\mathcal U^f(t,\tau)$, i.e.,
the existence of an $\mathcal H$-bounded set $\mathscr B_0$ which absorbs exponentially the trajectories issued from bounded subsets of $\mathcal H$ (see Proposition \ref{prop-S1dissipativity}). For this purpose, we derive that there exist suitably large constants $T_0,A_0>0$ such that for some constant $\varpi \in (0,1)$, 
\begin{equation}\label{discrete-monotonicity}
E_u(\tau)\geq A_0\quad \Rightarrow \quad E_u(\tau+T_0)\leq \varpi  E_u(\tau)
\end{equation}
(see Lemma \ref{Lemma-nonincreasing}), which turns out to be sufficient for the $\mathcal H$-dissipativity. The proof of (\ref{discrete-monotonicity}) involves an essential energy inequality
$$
\int_\tau^{\tau+T} E_u(t)\lesssim
E_u(\tau+T)+\int_\tau^{\tau+T}\int_D a(x)|\partial_tu|^2+\int_\tau^{\tau+T}\int_D\left(
u^2+|f \partial_t u|+|f|^2
\right)
$$
for any $\tau,T\geq 0$ (see Lemma \ref{Lemma-estimate}),
for which the $\Gamma$-type geometric condition (\ref{Damping-localization}) of $a(x)$ is necessary and the multiplier-type techniques will be used. 

\medskip

\noindent {\bf Step 2 ($(\mathcal H,\mathcal H^{\scriptscriptstyle4/7})$-asymptotic compactness).} Thanks to the $\mathcal H$-dissipativity, we are able to focus on the case where $u^\tau\in\mathscr B_0$. With this setting, we split a trajectory $u[\cdot]:=\mathcal U^f(\cdot,\tau)u^\tau$ via
$$
u[t]=U(t-\tau)u^\tau+w[t],
$$
where $w$ stands for the ``nonlinear part'' of $u$ and solves 
\begin{equation}\label{w-equation}
\begin{cases}
\boxempty w+a(x)\partial_t w+u^3=f(t,x),\quad x\in D,\\
w[\tau]=0.
\end{cases}
\end{equation}
Inspired by the work of \cite{JL-13},
the $\mathcal H^{\scriptscriptstyle4/7}$-boundedness of $w[\cdot]$ can be derived by means of a Strichartz-based regularization property of nonlinearity (see Lemma \ref{Lemma-AC-1}). The first assertion of Theorem~\ref{Thm-dynamicalsystem}
then follows, thanks to the damping effect resulted by $a(x)$ (see Lemma \ref{Linear-semigroup}).

\medskip

\noindent {\bf Step 3 ($(\mathcal H,\mathcal H^{1})$-asymptotic compactness).}
The proof of Theorem \ref{Thm-dynamicalsystem}(2) proceeds with the transitivity of exponential attractions. To be precise, the desired result will be derived from the intermediate results of 
\begin{enumerate}[leftmargin=2em]
\item $(\mathcal H,\mathcal H^{\scriptscriptstyle1/3})$-asymptotic compactness (see Corollary \ref{Corollary-bound}), and 

\item $(\mathcal H^{\scriptscriptstyle1/3},\mathcal H^{1})$-asymptotic compactness (see Lemma \ref{Lemma-AC-2}).
\end{enumerate}
We deduce directly the first result from the same argument as in Step 2, except that the $\mathcal H^{\scriptscriptstyle4/7}$-boundedness of $w[\cdot]$ is reduced to be of $\mathcal H^{\scriptscriptstyle1/3}$; notice that  only the $H^{\scriptscriptstyle1/3}$-regularity of $f(t,x)$ is available in this step. To obtain the second, 
we shall invoke the Strichartz estimate (see Proposition \ref{prop-Strichartz}) and the idea of discrete monotonicity analogous to (\ref{discrete-monotonicity}). These enable us to obtain 
$\mathcal H$-boundedness of $\theta[\cdot]$ with $\theta=\partial_tw$, where the extra assumption on the time regularity of $f(t,x)$ comes into play and which leads to the $\mathcal H^1$-boundedness of $w[\cdot]$.

\subsection{Global dissipativity}\label{Section-DS-1}

In this subsection, it suffices to assume that
$
f\in L^\infty(\mathbb R^+;H)
$. The generic constant $C$ involved in the remainder of this section would not depend on special choices of the parameters $u^\tau,f,\tau,T.$

\begin{proposition}\label{prop-S1dissipativity}
Assume that $a(x)$ satisfies {\rm(\ref{Damping-localization})} and let 
$R_1>0$ be arbitrarily given. Then there exists a bounded subset $\mathscr B_0$ of $\mathcal H$ and a constant $p>0$ such that 
$$
\mathcal U^f(\tau+t,\tau)u^\tau\in \mathscr B_0
$$
for any $u^\tau\in\mathcal H$, $f\in \overline{B}_{L^\infty(\R^+;H)}(R_1)$ and $t\geq T,\tau\geq 0$, where the elapsed time $T>0$ is given by  
\begin{equation}\label{elapsed-time}
T=p\ln{(1+pE_u(\tau))}
\end{equation}
with $u[\cdot]=\mathcal U^f(\cdot,\tau)u^\tau$.
\end{proposition}

To begin with, let us recall some elementary estimates for the energy function $E_u$. Notice first the flux estimate
\begin{equation}\label{energy-equality}
E_u(\tau+T)-E_u(\tau)=-\int_\tau^{\tau+T}\int_D a(x)|\partial_tu|^2dxdt+\int_\tau^{\tau+T}\int_D f\partial_t u dxdt
\end{equation}
for any $\tau,T\geq 0$. In addition, by multiplying the equation by $\partial_tu$ and integrating over $D$, one can obtain that
\begin{equation*}
\frac{d}{dt}E_{u}(t)\leq \int_Df\partial_tudx\leq \|f\|\|\partial_tu\|\leq \|f\|\sqrt{2} E_{u}^{1/2}(t)
\end{equation*}
and hence 
\begin{equation}\label{estimate-8}
E^{1/2}_u(t)-E^{1/2}_u(s)\leq \frac{\sqrt{2}}{2}(t-s)\|f\|_{_{L^\infty(\R^+;H)}}
\end{equation}
for any $t\geq s\geq \tau$.

What follows is an elementary but essential inequality for the energy function $E_u$, which is derived by means of the multiplier method as previously mentioned.

\begin{lemma}\label{Lemma-estimate}
Assume that $a(x)$ satisfies {\rm(\ref{Damping-localization})}. Then
there exists a constant $K_0>0$ such that
$$
\begin{aligned}
\int_\tau^{\tau+T} E_u(t)dt &\leq K_0 \left[
E_u(\tau+T)+\int_\tau^{\tau+T}\int_D a(x)|\partial_tu|^2dxdt+\int_\tau^{\tau+T}\int_D\left(u^2+|f \partial_t u|+|f|^2\right)dxdt
\right]
\end{aligned}
$$
for any $u^\tau\in\mathcal H,f\in L^\infty(\mathbb R^+;H)$ and $\tau,T\geq 0$, where $u[\cdot]=\mathcal U^f(\cdot,\tau)u^\tau$.
\end{lemma}

\begin{proof}[{\bf Proof}]
Let $q\in C^{1}(\overline D;\mathbb R^3)$. Multiplying (\ref{semilinear-problem-tau}) by $q\cdot \nabla u$ and integrating over $[\tau,\tau+T]\times D$, it follows that
\begin{align}\label{estimate-1}
& \left.\int_D\partial_t u(q\cdot \nabla u)dx\right|_\tau^{\tau+T}+\frac{1}{2}\int_\tau^{\tau+T}\int_D({\rm div\,} q)\left[
|\partial_tu|^2-|\nabla u|^2-\frac{1}{2}u^4
\right]dxdt \notag\\
& \quad +\sum_{j,k=1}^3\int_\tau^{\tau+T}\int_D\partial_k q_j\partial_ku\partial_judxdt+\int_\tau^{\tau+T}\int_D \left(
a(x)\partial_tu-f
\right)(q\cdot\nabla u)dxdt \notag\\
& =\frac{1}{2}\int_\tau^{\tau+T}\int_{\partial D}(q\cdot n)\left|\frac{\partial u}{\partial n}\right|^2dxdt.
\end{align}
In addition, for $\xi\in C^{1}(\overline D)$ we have
\begin{equation}\label{estimate-2}
\begin{aligned}
&\left.\int_D \xi u\partial_tudx\right|_\tau^{\tau+T}+\int_\tau^{\tau+T}\int_D\xi u(a(x)\partial_tu-f)dxdt+\int_\tau^{\tau+T}\int_D \xi\left(
|\nabla u|^2+u^4
\right)dxdt\\
&=\int_\tau^{\tau+T}\int_D\xi|\partial_tu|^2dxdt-\int_\tau^{\tau+T}\int_Du(\nabla u\cdot\nabla\xi)dxdt.
\end{aligned}
\end{equation}

Next, we take $q=m(x):=x-x_0$ and $\xi=1$ in (\ref{estimate-1}) and (\ref{estimate-2}). It is then obtained that
\begin{equation}\label{estimate-3}
\begin{aligned}
\int_\tau^{\tau+T} E_u(t)dt\leq &-\left.\int_D \partial_tu(m\cdot\nabla u+u)dx\right|_\tau^{\tau+T}\\
&-\int_\tau^{\tau+T}\int_D \left(
a(x)\partial_tu-f
\right)(m\cdot\nabla u+u)dxdt\\
&+\frac{1}{2}\int_\tau^{\tau+T}\int_{\Gamma(x_0)}(m\cdot n)\left|\frac{\partial u}{\partial n}\right|^2dxdt\\
=:& J_1+J_2+J_3,
\end{aligned}
\end{equation}
where the set   $\Gamma(x_0)$ is provided in Definition \ref{Def-Gamma}. 
Let us estimate $J_i$ separately. Taking (\ref{energy-equality}) into account, one sees that
\begin{equation*}
\begin{aligned}
J_1&\leq C\left[
\| u(\tau+T)\|_{_{H^1}}^2+\|\partial_tu(\tau+T)\|^2+\|u(\tau)\|_{_{H^1}}^2+\|\partial_tu(\tau)\|^2
\right]\\
&\leq C\left[
E_u(\tau+T)+E_u(\tau)
\right]\\
&=C\left[
2E_u(\tau+T)+\int_\tau^{\tau+T}\int_D a(x)|\partial_tu|^2dxdt-\int_\tau^{\tau+T}\int_D f\partial_t u dxdt
\right]\\
&=:C \tilde J_1.
\end{aligned}
\end{equation*}
For $J_2$, it is not difficult to check that
\begin{equation}\label{estimate-5}
\begin{aligned}
J_2\leq C\int_\tau^{\tau+T}\left(\|a\partial_t u\|^2+\|f\|^2\right)dt+\frac{1}{2}\int_\tau^{\tau+T}\|\nabla u\|^2dt.
\end{aligned}
\end{equation}
					
To deal with $J_3$, we introduce a cut-off function $h\in C^1(\overline D;\mathbb R^3)$ satisfying
$$
h=n\  {\rm on\ }\Gamma(x_0),\quad
h\cdot n\geq 0\  {\rm on\ }\partial D,\quad
h=0 \ {\rm in\ }D\setminus N_{\delta_1}(x_0),
$$
where $0<\delta_1<\delta$ is arbitrarily given. Then, letting $q=h$ in (\ref{estimate-1}), it follows that
\begin{align}
J_3\leq & \  C\int_\tau^{\tau+T}\int_{\Gamma(x_0)}(h\cdot n)\left|\frac{\partial u}{\partial n}\right|^2dxdt \notag\\
\leq & \ C\left[
\left.\int_{N_{\delta_1}(x_0)}\partial_t u(h\cdot \nabla u)dx\right|_\tau^{\tau+T}+\int_\tau^{\tau+T}\int_{N_{\delta_1}(x_0)}\left(|\partial_t u|^2+|\nabla u|^2+u^4+f^2\right)dxdt \right] \notag\\
\leq & \ 
C\left[
\tilde J_1+\int_\tau^{\tau+T}\int_{N_{\delta_1}(x_0)}\left(|\partial_t u|^2+|\nabla u|^2+u^4+f^2\right)dxdt \right]. \label{estimate-6}
\end{align}
We need to eliminate the terms $|\nabla u|^2$ and $u^4$ in the RHS of (\ref{estimate-6}). For this purpose, let us define another cut-off function $g\in C^1(\overline D;[0,1])$ via
$$
g=1\ {\rm in\ }N_{\delta_1}(x_0),\quad
g=0\  {\rm in\ }D\setminus N_{\delta}(x_0).
$$
We then apply (\ref{estimate-2}) again with $\xi=g$ to deduce that
\begin{align*}
&  \int_\tau^{\tau+T}\int_{N_{\delta}(x_0)}g\left(|\nabla u|^2+u^4
\right)dxdt\\
& =-\left.\int_{N_{\delta}(x_0)} g u\partial_tudx\right|_\tau^{\tau+T}-\int_\tau^{\tau+T}\int_{N_{\delta}(x_0)}g u(a(x)\partial_tu-f)dxdt \\
& \quad +\int_\tau^{\tau+T}\int_{N_{\delta}(x_0)}g|\partial_tu|^2dxdt -\int_\tau^{\tau+T}\int_{N_{\delta}(x_0)}u(\nabla u\cdot\nabla g)dxdt\\
& \leq C\left[\tilde J_1+\int_\tau^{\tau+T}\int_{N_{\delta}(x_0)}\left(u^2+|\partial_tu|^2+|f|^2\right)dxdt+
\int_\tau^{\tau+T}\int_{N_{\delta}(x_0)}|u(\nabla u\cdot\nabla g)|dxdt
\right].
\end{align*}
For the last term, one can derive that
\begin{align*}
\int_\tau^{\tau+T}\int_{N_{\delta}(x_0)}|u(\nabla u\cdot\nabla g)|dxdt\leq \varepsilon\int_\tau^{\tau+T}\int_{D} |\nabla u|^2dxdt+C(\varepsilon)\int_\tau^{\tau+T}\int_{N_{\delta}(x_0)} u^2dxdt,
\end{align*}
where $\varepsilon\in (0,1)$ and $C(\varepsilon)>0$ denotes a constant depending on $\varepsilon$.
Consequently,
\begin{equation*}
\begin{aligned}
&\int_\tau^{\tau+T}\int_{N_{\delta_1}(x_0)}\left(
|\nabla u|^2+u^4
\right)dxdt\\
&\leq\int_\tau^{\tau+T}\int_{N_{\delta}(x_0)}g\left(
|\nabla u|^2+u^4
\right)dxdt\\
&\leq C\left[
\tilde J_1+\int_\tau^{\tau+T}\int_{N_{\delta}(x_0)}\left(u^2+|\partial_tu|^2+|f|^2\right)dxdt
\right]+C\varepsilon\int_\tau^{\tau+T}\int_{D} |\nabla u|^2dxdt.
\end{aligned}
\end{equation*}
This together with (\ref{estimate-6}) leads to
\begin{equation}\label{estimate-7}
J_3\leq  C\left[
\tilde J_1+\int_\tau^{\tau+T}\int_{N_{\delta}(x_0)}\left(u^2+|\partial_t u|^2+f^2\right)dxdt \right]+C\varepsilon\int_\tau^{\tau+T}\int_{D} |\nabla u|^2dxdt.
\end{equation}
Putting condition (\ref{Damping-localization}) and inequalities (\ref{estimate-3})-(\ref{estimate-5}),(\ref{estimate-7}) (with a sufficiently small $\varepsilon$) all together, we deduce that
\begin{equation*}
\begin{array}{ll}
\displaystyle\int_\tau^{\tau+T} E_u(t)dt\leq C\left[
\tilde J_1+\int_\tau^{\tau+T}\int_{D} \left(a(x)|\partial_t u|^2+f^2+u^2
\right)dxdt+\int_\tau^{\tau+T}\int_{N_{\delta}(x_0)}|\partial_tu|^2dxdt
\right],
\end{array}
\end{equation*}
which leads to the conclusion of this lemma.
\end{proof}
				
On the basis of Lemma \ref{Lemma-estimate}, we can verify that when the energy of a solution is suitably large, it could enjoy a property of discrete monotonicity, which remains sufficient for the construction of an $\mathcal H$-absorbing set. 

\begin{lemma}\label{Lemma-nonincreasing}
Assume that $a(x)$ satisfies {\rm(\ref{Damping-localization})}. Let $\varpi\in(0,1)$ be arbitrarily given and $K_0>0$ established in Lemma {\rm\ref{Lemma-estimate}}. Take $T_0>0$ such that
\begin{equation}\label{large-T0}
T_0>\frac{K_0(13-4\varpi)}{\varpi}.
\end{equation}
Then for every $R_1>0$, there exists a constant $A_0=A_0(T_0,R_1,\varpi)>0$ such that the implication
\begin{equation*}
E_u(\tau)\geq A_0\quad \Rightarrow\quad E_u(\tau+T_0)\leq \varpi E_u(\tau)
\end{equation*}
holds for any $u^\tau\in\mathcal H,f\in \overline{B}_{L^\infty(\R^+;H)}(R_1)$ and $\tau\geq 0$, where $u[\cdot]=\mathcal U^f(\cdot,\tau)u^\tau$. 
\end{lemma}
\begin{proof}[{\bf Proof}]
We argue by contradiction. It is for the moment assumed that there exist sequences
$$
A^n\geq 1,\quad \tau^n\geq 0,\quad (u_0^n,u_1^n)\in\mathcal H,\quad f^n\in \overline{B}_{L^\infty(\R^+;H)}(R_1)
$$
such that
\begin{align}
& E_{u^n}(\tau^n)\geq A^n\rightarrow \infty, \label{contradiction-DS-1} \\
& E_{u^n}(\tau^n+T_0)>\varpi E_{u^n}(\tau^n), \label{contradiction-DS-2} 
\end{align}
where $u^n[\cdot]=\mathcal U^{f^n}(\cdot,\tau^n)(u_0^n,u_1^n)$.
					
Using (\ref{estimate-8}) and (\ref{contradiction-DS-1}), one has
\begin{equation*}
E_{u^n}^{1/2}(\tau^n+t)\leq E_{u^n}^{1/2}(\tau^n)+\frac{\sqrt{2}}{2} R_1 T_0\leq \frac{3}{2}E_{u^n}^{1/2}(\tau^n)
\end{equation*}
for any $t\in[0,T_0]$.
In addition, we invoke (\ref{estimate-8}) again and notice (\ref{contradiction-DS-2}) to derive
$$
E_{u^n}^{1/2}(\tau^n+t)\geq  E_{u^n}^{1/2}(\tau^n+T_0)-\frac{\sqrt{2}}{2} R_1 T_0\geq \frac{\varpi^{1/2}}{2}E_{u^n}^{1/2}(\tau^n),
$$ 
provided that $A_n^{1/2}\geq \sqrt{\frac{2}{\varpi}}R_1T_0$.
In summary,
\begin{equation}\label{estimate-9}
\frac{\varpi^{1/2}}{2}E_{u^n}^{1/2}(\tau^n)\leq E_{u^n}^{1/2}(\tau^n+t)\leq \frac{3}{2}E_{u^n}^{1/2}(\tau^n)
\end{equation}
for any $t\in[0,T_0]$.
					
At the same time, by noticing (\ref{energy-equality}) and (\ref{estimate-9}) we observe that
\begin{equation}\label{estimate-10}
\begin{aligned}
E_{u^n}(\tau^n+T_0)-E_{u^n}(\tau^n)&\leq -\int_{\tau^n}^{\tau^n+T_0}\int_D a(x)|\partial_tu^n|^2dxdt+ R_1\sqrt{2} \int_{\tau^n}^{\tau^n+T_0}E_{u^n}^{1/2}(t)dt\\
&\leq -\int_{\tau^n}^{\tau^n+T_0}\int_D a(x)|\partial_tu^n|^2dxdt+\frac{3\sqrt{2}}{2} R_1 T_0E_{u^n}^{1/2}(\tau^n).
\end{aligned}
\end{equation}
Moreover, an application of Lemma \ref{Lemma-estimate} (with $u=u^n$) leads to
\begin{equation*}
\begin{aligned}
&-\int_{\tau^n}^{\tau^n+T_0}\int_D a(x)|\partial_tu^n|^2dxdt\\
&\leq-\frac{1}{K_0} \int_{\tau^n}^{\tau^n+T_0} E_{u^n}(t)dt+E_{u^n}(\tau^n+T_0)+
\int_{\tau^n}^{\tau^n+T_0}\int_D\left[
(u^n)^2+|f^n \partial_t u^n|+(f^n)^2
\right]dxdt
\end{aligned}
\end{equation*}
Again by (\ref{estimate-9}), it can be derived that
\begin{equation*} 
\int_{\tau^n}^{\tau^n+T_0}\int_D(u^n)^2 dxdt\leq 2|D|^{1/2}\int_{\tau^n}^{\tau^n+T_0}E_{u^n}^{1/2}(t)dt \leq  3|D|^{1/2}T_0E_{u^n}^{1/2}(\tau^n),
\end{equation*}
where $|D|$ denotes the volomn of $D$,
and (similarly to (\ref{estimate-10}))
\begin{equation*}
\begin{array}{ll}
\displaystyle \int_{\tau^n}^{\tau^n+T_0}\int_D|f^n \partial_t u^n| dxdt\leq
\frac{3\sqrt{2}}{2} R_1 T_0E_{u^n}^{1/2}(\tau^n).
\end{array}
\end{equation*}
Then we infer that
\begin{equation*}
\begin{aligned}
&-\int_{\tau^n}^{\tau^n+T_0}\int_D a(x)|\partial_tu^n|^2dxdt\\
&\leq -\left(\frac{\varpi T_0}{4K_0}-\frac{9}{4}\right)E_{u^n}(\tau^n)+\left(3|D|^{1/2}T_0+	\frac{3\sqrt{2}}{2} R_1 T_0\right)E_{u^n}^{1/2}(\tau^n)+R_1^2T_0.
\end{aligned}
\end{equation*}
Inserting this into (\ref{estimate-10}) and noticing (\ref{contradiction-DS-2}), it follows that
\begin{align}
0< & \  E_{u^n}(\tau^n+T_0)-\varpi E_{u^n}(\tau^n)\notag\\
\leq & \  -\left[\frac{\varpi T_0}{4K_0}-\frac{9}{4}-(1-\varpi)\right]E_{u^n}(\tau^n)+\left(3|D|^{1/2}T_0+3\sqrt{2}R_1 T_0\right)E_{u^n}^{1/2}(\tau^n)+R_1^2T_0.\label{estimate-11}
\end{align}
Due to (\ref{large-T0}) and (\ref{contradiction-DS-1}), 
$$
{\rm RHS\ of}\ (\ref{estimate-11})\rightarrow -\infty
$$
as $n\rightarrow \infty$. This gives rise to a contradiction.
The proof is then complete.
\end{proof}
				
The discrete monotonicity of the energy for (\ref{semilinear-problem-tau}) makes it ``natural'' to derive its global dissipativity in the scale of $\mathcal H$.

\begin{proof}[{\bf Proof of Proposition \ref{prop-S1dissipativity}}]
Let $R_1>0$ be arbitrarily given and $T_0,A_0$ the constants established in Lemma \ref{Lemma-nonincreasing}. Making use of the discrete monotonicity, it is not difficult to check that the process $\mathcal U^f(t,\tau)$ is uniformly bounded for $t\geq\tau$. That is,
for every $R_2>0$, there exists a constant $C=C(R_1,R_2)>0$ such that
\begin{equation}\label{S1-boundedness}
\|\mathcal U^f(t,\tau)u^\tau\|_{_{\mathcal H}}\leq C
\end{equation}
for any $u^\tau\in 
\overline{B}_{\mathcal H}(R_2),f\in \overline{B}_{L^\infty(\R^+;H)}(R_1)$ and $t\geq \tau$. Next, let us define
$$
\mathscr B_{01}=\left\{
\psi\in \mathcal H;E(\psi)\leq A_0
\right\},\quad \mathscr B_0=\{ \mathcal U^f(t,\tau)u^\tau; t\geq \tau,
u^\tau\in \mathscr B_{01},f\in \overline{B}_{L^\infty(\R^+;H)}(R_1)
\},
$$
where $E$ is defined as in (\ref{energy-functional}).
Clearly, $\mathscr B_{01}\subset \mathscr B_0$. In addition, taking (\ref{S1-boundedness}) into account, one can observe that $\mathscr B_0$ is bounded in $\mathcal H$. What follows is to illustrate that $\mathscr B_0$ is a uniform absorbing set.

For an arbitrarily given $u^\tau\in\mathcal H$, we define
$$
M=\lceil |\ln{\varpi}|^{-1}\ln{(1+A_0^{-1}E(u^\tau))}\rceil.
$$
Below is to show that
\begin{equation}\label{claim-1}
\min\{E_u(\tau+nT_0);n= 0,1,\cdots,M\}\leq A_0.
\end{equation}
Otherwise, one can check readily that 
\begin{equation*}
E_u(\tau+nT_0)>A_0,\quad \forall\, n=0,1,\cdots,M,
\end{equation*}
where $u[\cdot]=\mathcal U^f(\cdot,\tau)u^\tau$.
Thanks to Lemma \ref{Lemma-nonincreasing}, it follows that
\begin{equation*}
E_u(\tau+nT_0)\leq  \varpi E_u(\tau+(n-1)T_0),\quad \forall\, n=1,\cdots,M,
\end{equation*}
which implies that
\begin{equation*}
E_u(\tau+MT_0)\leq \varpi^ME_u(\tau)=\varpi^ME(u^\tau)\leq A_0.
\end{equation*}
This leads to a contradiction, which means (\ref{claim-1}). Hence, there exists a time
$$
\tau'\in \{
\tau+nT_0;n=0,1,\cdots,M
\}
$$
such that the energy could not exceed $A_0$, i.e., $\mathcal U^f(\tau',\tau)u^\tau\in \mathscr B_{01}$. Accordingly,
$$
\mathcal U^f(\tau+t,\tau)u^\tau\in\mathscr B_0
$$
for any $t\geq MT_0$, where we have used the cocycle property
$$
\mathcal U^f(\tau+t,\tau)u^\tau=\mathcal U^f(\tau+t,\tau')\circ \mathcal{U}^f(\tau',\tau)u^\tau.
$$
The proof is then complete.
\end{proof}
				
For the sake of convenience, the uniform $\mathcal H$-boundedness for $\mathcal U^f(t,\tau)$, which has been presented by (\ref{S1-boundedness}), is collected as the following corollary.

\begin{corollary}\label{Corollary-S1bound}
Assume that $a(x)$ satisfies {\rm(\ref{Damping-localization})} and let $R_1>0$ be arbitrarily given. Then there exists a constant $C>0$ such that
\begin{equation*}
\|\mathcal U^f(t,\tau)u^\tau\|_{_{\mathcal H}}\leq C
\end{equation*}
for any $u^\tau\in
\overline{B}_{\mathcal H}(R_1),f\in \overline{B}_{L^{\infty}(\R^+;H)}(R_1)$ and $t\geq \tau$.
\end{corollary}

\subsection{Asymptotic compactness}\label{Section-DS-2}

We begin with a Strichartz-based regularization property of cubic nonlinearity. 

\begin{lemma}\label{Lemma-Regularity}
Let $R>0,s\in[0,1)$ and
$
\varepsilon=\min\{1-s,4/7\}. 
$
Then there exists a pair $(q,r)$ satisfying {\rm(\ref{Strichartz-index})} such that the following assertion holds: If $u\in L^\infty_tH^{1+s}_x$ is a function with finite Strichartz norms $\|u\|_{_{L^q_tL^r_x}}\leq R$, then $u^3\in L^1_tH_x^{s+\varepsilon}$ and
$$
\|u^3\|_{_{L^1_tH^{s+\varepsilon}_x}}\leq C\|u\|_{_{L^\infty_tH^{1+s}_x}},
$$
where the constant $C>0$ depends only on $q,r,R$.
\end{lemma}

This lemma is a special case of \cite[Corollary 4.2]{JL-13} (see also \cite[Theorem 8]{DLZ-03}). In general, such regularization property remains true with $u^3$ replaced by any defocusing and energy-subcritical nonlinearity $F$:
$$
F(0)=0,\quad sF(s)\geq 0,\quad |F(s)|\leq C(1+|s|)^p,\quad |F'(s)|\leq C(1+|s|)^{p-1},
$$
where $1\leq p<5$. In this case, one takes
$
\varepsilon=\min\{1-s,(5-p)/2,(17-3p)/14\}.
$

With the help of Lemma \ref{Lemma-Regularity}, we shall establish the $(\mathcal H,\mathcal H^{\scriptscriptstyle4/7})$-asymptotic compactness. Recall the constant $\gamma>0$ established in (\ref{linear-decay}).
				
\begin{lemma}\label{Lemma-AC-1}
Assume that $a(x)$ satisfies {\rm(\ref{Damping-localization})} and let $R_0>0$ be arbitrarily given. Let $\mathscr B_0$ be the absorbing set established in Proposition {\rm\ref{prop-S1dissipativity}}, where $R_1$ is chosen so that $\overline{B}_{L^\infty(\R^+;H^{\scriptscriptstyle4/7})}(R_0)\subset \overline{B}_{L^\infty(\R^+;H)}(R_1)$.
Then there exists a bounded subset $\mathscr B_{\scriptscriptstyle4/7}$ of $\mathcal H^{\scriptscriptstyle4/7}$ and a constant $C>0$
such that 
\begin{equation}\label{attracting-4/7}
{\rm dist}_{_{\mathcal H}}(\mathcal U^f(t,\tau)u^\tau,\mathscr B_{\scriptscriptstyle4/7})\leq Ce^{-\gamma (t-\tau)}  
\end{equation}
for any $ u^\tau\in\mathscr B_0,f\in \overline{B}_{L^\infty(\R^+;H^{\scriptscriptstyle4/7})}(R_0)$ and $t\geq \tau$. 
\end{lemma}
				
\begin{proof}[{\bf Proof}]
By means of (\ref{variation-formula-3}), it can be derived that
\begin{equation*}
\begin{aligned}
u[t]&=U(t-\tau)\left(
\begin{matrix}
u_0\\
u_1
\end{matrix}
\right) +\int_0^{t-\tau} U(s)\left(
\begin{matrix}
0\\
-u^3(t-s)
\end{matrix}
\right)ds+\int_\tau^{t} U(t-s)\left(
\begin{matrix}
0\\
f(s)
\end{matrix}
\right)ds\\
&=: I_1(t)+I_2(t)+I_3(t),
\end{aligned}
\end{equation*}
where $u[\cdot]=\mathcal U^f(\cdot,\tau)u^\tau$ with $u^\tau\in\mathscr B_0$ and $f\in \overline{B}_{L^\infty(\R^+;H^{\scriptscriptstyle4/7})}(R_0)$.
Let us treat the terms $I_i$ separately. 
For $I_1$, an application of Lemma \ref{Linear-semigroup} yields that
$$
\|I_1(t)\|_{_{\mathcal H}}\leq Ce^{-\gamma (t-\tau)}.
$$
For $I_2$, we write
\begin{align*}
	I_2(t)&=\sum_{k=0}^{\lfloor t-\tau\rfloor-1}\int_k^{k+1} U(s)\left(
	\begin{matrix}
		0\\
		-u^3(t-s)
	\end{matrix}
	\right)ds+\int_{\lfloor t-\tau\rfloor}^{t-\tau} U(s)\left(
	\begin{matrix}
		0\\
		-u^3(t-s)
	\end{matrix}
	\right)ds
	\\
	&=\sum_{k=0}^{\lfloor t-\tau\rfloor-1}U(k)\int_0^{1} U(s)\left(
	\begin{matrix}
		0\\
		-u^3(t-k-s)
	\end{matrix}
	\right)ds+\int_{\lfloor t-\tau\rfloor}^{t-\tau} U(s)\left(
	\begin{matrix}
		0\\
		-u^3(t-s)
	\end{matrix}
	\right)ds
\end{align*}
Then, making use of Proposition \ref{prop-Strichartz} and Corollary \ref{Corollary-S1bound}, one can observe that for every $(q,r)$ satisfying (\ref{Strichartz-index}),
\begin{align*}
	&\|u(t-k-\cdot)\|_{_{L^q(0,1;L^r(D))}}=
	\|u(\cdot)\|_{_{L^q(t-k-1,t-k;L^r(D))}}\\
	&\leq C\left(\|u[t-k-1]\|_{_{\mathcal H}}+\|-u^3+f\|_{_{L^1(t-k-1,t-k;L^2(D))}}\right)\\
	&\leq C,
\end{align*}
where the constant $C$ does not depend on $t,\tau,k$.  This together with Lemma \ref{Lemma-Regularity} (with $s=0$ and $\varepsilon=4/7$) means that
$$
\|u^3(t-k-\cdot)\|_{_{L^1_tH^{\scriptscriptstyle4/7}_x}}\leq C\|u(t-k-\cdot)\|_{_{L_t^\infty H_x^1}}\leq C.
$$
Analogously, 
$$
\|u^3(t-\lfloor t-\tau\rfloor-\cdot)\|_{_{L^1_tH^{\scriptscriptstyle4/7}_x}}\leq C.
$$
Consequently, we conclude that
$$
\|I_2(t)\|_{_{\mathcal H^{\scriptscriptstyle4/7}}}\leq C\left(\sum_{k=0}^{\lfloor t-\tau\rfloor-1}e^{-\gamma k}+1\right)\leq C\left(\frac{1}{1-e^{-\gamma}}+1\right).
$$
Finally, it is easy to get that
$$
\|I_3(t)\|_{_{\mathcal H^{\scriptscriptstyle4/7}}}\leq CR_0\int_\tau^te^{-\gamma (t-s)}ds\leq CR_0\gamma^{-1}.
$$
In conclusion, there exists a bounded subset $\mathscr B_{\scriptscriptstyle4/7}$ of $\mathcal H^{\scriptscriptstyle4/7}$ such that
$$
I_2(t)+I_3(t)\in \mathscr B_{\scriptscriptstyle4/7}
$$
for all $t\geq \tau$. This combined with the uniform exponential decay of $I_1(t)$ implies the conclusion of this lemma.
\end{proof}

From the proof of Lemma \ref{Lemma-AC-1}, one can also derive that the process $\mathcal U^f(t,\tau)$ sends, uniformly for $t\geq \tau$, bounded subsets of $\mathcal H^{\scriptscriptstyle4/7}$ into bounded subsets.
				
\begin{corollary}\label{Corollary-S2bound}
Assume that $a(x)$ satisfies {\rm(\ref{Damping-localization})} and let $R_0>0$ be arbitrarily given. Then there exists a constant $C>0$ such that
\begin{equation*}
\|\mathcal U^f(t,\tau)u^\tau\|_{_{\mathcal H^{\scriptscriptstyle4/7}}}\leq C
\end{equation*}
for any $u^\tau\in \overline{B}_{\mathcal H^{\scriptscriptstyle4/7}}(R_0),f\in \overline{B}_{L^\infty(\R^+;H^{\scriptscriptstyle4/7})}(R_0)$ and $t\geq \tau$.
\end{corollary}

One can notice that when the assumption of space regularity on $f(t,x)$ is relaxed, the regularity of the attracting set verifying (\ref{attracting-4/7}) becomes lower correspondingly. See the corollary below, where a boundedness result is also involved.

\begin{corollary}\label{Corollary-bound}
Assume that $a(x)$ satisfies {\rm(\ref{Damping-localization})} and let $R_0>0$ be arbitrarily given. 
Then the following assertions hold.
\begin{enumerate}[leftmargin=2em]
\item[$(1)$] Let $\mathscr B_0$ be the absorbing set established in Proposition {\rm\ref{prop-S1dissipativity}}, where $R_1$ is chosen so that $\overline{B}_{L^\infty(\R^+;H^{\scriptscriptstyle1/3})}(R_0)\subset \overline{B}_{L^\infty(\R^+;H)}(R_1)$. There exists a bounded subset $\mathscr B_{\scriptscriptstyle1/3}$ of $\mathcal H^{\scriptscriptstyle1/3}$ and a constant $C_1>0$
such that 
$$
{\rm dist}_{_{\mathcal H}}(\mathcal U^f(t,\tau)u^\tau,\mathscr B_{\scriptscriptstyle1/3})\leq C_1e^{-\gamma (t-\tau)}
$$
for any $ u^\tau\in\mathscr B_0,f\in \overline{B}_{L^\infty(\R^+;H^{\scriptscriptstyle1/3})}(R_0)$ and $t\geq \tau$. 

\item[$(2)$] There exists a constant $C_2>0$ such that
\begin{equation*}
\|\mathcal U^f(t,\tau)u^\tau\|_{_{\mathcal H^{\scriptscriptstyle1/3}}}\leq C_2
\end{equation*}
for any $u^\tau\in \overline{B}_{\mathcal H^{\scriptscriptstyle1/3}}(R_0),f\in \overline{B}_{L^\infty(\R^+;H^{\scriptscriptstyle1/3})}(R_0)$ and $t\geq \tau$.
\end{enumerate}
\end{corollary}

This corollary will be useful in establishing the second assertion of Theorem \ref{Thm-dynamicalsystem}.
Before that, let us complete the proof of the first assertion.
				
\begin{proof}[{\bf Proof of Theorem \ref{Thm-dynamicalsystem}(1)}]
Let $R_0>0$ be arbitrarily given.
We first apply  Proposition \ref{prop-S1dissipativity}, where $R_1$ is chosen so that $\overline{B}_{L^\infty(\R^+;H^{\scriptscriptstyle4/7})}(R_0)\subset \overline{B}_{L^\infty(\R^+;H)}(R_1)$.
It then follows that for every $u^\tau\in\mathcal H$, there exists an elapsed time $T$ of the form (\ref{elapsed-time}), such that
\begin{equation*}
\mathcal U^f(\tau+t,\tau)u^\tau\in\mathscr B_0
\end{equation*}
for any $f\in \overline{B}_{L^\infty(\R^+;H^{\scriptscriptstyle4/7})}(R_0)$ and $t\geq T,\tau\geq 0$. To continue, 
letting $\mathscr B_{\scriptscriptstyle4/7}$ be the set established in Lemma \ref{Lemma-AC-1}, we derive that
\begin{equation*}
{\rm dist}_{_{\mathcal H}}(\mathcal U^f(\tau+t,\tau)u^\tau,\mathscr B_{\scriptscriptstyle4/7})\leq Ce^{-\gamma (t-T)}.
\end{equation*}
This together with (\ref{elapsed-time}) implies that
\begin{equation}\label{ACproof1}
{\rm dist}_{_{\mathcal H}}(\mathcal U^f(\tau+t,\tau)u^\tau,\mathscr B_{\scriptscriptstyle4/7})\leq C(1+E_u(\tau))e^{-\kappa t},
\end{equation}
where $\kappa=\min\{\gamma,(4p)^{-1}\}$ with $p$ arising in (\ref{elapsed-time}).

In the case where $t\in[0,T]$, we make use of (\ref{estimate-8}) to deduce that 
\begin{equation*}
\begin{aligned}
E_u^{1/2}(\tau+t)& \leq E_u^{1/2}(\tau)+\frac{\sqrt{2}}{2}R_1T\\
& \leq E_u^{1/2}(\tau)+C(\ln(1+pE_u(\tau)))\\
& \leq C(1+E_u^{1/2}(\tau)).
\end{aligned}
\end{equation*}
This yields that 
\begin{equation}\label{ACproof2}
\begin{aligned}
{\rm dist}_{_{\mathcal H}}(\mathcal U^f(\tau+t,\tau)u^\tau,\mathscr B_1)& \leq C(1+E_u^{1/2}(\tau))e^{\kappa T}e^{-\kappa t}\\
& \leq C(1+E_u^{1/2}(\tau))(1+E_u(\tau))^{\kappa p}e^{-\kappa t}\\
& \leq C(1+E_u(\tau))e^{-\kappa t}
\end{aligned}
\end{equation}
for any $t\in[0,T]$. Finally, the desired conclusion follows from (\ref{ACproof1}) and (\ref{ACproof2}).
\end{proof}

In order to prove Theorem \ref{Thm-dynamicalsystem}(2), one thing to be done is to verify the $(\mathcal H^{\scriptscriptstyle1/3},\mathcal H^1)$-asymptotic compactness. Let us recall the following Sobolev embeddings:
$$
H^{\scriptscriptstyle1/3}\hookrightarrow L^{\scriptscriptstyle18/7}(D),\quad H^{\scriptscriptstyle4/3}\hookrightarrow L^{\scriptscriptstyle18}(D),
$$
which will be used later without mentioning explicitly.

\begin{lemma}\label{Lemma-AC-2}
Assume that $a(x)$ satisfies {\rm(\ref{Damping-localization})} and let $R_0>0$ be arbitrarily given. Let $\mathscr B_{\scriptscriptstyle1/3}$ be the attracting set established in Corollary {\rm\ref{Corollary-bound}(1)}.
Then there exists a bounded subset $\mathscr B_{1}$ of $\mathcal H^{1}$ and a constant $C>0$ 
such that 
$$
{\rm dist}_{\mathcal H^{\scriptscriptstyle1/3}}(\mathcal U^f(t,\tau)u^\tau,\mathscr B_{1})\leq Ce^{-\gamma (t-\tau)}
$$
for any $u^\tau\in\mathscr B_{\scriptscriptstyle1/3},f\in \overline{B}_F(R_0)$ and $t\geq \tau$.
\end{lemma}
				
\begin{proof}[{\bf Proof}]
We define
$$
z[\cdot]=U(t-\tau)u^\tau,\quad u[\cdot]=\mathcal U^f(\cdot,\tau)u^\tau
$$ 
for every $u^\tau=(u_0,u_1)\in \mathscr B_{\scriptscriptstyle1/3}$ and $f\in \overline{B}_F(R_0)$. Recall that the difference $w=u-z$ solves equation (\ref{w-equation}). Since by Lemma \ref{Linear-semigroup},
\begin{equation}\label{bound-DS-9}
\|z[t]\|_{_{\mathcal H^{\scriptscriptstyle1/3}}}\leq Ce^{-\gamma (t-\tau)}
\end{equation}
for any $t\geq \tau$, it suffices to check that for an appropriate choice of $\mathscr B_1\subset \mathcal H^1$, there holds
\begin{equation}\label{inclusion-1}
w[t]\in \mathscr B_1.
\end{equation}

Let $T_1>0$ be sufficiently large so that 
$\|U(T_1)\|_{_{\mathcal L(\mathcal H)}}\leq \frac{1}{2}.$
Differentiating (\ref{w-equation}) with respect to $t$, one can obtain an equation for $\theta:=\partial_t w$, i.e.,
\begin{equation}\label{theta-equation}
\begin{cases}
\boxempty \theta+a(x)\partial_t \theta+3u^2\partial_t z+3u^2\theta=\partial_tf,\quad x\in D,\\
\theta[\tau]=(0,-u^3_0+f(0)).
\end{cases}
\end{equation}
Then, making use of the formula of variations of constants,
we compute that
$$
\|\theta [\tau+T_1]\|_{_{\mathcal H}}\leq \frac{1}{2}\|\theta [\tau]\|_{_{\mathcal H}}+C\left(1+
\|u^2\theta\|_{_{L^1_tL^2_x}}+\|u^2\partial_tz\|_{_{L^1_tL^2_x}}
\right)
$$
for any $\tau\geq 0$. Let us first observe that
\begin{equation}\label{Stri-1}
\|u^2\theta\|_{_{L^1_tL^2_x}}\leq \|u\|_{_{L^3_tL^6_x}}^2\|\theta\|_{_{L^3_tL^6_x}}\leq C\|\theta\|_{_{L^3_tL^6_x}},
\end{equation}
by applying Corollary \ref{Corollary-S1bound}, where $C=C(\mathscr B_{\scriptscriptstyle1/3},R_0)>0$. This together with the interpolation inequality
\begin{equation*}
\|\theta\|_{_{L^3_tL^6_x}}\leq \|\theta\|^{1/6}_{_{L^1_tL^2_x }}\|\theta\|^{5/6}_{_{L^5_tL^{10}_x}}
\end{equation*}
implies that
\begin{equation*}
\|u^2\theta\|_{_{L^1_tL^2_x}}\leq \varepsilon \|\theta\|_{_{L^5_tL^{10}_x}}+C(\varepsilon)\|\theta\|_{_{L^1_tL^2_x}}
\end{equation*}
with $\varepsilon\in(0,1)$ and $C(\varepsilon)>0$.
At the same time, it follows that
\begin{align}\label{Stri-3}
\|u^2\partial_tz\|_{_{L^1_tL^2_x}}\leq C\int_\tau^{\tau+T_1}
\|u\|_{_{H^{\scriptscriptstyle4/3}}}^2\|\partial_tz\|_{_{H^{\scriptscriptstyle1/3}}}dt 
\leq C.
\end{align}
Here, we have tacitly used Corollary \ref{Corollary-bound}(2) and (\ref{bound-DS-9}). In summary, one has
\begin{equation}\label{AC-estimate-1}
\|\theta [\tau+T_1]\|_{_{\mathcal H}}
\leq \frac{1}{2}\|\theta [\tau]\|_{_{\mathcal H}}+C+\varepsilon\|\theta\|_{_{L^5_tL^{10}_x}}+C(\varepsilon)\|\theta\|_{_{L^1_tL^{2}_x}}. 
\end{equation}

To deal with the term $\|\theta\|_{_{L^5_tL^{10}_x}}$, we apply Proposition \ref{prop-Strichartz} with $(q,r)=(5,10)$, in order to infer that 
\begin{equation*}
\begin{aligned}
\|\theta\|_{_{L^5_tL^{10}_x}}&\leq C\left(\|\theta[\tau]\|_{_{\mathcal H}}+\|-3u^2\theta-3u^2\partial_tz+\partial_tf\|_{_{L^1_tL^2_x}}\right)\\
&\leq C\left(1+\|\theta[\tau]\|_{_{\mathcal H}}+\|\theta\|_{_{L^3_tL^6_x}}
\right)\\
&\leq C\left(1+\|\theta[\tau]\|_{_{\mathcal H}}+\|\theta\|_{_{L^1_tL^2_x}}
\right)+\frac{1}{2}\|\theta\|_{_{L^5_tL^{10}_x}},
\end{aligned}
\end{equation*}
where we have also invoked (\ref{Stri-1})-(\ref{Stri-3}).
Thus, we conclude that
$$
\|\theta\|_{_{L^5_tL^{10}_x}}\leq C \left[1+
\|\theta [\tau]\|_{_{\mathcal H}} +\|\theta\|_{_{L^1_tL^{2}_x}}\right].
$$
Inserted into (\ref{AC-estimate-1}), this means that
\begin{align*}
\|\theta [\tau+T_1]\|_{_{\mathcal H}}&\leq \frac{3}{4}\|\theta [\tau]\|_{_{\mathcal H}}+C\left[1+\|\partial_t(u-z)\|_{_{L^1_tL^2_x}}\right]\\
&\leq \frac{3}{4}\|\theta [\tau]\|_{_{\mathcal H}}+C
\end{align*}
for a sufficiently small $\varepsilon$; here we have used Corollary \ref{Corollary-S1bound} again. Then, in view of (\ref{theta-equation}), it follows that there exists a constant $C=C(\mathscr B_{\scriptscriptstyle1/3},R_0,T_1)>0$ such that 
\begin{equation}\label{estimate-theta}
\|\theta[t]\|_{_{\mathcal H}}\leq C
\end{equation}
for any $u^\tau\in \mathscr B_{\scriptscriptstyle1/3},f\in \overline{B}_F(R_0)$ and $t\geq \tau$. 
					
Finally, since
$$
\partial_t w=\theta,\quad -\Delta w=-\partial_{t}\theta-a(x)\theta-u^3+f,
$$
the desired inclusion (\ref{inclusion-1}) holds for $\mathscr B_1=\overline{B}_{\mathcal H^1}(R)$ 
with a sufficiently large $R>0$, according to  (\ref{estimate-theta}). The proof is then complete.
\end{proof}

To conclude this section, we complete the proof of Theorem \ref{Thm-dynamicalsystem}.

\begin{proof}[{\bf Proof of Theorem \ref{Thm-dynamicalsystem}(2)}]
Let $R_0>0$ be arbitrarily given, and choose $R_1$ in Proposition \ref{prop-S1dissipativity} so that $$\overline{B}_F(R_0)\subset\overline{B}_{L^\infty(\R^+;H^{\scriptscriptstyle1/3})}(R_0)\subset \overline{B}_{L^\infty(\R^+;H)}(R_1).$$ Then, for every $u^\tau\in\mathcal H$, there exists an elapsed time $T$ of the form (\ref{elapsed-time}), such that
\begin{equation}\label{ACproof-1}
\mathcal U^f(\tau+t,\tau)u^\tau\in\mathscr B_0
\end{equation}
for any $f\in \overline{B}_{F}(R_0)$ and $t\geq T,\tau\geq 0$. In addition, let $\mathscr B_{\scriptscriptstyle1/3}$ and $\mathscr B_1$ be the sets established in Corollary \ref{Corollary-bound}(1) and Lemma \ref{Lemma-AC-2}, respectively.

In what follows we assume $\tilde u^\tau\in \mathscr B_0$ and set $t=t_1+t_2$ with $t_i\geq 0$. Then, there exists $\phi\in\mathscr B_{\scriptscriptstyle1/3}$ such that
\begin{equation*}
\|\mathcal U^f(\tau+t_1,\tau)\tilde u^\tau-\phi\|_{_{\mathcal H}}\leq Ce^{-\gamma t_1}.
\end{equation*}
From Proposition \ref{prop-wellposedness}(1), it then follows that there exists a constant $L>0$ such that
\begin{equation*}
\|\mathcal U^f(\tau+t,\tau)\tilde u^\tau-\mathcal U^f(\tau+t,\tau+t_1)\phi\|_{_{\mathcal H}}\leq CLe^{Lt_2}e^{-\gamma t_1}.
\end{equation*}
Furthermore,  there exists $\psi\in\mathscr B_1$ such that
\begin{equation*}
\|\mathcal U^f(\tau+t,\tau+t_1)\phi-\psi\|_{_{\mathcal H}}\leq Ce^{-\gamma t_2}.
\end{equation*}
In summary, 
\begin{equation}\label{ACproof-2}
\|\mathcal U^f(\tau+t,\tau)\tilde u^\tau-\psi\|_{_{\mathcal H}}\leq CLe^{Lt_2}e^{-\gamma t_1}+Ce^{-\gamma t_2}.
\end{equation}
Now, letting 
$$
t_1=(1-\varepsilon)t,\quad t_2=\varepsilon t,\quad \varepsilon\in (0,1)
$$ 
in (\ref{ACproof-2}), it follows that
$$
{\rm dist}_{_{\mathcal H}}(\mathcal U^f(\tau+t,\tau)\tilde u^\tau,\mathscr B_1)\leq Ce^{-[\gamma(1-\varepsilon)-L\varepsilon] t}+Ce^{-\varepsilon\gamma t}.
$$
Taking $\varepsilon$ sufficiently small so that $\gamma(1-\varepsilon)>L\varepsilon$, we conclude that
\begin{equation}\label{ACproof-3}
{\rm dist}_{_{\mathcal H}}(\mathcal U^f(\tau+t,\tau)\tilde u^\tau,\mathscr B_1)\leq Ce^{-\kappa t},
\end{equation}
where we take $\kappa=\min\{ \gamma(1-\varepsilon)-L\varepsilon,\varepsilon\gamma,(4p)^{-1}\}$ with $p$ arising in (\ref{elapsed-time}).

Now, putting (\ref{elapsed-time}),(\ref{ACproof-1}) and (\ref{ACproof-3}) all together, it follows that
\begin{equation*}
{\rm dist}_{_{\mathcal H}}(\mathcal U^f(\tau+t,\tau)u^\tau,\mathscr B_1)\leq C(1+E_u(\tau))e^{-\kappa t}
\end{equation*}
for any $t\geq T$. Finally, the case of $t\in[0,T]$ can be addressed by repeating the deduction as in (\ref{ACproof2}). The proof is then complete.
\end{proof}

\section{Stabilization analysis for the controlled systems}\label{Section-control}

We in this section demonstrate an exact and stronger statement of Theorem \ref{Thm-informal-SQ}, regarding the squeezing property of a controlled system (\ref{controlproblem-0}) and collected as Theorem \ref{Th1} below. The squeezing result constitutes the main ingredient in the verification of coupling hypothesis (see hypothesis $(\mathbf{H})$ in Section \ref{Section-1-1}) for (\ref{Problem-0}); see Section \ref{Section-6-3}. 
The proof of Theorem \ref{Th1} will be based on a contractibility result for the linearized system, which is formulated as Proposition~\ref{Prop-linear} below. Both of these results and outline of proof are included in Section \ref{Section-controlresult}. The details of proof are then provided in Sections~\ref{Section-duality}-\ref{Section-controlstructure}.

\subsection{Results and outline of proof}\label{Section-controlresult}

The system under consideration reads
\begin{equation}\label{controlproblem-0}
\begin{cases}
\boxempty u+a(x)\partial_t u+u^3=h(t,x)+\chi \mathscr P^{\scriptscriptstyle T}_{\scriptscriptstyle N}\zeta(t,x),\quad x\in D,\\
u[0]=(u_0,u_1)=u^0,
\end{cases}
\end{equation}
on time interval $[0,T]$. Here, the parameters $T>0$ and $N\in\N^+$ will be determined later; $h=h(t,x)$ is a given external force, while $\zeta=\zeta(t,x)$ stands for the control; $\mathscr P^{\scriptscriptstyle T}_{\scriptscriptstyle N}$ is the projection in $L^2(D_T)$ onto the finite-dimensional subspace spanned by  $ e_j\alpha^{\scriptscriptstyle T}_k,1\leq j,k\leq N$. 
The functions $a(x),\chi(x)$ are geometrically localized in the sense of  $(\mathbf{S1})$. 

\subsubsection{Statement of main results}
 
Define a mapping by  
\begin{equation*}
\mathcal S\colon\mathcal H\times L^2(D_T)\rightarrow C([0,T];\mathcal H),\quad\mathcal S(u_0,u_1,f)=u[\cdot],
\end{equation*}
where $u\in\mathcal X_T$ stands for the solution of (\ref{semilinear-problem}).
Obviously, system (\ref{controlproblem-0}) is obtained by replacing $f$ with $h+\chi\mathscr P^{\scriptscriptstyle T}_{\scriptscriptstyle N}\zeta$ in (\ref{semilinear-problem}), so that its solutions can also be represented by the mapping $\mathcal{S}$. Recall the set $B_R=B_{C([0,T];H^{\scriptscriptstyle11/7})}(R)$ is defined by (\ref{potential-space}). For every $\varepsilon\in(0,1)$, we take $T'_\varepsilon>0$  to be suitably large so that 
\begin{equation}\label{asymptotic-1}
\|U(t)\|_{_{\mathcal L(\mathcal H)}}\leq \frac{\varepsilon}{2},\quad\forall\, t\geq T'_\varepsilon;
\end{equation}
the existence of such $T'_\varepsilon$ is assured by Lemma \ref{Linear-semigroup}. We further set
\begin{equation}\label{largetime-2}
T''=2\sup_{x\in D}|x-x_1|,\quad T_\varepsilon=\max\{T'_\varepsilon,T''\}, 
\end{equation}
where the point $x_1$ arises in (\ref{Gamma-condition}).

With the above preparations, the main result of this subsection is collected as follows.

\begin{theorem}\label{Th1}
Assume that $a(x),\,\chi(x)$ satisfy setting  $(\mathbf{S1})$. 
Let $\varepsilon\in(0,1)$, $T>T_\varepsilon$ and $R>0$ be arbitrarily given. Then there exist constants  $d=d(\varepsilon,T,R)>0$,  $N=N(\varepsilon,T,R)\in\N^+$ and a mapping 
$
\Phi\colon B_R\rightarrow \mathcal L(\mathcal H;L^2(D_T))
$ such that the following assertions hold.
\begin{enumerate}[leftmargin=2em]
\item[$(1)$] {\rm(}Squeezing{\rm)} Let $\hat u^0\in\mathcal H^{\scriptscriptstyle4/7}$ and $h\in L^2_tH^{\scriptscriptstyle4/7}_x$ such that
$\hat u\in B_R$ with
$
\hat u[\cdot]=\mathcal S(\hat u^0,h).
$
For every $u^0\in \mathcal H$, if $$\|u^0-\hat u^0\|_{_{\mathcal H}}\leq d,$$ there is a control $\zeta\in L^2(D_T)$ such that 
\begin{equation}\label{implication-0}
\|u[T]-\hat u[T]\|_{_{\mathcal H}}\leq \varepsilon\|u^0-\hat u^0\|_{_{\mathcal H}}
\end{equation}
holds, where $u[\cdot]=\mathcal S(u^0,h+\chi \mathscr P^{\scriptscriptstyle T}_{\scriptscriptstyle N}\zeta)$.
			
\item[$(2)$] {\rm(}Structure of control{\rm)} The control $\zeta$ verifying {\rm(\ref{implication-0})} has the form 
$$\zeta=\Phi(\hat u)(u^0-\hat u^0).
$$ 
Moreover, the mapping 
$
\Phi
$
is Lipschitz and continuously differentiable.
\end{enumerate}
\end{theorem}

In the verification of coupling hypothesis for (\ref{Problem-0}) (see Section \ref{Section-6-3}), we shall apply Theorem~\ref{Th1} by taking $R>0$ sufficiently large so that
$$
\left\{ \hat u[\cdot]=\mathcal S(\hat u^0,h);\hat u^0\in\mathcal Y_\infty,h\in\mathcal E\right\}\subset B_R,
$$
where $\mathcal Y_\infty$ is the attainable set from the pathwise attracting set $\mathscr B_{\scriptscriptstyle4/7}$ (see Theorem \ref{Thm-informal-AC} and Theorem \ref{Thm-dynamicalsystem}), and $\mathcal E$ stands for the support of $\mathscr D(\eta_n)$. Then, combined with two classical results for optimal coupling (see Proposition \ref{Coupling lemma 1} and Lemma \ref{Coupling lemma 2}) and
an estimate for the total variation distance (see Lemma \ref{Coupling lemma 3}), the squeezing property established in Theorem \ref{Th1} could yield the coupling condition. In particular, inequality (\ref{implication-0}) leads to the availability of Lemma \ref{Coupling lemma 2}, while the structure of control will be used in the step where Lemma \ref{Coupling lemma 3} comes into play.

The proof of Theorem \ref{Th1} is based on a ``linear test''. That is, it suffices to establish the contractibility for the linearized system along the target solution $\hat u$; the issue of contractibility is the existence and construction of controls such that the states of controlled solutions become ``smaller'' in time $T$. The linearized controlled system under consideration is of the form  
\begin{equation}\label{Problem-linearized}
\begin{cases}
\boxempty v+a(x)\partial_t v+3\hat u^2v=\chi\mathscr P^{\scriptscriptstyle T}_{\scriptscriptstyle N}\zeta(t,x),\quad x\in D,\\
v[0]=(v_0,v_1)=v^0.
\end{cases}
\end{equation}
It is worth mentioning that in the study of the contractibility,
system (\ref{Problem-linearized}) can be considered individually for a general function 
$\hat u\in C([0,T];H^{\scriptscriptstyle11/7}),$
i.e., it need not be an uncontrolled solution of (\ref{controlproblem-0}). 
	
In a slight abuse of the previous notations,
we denote by 
$
v=\mathcal V_{\hat u}(v^0,f)
$
the solution of  (\ref{linear-problem}) with $b,p$ replaced by $a,3\hat u^2$, respectively, where $\hat u\in B_R,v^0\in\mathcal H$ and $f\in L^2(D_T)$. By this setting a solution of (\ref{Problem-linearized}) can be written as $\mathcal V_{\hat u}(v^0,\chi\mathscr P^{\scriptscriptstyle T}_{\scriptscriptstyle N}\zeta)$. 
In the case where the initial condition is replaced with the terminal condition
$
v[T]=(v_0^T,v_1^T)=v^T\in\mathcal H,
$
the corresponding solution is denoted by 
$
v=\mathcal V^T_{\hat u}(v^T,f).
$
	
The contractibility result for system (\ref{Problem-linearized}) is stated as follows.

\begin{proposition}\label{Prop-linear}
Assume that $a(x),\,\chi(x)$ satisfy setting  $(\mathbf{S1})$.  
Let $\varepsilon\in(0,1)$, $T>T_\varepsilon$ and $R>0$ be arbitrarily given. Then there exists a constant $N=N(\varepsilon,T,R)\in\N^+$ and a mapping $
\Phi\colon B_R\rightarrow \mathcal L(\mathcal H;L^2(D_T))
$ such that 
the following assertions hold.
\begin{enumerate}[leftmargin=2em]
\item[$(1)$] {\rm(}Contractibility{\rm)} For every $\hat u\in B_R$ and $v^0\in \mathcal H$, there exists a control $\zeta\in L^2(D_T)$ such that 
\begin{equation}\label{binding-property}
\|v[T]\|_{_{\mathcal H}}\leq \varepsilon\|v^0\|_{_{\mathcal H}},
\end{equation}
where $v=\mathcal V_{\hat u}(v^0,\chi\mathscr P^{\scriptscriptstyle T}_{\scriptscriptstyle N}\zeta)$.

\item[$(2)$] {\rm(}Structure of control{\rm)} The control $\zeta$ verifying {\rm(\ref{binding-property})} has the form
\begin{equation}\label{structure-control}
\zeta=\Phi(\hat u)v^0.
\end{equation}
Moreover, the mapping $\Phi$ is Lipschitz and continuously differentiable.
\end{enumerate}		
\end{proposition}

The proof of Proposition \ref{Prop-linear} constitutes the bulk of this section. See Section \ref{Section-outline} below for an outline of its proof, while the technical details are  contained in Sections \ref{Section-duality}-\ref{Section-controlstructure}.

By using a perturbation argument which is rather standard (see, e.g., \cite{ADS-16,BRS-11}), it can be derived that the control contracting system (\ref{Problem-linearized}) also squeezes (\ref{controlproblem-0}), and then the conclusions of Theorem \ref{Th1} are proved.
The details relevant to the implication ``Proposition~\ref{Prop-linear}~$\Rightarrow$~Theorem~\ref{Th1}'' are left to Appendix \ref{Appendix-squeezing}.

\subsubsection{Outline of proof for Proposition {\rm\ref{Prop-linear}}}\label{Section-outline}

The strategy for constructing the desired controls is the frequency analysis, which has been briefly stated in Section \ref{Section-1-3}. 
More precisely, we split (\ref{Problem-linearized}) into two parts, i.e., a low-frequency system coupled with a high-frequency system.  The controllability is available for the former, while extra dissipation analysis for the latter is established. The contractibility then follows from the results established both for the low- and high-frequency systems. 

Let $\mathbf P_m\ (m\in\N^+)$ denote the projection of $\mathcal H$ onto 
$$
\mathbf H_m:=H_m\times H_m\quad\text{with }H_m={\rm span}\{e_j;1\leq j\leq m\}.
$$
We also introduce the so-called adjoint system of (\ref{Problem-linearized}), reading 
\begin{equation}\label{adjoint-system-1}
\begin{cases}
\boxempty \varphi-a(x)\partial_t\varphi+3\hat u^2\varphi=0,\quad x\in D,\\
\varphi[T]=(\varphi^T_0,\varphi^T_1)=:\varphi^T.
\end{cases}
\end{equation}
In the sequel, our proof of Proposition \ref{Prop-linear} can be summarized as four steps. 

\medskip 

\noindent {\bf Step 1 (low-frequency controllability dual with observability).} We first establish the equivalence of the following two statements. 
\begin{enumerate}[leftmargin=2em]
\item Controllability of (\ref{Problem-linearized}): for every $v^0\in\mathcal H^s\ (s\in(0,1))$, there is a control $\zeta\in L^2_tH^s_x$ such that 
\begin{equation}\label{LF-nullcontrollability}
\mathbf P_mv[T]=0\quad \text{and}\quad  \int_0^T\|\zeta(t)\|_{_{H^s}}^2dt\lesssim \|v^0\|_{_{\mathcal{H}^s}}^2.
\end{equation}

\item Observability of (\ref{adjoint-system-1}): the inequality of type
\begin{equation}\label{OI}
\int_0^T\|\mathscr P^{\scriptscriptstyle T}_{\scriptscriptstyle N}(\chi\varphi)\|_{_{H^{-s}}}^2 \gtrsim \|\varphi^T\|^2_{_{\mathcal H^{-1-s}}},
\end{equation}
is valid for those solutions $\varphi$ whose terminal state has the form $\varphi^T=(q_2,-q_1+aq_2)$ with $ (q_1,q_2)\in \mathbf{H}_m$.
\end{enumerate}
In control theory, such type of equivalence is called ``duality between controllability and observability''; see Coron \cite{Coron-07}. This is in fact an application of a classical result from functional analysis, illustrating the equivalence between the surjective property of a bounded linear operator and the coercivity of its adjoint (see Lemma \ref{Theorem-surjective}). A precise description and demonstration of the equivalence ``$(\ref{LF-nullcontrollability})\Leftrightarrow(\ref{OI})$'' will be found in Section \ref{Section-duality}.

\medskip

\noindent {\bf Step 2 (observability).} The next task is naturally to address the issue of observability inequality (\ref{OI}). In fact, the verification of observability is a complicated part of our duality method. So as not to interrupt the flow of main ideas, the  sketch of proof for observability, divided into Steps 2.1-2.3, will be placed at the end of the outline. The relevant details are contained in Section \ref{Section-observability}.

\medskip 

Once the analysis involved in Step 2 is accomplished, the null controllability in the low frequency, i.e. (\ref{LF-nullcontrollability}),  follows immediately from the duality stated in Step 1. 

\medskip 

\noindent {\bf Step 3 (high-frequency dissipation and contractibility).} With the help of (\ref{LF-nullcontrollability}), the strong dissipation in the high frequency, i.e.,
\begin{equation}\label{HF-dissipation}
\|(I-\mathbf P_m)v[T]\|_{_{\mathcal H}}\leq \frac{\varepsilon}{2} \|v^0\|_{_{\mathcal H}} 
\end{equation}
with an appropriately chosen $m\in\N^+$ (depending on $\varepsilon\in(0,1)$), can be then derived. More precisely, we invoke the method of asymptotic regularity, coming from the theory of dynamical system (see, e.g., \cite{BV-92}). As a consequence, it will be shown that for every $v^0\in\mathcal H$, there is a control $\zeta\in L^2_tH_x^s$ such that
\begin{equation}\label{LF-controllability}
\mathbf P_mv[T]=\mathbf P_mU(T)v^0\quad \text{and}\quad\int_0^T\|\zeta(t)\|_{_{H^s}}^2dt\lesssim \|v^0\|_{_{\mathcal{H}}}^2.
\end{equation}
In particular, the $H^s$-regularity of $\zeta$  would yield the high-frequency dissipation (\ref{HF-dissipation}). Thanks to the decay of $U(t)$ (see Lemma \ref{Linear-semigroup}), the combination of (\ref{HF-dissipation}) and (\ref{LF-controllability})   gives rise to the contractibility (\ref{binding-property}). That is, the first assertion of Proposition \ref{Prop-linear} follows. See Section \ref{Section-Contractibility} for more details of this step.

\medskip 

\noindent {\bf Step 4 (structure of the control).} By now it remains to investigate the structure for the control $\zeta$ verifying (\ref{LF-nullcontrollability}) or (\ref{LF-controllability}), in order to prove the second assertion of Proposition~\ref{Prop-linear}. Roughly speaking, the proof is based on an essential observation: the control $\zeta$ can be constructed as the  minimizer of the functional
$$\tilde\zeta\mapsto\int_0^T\|\tilde\zeta(t)\|^2_{_{H^s}}dt,$$
where $\tilde\zeta\in L^2_tH^s_x$ takes over the set of all controls verifying the equality in (\ref{LF-nullcontrollability}).
Invoking the idea of HUM due to Lions \cite{Lions-88},
such minimality implies that the control $\zeta$ can be expressed via a solution of adjoint system (\ref{adjoint-system-1}), where the terminal state $\varphi^T$ is the unique optimal solution of another minimization problem defined on $\mathbf H_m$. For the problem we encounter here, the main advantage of the finite-dimensional minimization problem is that it can induce a control map, whose dependence on $v_0,v_1,\hat u$ can be further characterized by adapting the argument developed in \cite[Proposition 5.5]{Shi-15}. See Section \ref{Section-controlstructure} for more details.

\medskip 

To complete the outline, let us give a brief sketch of verification for the observability (\ref{OI}), which is the main purpose of Step 2. Our approach involves several various techniques in controllability and observability, including Carleman estimates, regularization analysis of control map, compact-uniqueness argument and truncation technique.

\begin{itemize}[leftmargin=1em]
\item {\bf Step 2.1.} We shall first prove (\ref{OI}) for a special case where $s=0$ and $N=\infty$ (i.e., $\mathscr P^{\scriptscriptstyle T}_{\scriptscriptstyle N}$ becomes the identity):
\begin{equation}\label{basic-ob1}
\int_0^T\|\chi\varphi\|^2 \gtrsim \|\varphi^T\|^2_{_{\mathcal H^{-1}}}.
\end{equation}
To this end, we make use of the Carleman estimates (see, e.g.,  \cite{Zhang-00,Zhang-PRSL}) combined with  energy method involved in Proposition \ref{Energy-estimate}(2). As a by-product, the inequality of type (\ref{basic-ob1}) could imply a full-frequency controllability for (\ref{controlproblem-0}) with $N=\infty$: for every $v^0\in\mathcal H$, there is a unique control $\zeta\in L^2(D_T)$ such that the HUM-based minimality (as stated in Step 4) holds,
\begin{equation*}
v[T]=0\quad \text{and}\quad  \int_0^T\|\zeta(t)\|^2dt\lesssim \|v^0\|_{_{\mathcal{H}}}^2
\end{equation*}
This induces a ``control map'', i.e.  $\Lambda\colon \mathcal H\rightarrow L^2(D_T)$, $\Lambda(v^0)=\zeta$.

\item {\bf Step 2.2.} The next thing to be done is to demonstrate 
\begin{equation}\label{basic-ob2}
\int_0^T\|\chi\varphi\|_{_{H^{-s}}}^2 \gtrsim \|\varphi^T\|^2_{_{\mathcal H^{-1-s}}}\quad \text{with }\hat u\equiv 0,
\end{equation}
where the inequality corresponds to (\ref{OI}) in another special case of $s\in(0,1)$ and $N=\infty$. By using the duality between controllability and observability (see Step 1), the issue of (\ref{basic-ob2}) is converted into a regularization problem of $\Lambda$ (with $\hat u\equiv 0$). More precisely, inequality (\ref{basic-ob2}) will be derived from the following assertion:  when $v^0\in\mathcal H^s$, the resulting control $\Lambda(v^0)$ has an extra regularity in space, i.e. 
\begin{equation}\label{regularity-controlmap}
\Lambda(v^0)\in L^2_tH^s_x\quad \text{with }\int_0^T\|\Lambda(v^0)\|_{_{H^s}}^2dt\lesssim \|v^0\|_{_{\mathcal{H}}}^2\footnote{
Inspired by \cite{EZ-10}, the time-regularity of $\Lambda(v^0)$ can also be improved. Nevertheless, it is not necessary in the analysis of (\ref{basic-ob2}), so we do not illustrate such property in the present paper.}.
\end{equation}
In order to assure (\ref{regularity-controlmap}), we shall adopt the general method developed in \cite{EZ-10}.

\item {\bf Step 2.3.} On the basis of (\ref{basic-ob1}) and (\ref{basic-ob2}), we are able to extend the observability to for a more general case:
\begin{equation}\label{basic-ob3}
\int_0^T\|\chi\varphi\|_{_{H^{-s}}}^2 \gtrsim \|\varphi^T\|^2_{_{\mathcal H^{-1-s}}}\quad \text{with }\hat u\in B_R.
\end{equation}
Evantually, inequality (\ref{basic-ob3}) could imply (\ref{OI}) as desired. The proofs of (\ref{basic-ob3}) and (\ref{OI}) follow the ideas of compactness-uniqueness argument and truncation technique, respectively; both of these arguments are inspired by the analysis in \cite[Section 4]{ADS-16}.
\end{itemize}

\subsection{Low-frequency controllability dual with observability}\label{Section-duality}

The main context of this subsection is to make the analysis in Step 1 of Section \ref{Section-outline} rigorous, establishing the duality between controllability for system (\ref{Problem-linearized}) and observability for system (\ref{adjoint-system-1}). See Proposition \ref{theorem-duality} below.

For the sake of convenience we denote by $$\varphi=\mathcal W^T_{\hat u}(\varphi^T_0,\varphi^T_1)=\mathcal W^T_{\hat u}(\varphi^T)$$ the solution of adjoint system (\ref{adjoint-system-1}). Let us write $u^\bot[t]=(-\partial_t u,u)(t)$ with $u\in C^1([0,T];H^s)$ $(s\in\R)$ for simplicity. We also denote $\mathcal H^s_*=H^{-1-s}\times H^{-s}$ and $\mathcal H_*=\mathcal H^0_*$.

\begin{proposition}\label{theorem-duality}
Let $T,R>0$ and $m,N\in\N^+$ be arbitrarily given\footnote{Although we assume that these parameters are arbitrary here, the verification of observability (\ref{observability-1}) below involves special choices of $T,N$. Roughly speaking, $T$ will be determined by the geometric condition {\rm(\ref{Gamma-condition})} on $\chi$, while $N$ is carefully chosen according to the values of $T,R,m$. See Proposition {\rm\ref{prop-obs}} later for more details.}. Then the following two statements are equivalent for every $\hat u\in B_R$.
\begin{enumerate}[leftmargin=2em]
\item[$(1)$] There exists a constant $C_1>0$ such that for every $v^0\in \mathcal H^{\scriptscriptstyle1/5}$, there exists a control $\zeta\in L^2_tH^{\scriptscriptstyle1/5}_x$ such that 
\begin{equation}\label{Null-controllability-1}
\mathbf P_mv[T]=0\quad {\it and}\quad \int_0^T \|\zeta(t)\|_{_{H^{\scriptscriptstyle1/5}}}^2dt\leq C_1\|v^0\|^2_{_{\mathcal H^{\scriptscriptstyle1/5}}},
\end{equation} 
where $v=\mathcal V_{\hat u}(v^0,\chi\mathscr P^{\scriptscriptstyle T}_{\scriptscriptstyle N}\zeta)$.
		
\item[$(2)$] There exists a constant $C_2>0$ such that 
\begin{equation}\label{observability-1}
\int_0^T\|\mathscr P^{\scriptscriptstyle T}_{\scriptscriptstyle N}(\chi\varphi)(t)\|^2_{_{H^{\scriptscriptstyle-1/5}}}dt\geq C_{2}\|\varphi^T\|^2_{_{\mathcal H^{\scriptscriptstyle-6/5}}}
\end{equation}
for any $(q_1,q_2)\in \mathbf H_m$, where 
$\varphi=\mathcal W^T_{\hat u}(\varphi^T)$ with $\varphi^T=(q_2,-q_1+aq_2)$.
\end{enumerate}
Moreover, if inequality {\rm (\ref{observability-1})} holds, then the constant $C_1$ arising in {\rm(\ref{Null-controllability-1})} can be 
chosen so that it is
expressed in function of $T,R,C_2$.
\end{proposition}

\begin{remark}
Notice that the norms $\|\cdot\|_{_{H^s}}\ (s\in\R)$ are equivalent on the finite-dimensional space $H_m$, which means that the RHS of {\rm(\ref{observability-1})} can be in principle replaced by other norms on $\mathcal H^s$. In particular, our decision to use the $\mathcal H^{\scriptscriptstyle-6/5}$-norm there is to ensure that the relevant constants $C_1,C_2$ are independent of $m,N$ in verifying the observability, which is essential in the proof of  contractibility {\rm(\ref{binding-property})}. See Sections {\rm\ref{Section-observability}} and {\rm\ref{Section-Contractibility}} later.
\end{remark}

The proof of Proposition \ref{theorem-duality} will make use of the following classical result of functional analysis; see, e.g., \cite[Proposition 2.16]{Coron-07}.

\begin{lemma}\label{Theorem-surjective}
Assume that $\mathcal X$ and $\mathcal Y$ are Hilbert spaces and $\mathcal F\in \mathcal L(\mathcal X;\mathcal Y)$. Then $\mathcal F$ is surjective if and only if there exists a constant $C>0$ such that
\begin{equation}\label{surjective-criterion}
	\|\mathcal F^*y\|_{_{\mathcal X}}\geq C\|y\|_{_{\mathcal Y}}
\end{equation}
for any $y\in\mathcal Y$. Moreover, if {\rm(\ref{surjective-criterion})} holds, then 
there exists $\mathcal G\in \mathcal L(\mathcal Y;\mathcal X)$ such that the following assertions hold.
\begin{enumerate}[leftmargin=2em]
\item[$1)$] The operator $\mathcal G$ is a right inverse of $\mathcal F$, satisfying that
\begin{equation}\label{surjective-results}
(\mathcal F\circ \mathcal G)y=y,\quad \|\mathcal Gy\|_{_{\mathcal X}}\leq C^{-1}\|y\|_{_{\mathcal Y}}
\end{equation}
for any $y\in\mathcal Y$.

\item[$2)$] It follows that 
\begin{equation}\label{minimizing-1}
\|\mathcal Gy\|_{_{\mathcal X}}\leq \|x\|_{_{\mathcal X}}
\end{equation}
for any $y\in\mathcal Y$ and $x\in\mathcal F^{-1}(\{y\})$, i.e., $\|\mathcal Gy\|_{_{\mathcal X}}=\inf_{x\in\mathcal F^{-1}(\{y\})}\|x\|_{_{\mathcal X}}$. Moreover, {\rm(\ref{minimizing-1})} holds with equality if and only if $x=\mathcal Gy$.
\end{enumerate}
\end{lemma}

\begin{proof}[{\bf Proof of Proposition \ref{theorem-duality}}] 
Let us define a mapping by
\begin{equation*}
\mathcal F_T\colon L^2_tH^{\scriptscriptstyle1/5}_x\rightarrow \mathbf H_m\subset \mathcal  H^{\scriptscriptstyle1/5},\quad\mathcal F_T(\zeta)=\mathbf P_mv[T],
\end{equation*}
where 
$
v=\mathcal V_{\hat u}(0,0,\chi\mathscr P^{\scriptscriptstyle T}_{\scriptscriptstyle N}\zeta).
$
It is not difficult to check that $\mathcal F_T$ is a bounded linear operator. 
Moreover, we claim that the adjoint  of $\mathcal F_T$ can be represented by
\begin{equation}\label{adjoint-operator}
\mathcal F_T^*\colon\mathbf H_m\subset\mathcal H_*^{\scriptscriptstyle1/5}\rightarrow L^2_tH^{\scriptscriptstyle-1/5}_x,\quad\mathcal F_T^*(q)=\mathscr P^{\scriptscriptstyle T}_{\scriptscriptstyle N}(\chi \varphi)
\end{equation}
for every $q=(q_1,q_2)\in \mathbf H_m$\footnote{Notice that if $H_m$ is endowed with the norm on $H^s\ (s>0)$, its dual space is isometrically isomorphic to $H_m$ endowed with the norm on $H^{-s}$. Accordingly, we identify the dual space of $\mathbf H_m$, endowed with the $\mathcal H^s$-norm, as $\mathbf H_m$ endowed with the $\mathcal H^s_*$-norm.}, where 
$
\varphi=\mathcal W^T_{\hat u}(q_2,-q_1+aq_2).
$
To demonstrate this, notice first that
\begin{equation}\label{dual-1}
\begin{array}{ll}
\displaystyle\langle
\mathcal F_T(\zeta), q
\rangle_{_{\mathcal H^{\scriptscriptstyle1/5},\mathcal H_*^{\scriptscriptstyle1/5}}}=
(
v[T],\tilde\varphi^\bot[T]
)_{_{H\times H}},
\end{array}
\end{equation}
where $\tilde\varphi=\mathcal W_{\hat u}^T(q_2,-q_1).$ 
We then derive that
\begin{equation}\label{dual-2}
\begin{aligned}
(
v[T], \tilde\varphi^\bot[T]
)_{_{H\times H}}
&=\int_0^T
\left(
\left(
\begin{matrix}
\partial_tv\\
\Delta v-a\partial_tv-3\hat u^2v+\chi\mathscr P^{\scriptscriptstyle T}_{\scriptscriptstyle N}\zeta
\end{matrix}
\right)
,\left(
\begin{matrix}
-\partial_t\tilde\varphi\\
\tilde\varphi
\end{matrix}
\right)
\right)_{_{H\times H}}dt\\
&\quad\quad+\int_0^T
\left(
\left(
\begin{matrix}
	v\\
	\partial_tv
\end{matrix}
\right)
,\left(
\begin{matrix}
	-\Delta\tilde\varphi -a\partial_t\tilde\varphi+3\hat u^2\tilde\varphi
	\\
	\partial_t\tilde\varphi
\end{matrix}
\right)
\right)_{_{H\times H}}dt.
\end{aligned}
\end{equation}
It can be seen that
\begin{equation}\label{identity-6}
{\rm RHS\ of\ }(\ref{dual-2})= -(av(T),q_2)+\int_0^T(\chi\mathscr P^{\scriptscriptstyle T}_{\scriptscriptstyle N}\zeta,\tilde\varphi)dt. 
\end{equation}
Notice that 
$$
(av(T),q_2)=\langle v[T], \hat \varphi^\bot[T]\rangle_{_{\mathcal H^{\scriptscriptstyle1/5},\mathcal H_*^{\scriptscriptstyle1/5}}}=\int_0^T(\chi\mathscr P^{\scriptscriptstyle T}_{\scriptscriptstyle N}\zeta,\hat\varphi)dt,
$$
where $\hat\varphi=\mathcal W_{\hat u}(0,-aq_2)$, by repeating the deduction presented in (\ref{dual-2}),(\ref{identity-6}).
Inserting this into (\ref{identity-6}) and using  (\ref{dual-1}), it follows that
\begin{equation}\label{dual-3}
\begin{array}{ll}
\displaystyle\langle
\mathcal F_T(\zeta), q
\rangle_{_{\mathcal H^{\scriptscriptstyle1/5},\mathcal H_*^{\scriptscriptstyle1/5}}}=\int_0^T(\zeta,\mathscr P^{\scriptscriptstyle T}_{\scriptscriptstyle N}(\chi\varphi))dt,
\end{array}
\end{equation}
where $\varphi=\tilde\varphi-\hat \varphi=\mathcal W^T_{\hat u}(q_2,-q_1+aq_2)$. Here we have also noticed that the operator $\mathscr P^{\scriptscriptstyle T}_{\scriptscriptstyle N}$ is self-adjoint on $L^2(D_T)$.
Then, taking into account
$$
{\rm LHS\ of\ } (\ref{dual-3})=\int_0^T(\zeta,\mathcal F_T^*(q))dt,
$$ 
the desired claim (\ref{adjoint-operator}) is proved.

Thanks to Lemma \ref{Theorem-surjective}, the statement
\begin{enumerate}[leftmargin=2em]
\item[(3)] The mapping $\mathcal F_T$ is surjective.
\end{enumerate}
holds if and only if there exists a constant $C>0$ such that
\begin{equation}\label{observability-2}
\int_0^T\|\mathcal F_T^*(q)(t)\|^2_{_{H^{\scriptscriptstyle-1/5}}}dt
\geq C\|q\|^2_{_{\mathcal H_*^{\scriptscriptstyle1/5}}}
\end{equation}
for any $q=(q_1,q_2)\in\mathbf H_m$. It can be derived that statement (3) is equivalent to (2). Indeed, for every $s\in[0,1/5]$
there exist constants $c_1,c_2>0$ such that
\begin{equation*}
c_1\|p\|^2_{_{\mathcal H^s_*}}\leq \|(p_2,-p_1+ap_2)\|_{_{\mathcal  H^{-1-s}}}^2\leq c_2\|p\|^2_{_{\mathcal H^s_*}}
\end{equation*}
for any $p=(p_1,p_2)\in \mathcal H_*^s$.
Then, letting $s=1/5$ and $p=q$ and recalling  (\ref{adjoint-operator}), inequalities (\ref{observability-1}) and (\ref{observability-2}) are equivalent. This implies immediately the equivalence of statements (2) and (3).

It remains to show that the statements (1) and (3) are equivalent. To this end, assume for the moment that (3) holds. Then, applying Lemma \ref{Theorem-surjective} again yields that there exists $\mathcal G\in\mathcal L(\mathbf H_m;L^2_tH^{\scriptscriptstyle1/5}_x)$ such that
\begin{equation}\label{surjective-results-1}
\mathcal F_T(\zeta)=v^T,\quad 
\int_0^T\|\zeta(t)\|_{_{H^{\scriptscriptstyle1/5}}}^2dt\leq C\|v^T\|^2_{_{\mathcal H^{\scriptscriptstyle1/5}}}
\end{equation}
for every $v^T=(v_0^T,v_1^T)\in \mathbf H_m$, where $\zeta=\mathcal G(v^T)$. We then define 
\begin{equation}\label{problem-surjective}
\tilde v=\mathcal V_{\hat u}(0,0,\chi\mathscr P^{\scriptscriptstyle T}_{\scriptscriptstyle N}\zeta).
\end{equation}
In view of the construction of $\zeta$, it follows that
\begin{equation}\label{terminal-2}
\mathbf P_m\tilde v[T]=v^T.
\end{equation}
At the same time, let 
$
\hat v=\mathcal V_{\hat u}(v^0,0)
$
with a state $v^0\in\mathcal H$ to be controlled. 
One can in the sequel obtain that the sum $v:=\tilde v+\hat v$ verifies $v=\mathcal V_{\hat u}(v^0,\chi\mathscr P^{\scriptscriptstyle T}_{\scriptscriptstyle N}\zeta)$ and
\begin{equation*}
\mathbf P_mv[T]=v^T+\mathbf P_m\hat v[T].
\end{equation*}

Accordingly, in order to construct a control steering system (\ref{Problem-linearized}) from $v^0$ to $0$ in $\mathbf H_m$, it suffices to take
$$
v^T=-\mathbf P_m\hat v[T]
$$
in (\ref{surjective-results-1}). With this setting, the first property described in (\ref{Null-controllability-1}) is clearly obtained, while the combination of (\ref{energy-1}) with (\ref{surjective-results-1}) leads to the second. Statement (1) thus follows.

To show the converse implication $(1)\Rightarrow(3)$, we define 
$$
w=\mathcal V^T_{\hat u}(-v^T,0)
$$
for an arbitrarily given $v^T\in \mathbf H_m$.
Then, we use the property described in statement (1) with
$
v^0=w[0];
$
the resulting control and controlled solution are still denoted by $\zeta$ and $v$, respectively. As a consequence, the difference $\tilde v:=v-w$ satisfies (\ref{problem-surjective}) and (\ref{terminal-2}), where we have also used the estimate of type (\ref{energy-1}) for $\mathcal V^T_{\hat u}$. This implies that $\mathcal F_T(\zeta)=v^T$, as desired.

Finally, the characterization of the constant $C_1$ in (\ref{Null-controllability-1}) can be achieved by following the flow of statements $(2)\Rightarrow (3)\Rightarrow (1)$ among the above arguments, as well as noticing the bound in (\ref{surjective-results}) for the right inverse $\mathcal G$. The proof of Proposition \ref{theorem-duality} is then complete.
\end{proof}

Taking (\ref{minimizing-1}) into account, one can observe that if inequality  (\ref{observability-1}) holds, the control $\zeta$ established in  (\ref{Null-controllability-1}) is in fact constructed as the minimizer of the functional $$\tilde\zeta\mapsto \int_0^T\|\tilde\zeta(t)\|_{_{H^{\scriptscriptstyle1/5}}}^2dt$$ over the set of controls steering system (\ref{Problem-linearized}) to the origin in $\mathbf H_m$. That is, 
if $\tilde\zeta\in L^2_tH^{\scriptscriptstyle1/5}_x$ satisfies that $\mathbf P_mv[T]=0$ with $v=\mathcal V_{\hat u}(v^0,\chi\mathscr P^{\scriptscriptstyle T}_{\scriptscriptstyle N}\zeta)$, then 
\begin{equation}\label{minimizer-2}
\int_0^T \|\zeta(t)\|_{_{H^{\scriptscriptstyle1/5}}}^2dt\leq\int_0^T \|\tilde\zeta(t)\|_{_{H^{\scriptscriptstyle1/5}}}^2dt,
\end{equation}
with equality if and only if $\tilde\zeta=\zeta$.

\begin{remark}\label{Case-N-infinity}
There is a full-frequency version of Proposition {\rm\ref{theorem-duality}}. More precisely, the following two statements are equivalent for every $\hat u\in B_R$.
\begin{enumerate}[leftmargin=2em]
\item[$(1)$] There exists a constant $C_1>0$ such that for every $v^0\in \mathcal H^{\scriptscriptstyle1/5}$
there is a control $\zeta\in L^2_tH^{\scriptscriptstyle1/5}_x$ satisfying 
\begin{equation}\label{Null-controllability-3}
v[T]=0\quad\text{and}\quad \int_0^T \|\zeta(t)\|_{_{H^{\scriptscriptstyle1/5}}}^2dt\leq C_1\|v^0\|^2_{_{\mathcal H^{\scriptscriptstyle1/5}}},
\end{equation} 
where $v=\mathcal V_{\hat u}(v^0,\chi\zeta)$. 
		
\item[$(2)$] There exists a constant $C_2>0$ such that
\begin{equation}\label{observability-3}
\int_0^T\|\chi\varphi(t)\|^2_{_{H^{\scriptscriptstyle-1/5}}}dt\geq C_{2}\|\varphi^T\|^2_{_{\mathcal  H^{\scriptscriptstyle-6/5}}}
\end{equation}
for any $\varphi^T\in \mathcal H^{\scriptscriptstyle-6/5}$, where $\varphi=\mathcal W^T_{\hat u}(\varphi^T)$.
\end{enumerate}
Moreover, if inequality {\rm (\ref{observability-3})} holds, then the constant $C_1$ arising in {\rm(\ref{Null-controllability-3})} can be chosen so that it is expressed in function of $T,R,C_{2}$. These above can be proved by repeating the proof of Proposition {\rm\ref{theorem-duality}} step by step, except that the parameters $m,N$ are taken to be ``infinity'', i.e. the projections $\mathbf P_m$ and $\mathscr P^{\scriptscriptstyle T}_{\scriptscriptstyle N}$ become the identity operator $I$.
\end{remark}

\subsection{Verification of the observability}\label{Section-observability}

The main result of this subsection is contained in the following proposition, providing a precise version of the observability (\ref{OI}). The proof of this proposition follows the procedure as described in Step 2 of Section \ref{Section-outline}. 

Recall the constant $T''$  established in (\ref{largetime-2}). 

\begin{proposition}\label{prop-obs}
Let $T>T''$ and $R>0$ be arbitrarily given. Then the following assertions hold.
\begin{enumerate}
[leftmargin=2em]
\item[$(1)$] There exists a constant $C_0=C_0(T,R)>0$ such that 
\begin{equation}\label{full-observability}
\int_0^T\|\chi\varphi(t)\|^2_{_{H^{\scriptscriptstyle-1/5}}}dt\geq C_0\|\varphi^T\|^2_{_{\mathcal H^{\scriptscriptstyle-6/5}}}.
\end{equation}
for any $\hat u\in B_R$ and $\varphi^T\in \mathcal H^{\scriptscriptstyle-6/5}$, where $\varphi=\mathcal W^T_{\hat u}(\varphi^T)$.

\item[$(2)$] For every $m\in \N^+$, there exists an integer $N=N(T,R,m)\in\N^+$ such that 
\begin{equation}\label{truncated-observability}
\int_0^T\|\mathscr P^{\scriptscriptstyle T}_{\scriptscriptstyle N}(\chi\varphi)(t)\|^2_{_{H^{\scriptscriptstyle-1/5}}}dt\geq
\frac{C_0}{4}\|\varphi^T\|^2_{_{\mathcal H^{\scriptscriptstyle-6/5}}},
\end{equation}
for any $\hat u\in B_R$ and $(q_1,q_2)\in \mathbf H_m$, where $\varphi=\mathcal W_{\hat u}^T(\varphi^T)$ with $\varphi^T=(q_2,-q_1+aq_2)$. 
\end{enumerate}
\end{proposition}

Taking the first assertion of Proposition \ref{prop-obs} for granted, we prove in what follows the second, regarding the ``truncated'' observability inequality (see Step 2.3 in Section \ref{Section-outline}).

\begin{proof}[{\bf Proof of Proposition \ref{prop-obs}(2)}]
We first claim that for an arbitrarily given $m\in\N^+$, there exists a constant $C_m>0$ (depending also on $T,R$) such that
\begin{equation}\label{bound-29}
\|\chi\varphi\|^2_{_{H^1(D_T)}}\leq C_m\int_0^T\|\chi\varphi\|^2_{_{H^{\scriptscriptstyle-1/5}}}dt
\end{equation}
for any $\hat u\in B_R$ and $(q_1,q_2)\in \mathbf H_m$, where $\varphi=\mathcal W_{\hat u}^T(\varphi^T)$ with $\varphi^T=(q_2,-q_1+aq_2)$. This can be proved by noticing, in view of (\ref{energy-1}) and (\ref{full-observability}), that 
\begin{equation*}
\|\chi\varphi\|^2_{_{H^1(D_T)}}\leq K_1\|\varphi^T\|_{_{\mathcal H}}^2\leq K_2\|\varphi^T\|_{_{\mathcal H^{\scriptscriptstyle-{6/5}}}}^2\leq K_3\int_0^T\|\chi\varphi\|^2_{_{H^{\scriptscriptstyle-1/5}}}dt,
\end{equation*}
where the constants $K_i>0$ do not depend on $\hat u,q_1,q_2$. 
At the same time, notice that there exists a sequence $\{\mu_N;N\in\N^+\}$ such that $\mu_N\rightarrow 0^+$ and
$$
\int_0^T\|(I-\mathscr P^{\scriptscriptstyle T}_{\scriptscriptstyle N})f\|_{_{H^{\scriptscriptstyle-1/5}}}^2dt\leq \mu_N\|f\|_{_{H^1(D_T)}}^2
$$
for any $f\in H^1(D_T)$. This together with (\ref{bound-29}) yields that
\begin{equation*}
\int_0^T\|\chi \varphi(t)\|^2_{_{H^{\scriptscriptstyle-1/5}}}dt
\leq 2\int_0^T\|\mathscr P^{\scriptscriptstyle T}_{\scriptscriptstyle N}(\chi \varphi)(t)\|^2_{_{H^{\scriptscriptstyle-1/5}}}dt+2C_m\mu_N\int_0^T\|\chi\varphi\|^2_{_{H^{\scriptscriptstyle-1/5}}}dt.
\end{equation*}
Therefore, one can choose $N=N(C_m)\geq 1$ sufficiently large so that $C_m\mu_N\leq 1/4$, and hence
\begin{equation}\label{bound-30}
\int_0^T\|\chi \varphi(t)\|^2_{_{H^{\scriptscriptstyle-1/5}}}dt\leq 4\int_0^T\|\mathscr P^{\scriptscriptstyle T}_{\scriptscriptstyle N}(\chi \varphi)(t)\|^2_{_{H^{\scriptscriptstyle-1/5}}}dt.
\end{equation}
Finally, inequality (\ref{truncated-observability}) follows from (\ref{full-observability}) and (\ref{bound-30}).
\end{proof}

Based on the above analysis, it remains to establish inequality (\ref{full-observability}) for the first part of Proposition \ref{prop-obs}, i.e., the ``full'' observability inequality. Its proof is based on the following intermediate result, which provides the precise statements for  (\ref{basic-ob1}) and (\ref{basic-ob2}) (see Steps 2.1 and 2.2 in Section \ref{Section-outline}), respectively.

\begin{lemma}\label{basic-observability}
Let $T>T''$ be arbitrarily given. Then the following assertions hold.
\begin{enumerate}[leftmargin=2em]
\item[$(1)$] For every $R>0$, there exists a constant $C=C(T,R)>0$ such that 
\begin{equation}\label{basic-case-1}
\int_0^T\|\chi\varphi(t)\|^2dt\geq C\|\varphi^T\|^2_{_{\mathcal H^{-1}}}
\end{equation}
for any $\hat u\in B_R$ and $\varphi^T\in \mathcal H^{-1}$, where $\varphi=\mathcal W^T_{\hat u}(\varphi^T)$. 

\item[$(2)$] There exists a constant $C=C(T)>0$ such that 
\begin{equation}\label{basic-case-2}
\int_0^T\|\chi\varphi(t)\|^2_{_{H^{\scriptscriptstyle-1/5}}}dt\geq C\|\varphi^T\|^2_{_{\mathcal  H^{\scriptscriptstyle-6/5}}}
\end{equation}
for any $\varphi^T\in \mathcal  H^{\scriptscriptstyle-6/5}$, where $\varphi=\mathcal W^T_{0}(\varphi^T)$.
\end{enumerate}
\end{lemma}

Inequalities (\ref{basic-case-1}) and (\ref{basic-case-2}) are well-understood when $a(x)\equiv 0$. In such case, inequality (\ref{basic-case-1}) can be found in \cite{Zhang-PRSL} (see also \cite{Zhang-00} for the case of boundary control), while the reader is referred to \cite{DL-09} for (\ref{basic-case-2}). On the other hand, we are not able to find an accurate proof in the literature dealing with the space-dependent coefficient $a(x)$. 
Though it is  believed that the presence of $a(x)$ could not lead to essential obstacles, the space dependence of coefficient
would cause some technical complications. 
So, for the reader’s convenience, we provide a sketch of proof for Lemma \ref{basic-observability} below.

\begin{proof}[{\bf Sketch of proof for (\ref{basic-case-1})}]
Let us introduce some notations that will be useful later:
\begin{equation*}
\begin{aligned}
&\mathcal Q=(0,T)\times (0,T)\times D,\\
&T_i=\tfrac{T}{2}-\varepsilon_iT,\quad T_i'=\tfrac{T}{2}+\varepsilon_iT,\\
&\mathcal Q_i=(T_i,T_i')\times (T_i,T_i')\times D,
\end{aligned}
\end{equation*}
where $i=0,1,2$ and 
$
0<\varepsilon_0<\varepsilon_1<\varepsilon_2<\frac{1}{2}
$
to be determined below. Recall the point $x_1\in\R^3\setminus \overline{D}$ established in (\ref{Gamma-condition}).
With $R_0:=\inf_{x\in D}|x-x_1|$ and $R_1:=\sup_{x\in D}|x-x_1|$, let $\alpha\in(0,1)$ be sufficiently close to $1$ so that $R_1^2<\frac{\alpha T^2}{4}$ (in view of $T>T''$).
We then introduce a real function 
$$
\psi(t,s,x)=\frac{1}{2}\left[
|x-x_1|^2-\alpha\left(t-\frac{T}{2}\right)^2-\alpha\left(s-\frac{T}{2}\right)^2
\right]
$$
and define the sets
$$
\Lambda_j=\left\{
(t,s,x)\in \mathcal Q;2\psi(t,s,x)\geq \frac{R_0^2}{j+2}
\right\},\quad j=0,1.
$$
Then, choose $\varepsilon_1$ close to $1/2$ so that $\psi(t,s,x)<0$ for all $(t,s,x)\notin \mathcal Q_1$ and $\Lambda_1\subset \mathcal Q_1$.
At the same time, since $(T/2,T/2,x)\in \Lambda_0$ for every $x\in D$, there holds $\mathcal Q_0\subset \Lambda_0$ for a sufficiently small $\varepsilon_0\in(0,\varepsilon_1)$. Finally, let $\varepsilon_2\in(\varepsilon_1,1/2)$ be arbitrarily given. Summarizing the above, we can conclude the following hierarchy:
\begin{equation}\label{inclusions}
\mathcal Q_0\subset\Lambda_0\subset \Lambda_1\subset \mathcal Q_1\subset \mathcal Q_2.
\end{equation}

We consider a more regular quantity $z$ in the scale of $\mathcal H$: 
\begin{equation*}
z(t,s,x)=\int_s^t\varphi(\xi,x)d\xi.
\end{equation*}
The desired estimate for $\varphi$ is then obtained by using some useful estimates for $z$. Notice that the function $z$
verifies the equation
\begin{equation}\label{z-equation-0}
\partial^2_{tt}z+\partial^2_{ss}z-\Delta z=a(x)(\partial_tz+\partial_sz)-3\int_s^t\hat u^2(\xi,x)\partial_tz(\xi,s,x)d\xi.
\end{equation}
We also use a function $\theta=e^{\lambda \psi}$ with $\lambda >1$. Our goal is to derive that 
\begin{equation}\label{Carleman-1}
\begin{aligned}
\int_{\mathcal Q_0}
(\partial_tz)^2+(\partial_sz)^2
&\lesssim e^{-\lambda }\int_{\mathcal Q}
z^2+(\partial_tz)^2+(\partial_sz)^2\\
&\quad+\int_{0}^{T}\int_{0}^{T}\int_{N_{\delta'}(x_1)}\left[z^2+
(\partial_tz)^2+(\partial_sz)^2
\right]
\end{aligned}
\end{equation}
for a sufficiently large $\lambda$, after some calculations for the weighted function $\theta z$. Inequality (\ref{Carleman-1}) is in fact the Carleman-type estimate, where the set $N_{\delta'}(x_1)$ arises in condition (\ref{Gamma-condition}).

To establish (\ref{Carleman-1}), we make use of \cite[Lemma 2.7]{Zhang-00} (together with some fundamental calculations) to deduce that 
\begin{align}
 \int_{\mathcal Q_1}\theta^2\left[
(\partial_tz)^2+(\partial_sz)^2+|\nabla z|^2
\right] 
&\, \lesssim \lambda^{-1}\int_{\mathcal Q}\theta^2\left[
(\partial_tz)^2+(\partial_sz)^2
\right] \notag \\
& \quad +\lambda^p\int_{\mathcal Q_2}\left[z^2+(\partial_tz)^2+(\partial_sz)^2+|\nabla z|^2 \right] \label{Carleman-2}\\
& \quad +e^{C\lambda}\int_{T_2}^{T_2'}\int_{T_2}^{T_2'}\int_{\Gamma(x_1)}\left|\frac{\partial z}{\partial n}\right|^2,\notag
\end{align}
where $p\in\N^+$ is an absolute constant. We estimate each integral in the RHS as follows:
\begin{enumerate}[leftmargin=2em]
\item Split the first integral in the form
$\int_{\mathcal Q}=\int_{\Lambda_1}+\int_{\mathcal Q\setminus\Lambda_1}$. Then, by the definition of $\Lambda_1$ it follows that
$$
\lambda^{-1}\int_{\mathcal Q\setminus\Lambda_1}\theta^2\left[
(\partial_tz)^2+(\partial_sz)^2
\right]\leq \lambda^{-1}e^{\lambda R_0^2/3}\int_{\mathcal Q}\left[(\partial_tz)^2+(\partial_sz)^2\right],
$$
while the integral on $\Lambda_1$ can be absorbed by the LHS for sufficiently large $\lambda$.

\item The key for dealing with the second integral is to eliminate  $\int_{\mathcal Q_2}|\nabla z|^2$. Roughly speaking, we multiply equation (\ref{z-equation-0}) by $\zeta z$ with $\zeta(t,s)=t(T-t)s(T-s)$, in order to see that
$$
\int_{\mathcal Q_2}|\nabla z|^2\lesssim\int_{\mathcal Q}\zeta|\nabla z|^2\lesssim  \int_{\mathcal Q}\left[
z^2+(\partial_tz)^2+(\partial_sz)^2
\right],
$$
as desired.

\item The integral on $\Gamma(x_1)$ would be bounded by an integral on the neighborhood $N_\delta(x_1)$, i.e.,
$$
\int_{T_2}^{T_2'}\int_{T_2}^{T_2'}\int_{\Gamma(x_1)}\left|\frac{\partial z}{\partial n}\right|^2\lesssim \int_{0}^{T}\int_{0}^{T}\int_{N_{\delta'}(x_1)}\left[z^2+
(\partial_tz)^2+(\partial_sz)^2
\right].
$$
This will be done by means of the well-known multiplier technique; see, e.g., \cite[Chapter VII]{Lions-88} (and also \cite{Zhang-PRSL}).
\end{enumerate}
Thus, inequality  (\ref{Carleman-1}) follows since
\begin{equation*}
{\rm LHS\ of\ }(\ref{Carleman-2})\geq e^{\lambda R^2_0/2}\int_{\mathcal Q_0}\left[(\partial_tz)^2+(\partial_sz)^2
\right],
\end{equation*}
where we also use (\ref{inclusions}) and the fact $2\psi(t,s,x)\geq R^2_0/2$ for all $(t,s,x)\in\Lambda_0$.

Together with condition (\ref{Gamma-condition}),
inequality (\ref{Carleman-1}) leads to the following estimate for $\varphi$:
\begin{equation}\label{Carleman-3}
\int_{T_0}^{T_0'}\|\varphi(t)\|^2dt\lesssim e^{-\lambda R_0^2/8}\int_0^T\|\varphi(t)\|^2 dt +e^{C\lambda}\int_{0}^{T}\|\chi\varphi(t)\|^2 dt
\end{equation}
with a sufficiently large $\lambda$. Finally, the observability (\ref{basic-case-1}) can be proved by combining (\ref{Carleman-3}), the energy estimate (\ref{energy-2}), and the fact 
\begin{equation*}
\int_{S_0}^{S_0'}\|\partial_t\varphi(t)\|^2_{_{H^{-1}}}dt\lesssim\int_{T_0}^{T_0'}\|\varphi(t)\|^2dt,\quad \forall\, S_0\in(T_0,T/2),S_0'\in(T/2,T_0'),
\end{equation*}
whose verfication follows the same idea as in \cite[Lemma 3.4]{Zhang-PRSL}. 
\end{proof}

\begin{remark}\label{remark-general potential}
By analyzing the above sketch, one can notice that the proof of {\rm(\ref{basic-case-1})}
does involve the $L^\infty$-norm of $\hat u$ rather than its $H^{\scriptscriptstyle11/7}$-norm. As a consequence, 
inequality {\rm(\ref{basic-case-1})} remains true in the case where the potential term $3\hat u^2\varphi$ is replaced by $p\varphi$ with $p\in L^\infty(D_T)$. In addition,
the uniformity of the constant $C$ therein is valid for $\|p\|_{_{L^\infty(D_T)}}\leq R$.
\end{remark}

\begin{proof}[{\bf Sketch of proof for (\ref{basic-case-2})}]
Thanks to the duality between controllability and observability (see  Remark \ref{Case-N-infinity}), it suffices to show that for every $v^0\in\mathcal H^{\scriptscriptstyle1/5}$, there is a control $\zeta\in L^2_tH^{\scriptscriptstyle1/5}_x$ satisfying
\begin{equation}\label{regular-controlmap}
v[T]=0\quad \text{and} \quad\int_0^T\|\zeta(t)\|^2_{_{H^{\scriptscriptstyle1/5}}}dt\lesssim \|(v_0,v_1)\|^2_{_{\mathcal H^{\scriptscriptstyle1/5}}},
\end{equation}
where $v=\mathcal V_0(v^0,\chi\zeta)$.

Let the operators
$
A\colon D(A)\subset \mathcal H\rightarrow\mathcal H$ and $B\colon H\rightarrow \mathcal H
$
defined as
$$
A=\left(
\begin{matrix}
0 & -1 \\
-\Delta & a(x)
\end{matrix}
\right),\quad D(A)=\mathcal H^1,\quad
Bf=\left(
\begin{matrix}
	0 \\
	\chi f
\end{matrix}
\right).
$$
Note that the infinitesimal generator of $U(t)$ is in fact $-A$. In addition, the adjoint operators of $A,B$ are
$$
A^*=\left(
\begin{matrix}
	0 & 1 \\
	\Delta & a(x)
\end{matrix}
\right),\quad D(A^*)=D(A),\quad
B^*\left(
\begin{matrix}
	f_0 \\
	f_1
\end{matrix}
\right)
=\chi f_1.
$$
The adjoint $U^*(t)$ of $U(t)$ is generated by $-A^*$.
With the above setting, the controlled system under consideration can be rewritten as
\begin{equation*}
\frac{dy}{dt}+Ay(t)=B\zeta(t),\quad 
y(0)=y_0,
\end{equation*}
while its adjoint system is of the form
\begin{equation}\label{Abstract-system-2}
\frac{dq}{dt}=A^*q(t), \quad q(T)=q^T.
\end{equation}

Thanks to the $\mathcal H^{-1}$-observability (\ref{basic-case-1}) (with $\hat u\equiv 0$), one can use the same argument as in \cite[Chapter 1.4]{Coron-07} for the construction of an HUM map. More precisely, for every $y^{T}\in \mathcal H$, there exists a unique $q^{T}\in \mathcal H$ such that the solution $y\in C([0,T];\mathcal H)$ of system 
\begin{equation}\label{Abstract-system-4}
\frac{dy}{dt}+Ay(t)=BB^*q(t),\quad y(0)=0
\end{equation}
verifies $y(T)=y^{T}$, where $q\in C([0,T];\mathcal H)$ is the solution of (\ref{Abstract-system-2}). This defines a control map 
\begin{equation*}
\Lambda\colon\mathcal H\rightarrow \mathcal H,\quad \Lambda(y^{T})=q^{T}.
\end{equation*}
It is rather standard to verify that $\Lambda\in\mathcal L(\mathcal H)$, while much of efforts will be in ensuring that $\Lambda\in\mathcal L(\mathcal H^1)$ (and hence $\Lambda\in\mathcal L(\mathcal H^s),\, s\in(0,1)$ by interpolation).

To this end, we shall make use of the dual identity
\begin{equation*}
(y^{T},\hat q^{T})_{_{\mathcal H}}= \int_0^{T}(B^*q(t),B^*\hat q(t))dt
\end{equation*}
for any $y^T,\hat q^T\in\mathcal H$,
where $y\in C([0,T];\mathcal H)$ is the solution of (\ref{Abstract-system-4}) with $q^{T}=\Lambda(y^{T})$, and $\hat q\in C([0,T];\mathcal H)$ stands for the solution of adjoint system with $\hat q(T)=\hat q^{T}$. To proceed further, the element $\hat q^{T}$ will be taken as 
$$
\hat q^{T}=\frac{q(T+\sigma)-2\Lambda(y^{T})+q(T-\sigma)}{\sigma^2},\quad \sigma\in(0,1);
$$
notice that $q(T-\sigma)=U^*(\sigma)\Lambda(y^{T})$ and $\Lambda(y^{T})=U^*(\sigma)q(T+\sigma)$.
Hence, it can be checked that 
$$
\hat q(t)=\frac{q(t+\sigma)-2q(t)+q(t-\sigma)}{\sigma^2}.
$$
With this setting, an application of dual identity, together with the observability (\ref{basic-case-1}), gives rise to 
$$
\left\|
\frac{U^*(-\sigma)\Lambda(y^{T})-\Lambda(y^{T})}{\sigma}
\right\|_{_{\mathcal H}}^2\lesssim\|y^{T}\|^2_{_{\mathcal H^1}},
$$
provided that $y^T\in\mathcal H^1$. This implies $\Lambda(y^T)\in D(A^*)=\mathcal H^1$ and $\|\Lambda(y^T)\|_{_{\mathcal H^1}}\lesssim \|y^{T}\|^2_{_{\mathcal H^1}}$. In conclusion, the $\mathcal H^{\scriptscriptstyle1/5}$-controllability (\ref{regular-controlmap}) is obtained; in fact, the relevant control $\zeta$ is constructed via 
$
\zeta(t)=\chi\partial_t\varphi,
$
where $q=(\varphi,\partial_t\varphi)\in C([0,T];\mathcal H)$ is the solution of (\ref{Abstract-system-2}) with 
$
q^{T}=-\Lambda(U(T)v^0).
$
\end{proof}

As stated in Step 2.3 of Section \ref{Section-outline}, the basic inequalities (\ref{basic-case-1}),(\ref{basic-case-2}) enable us to accomplish the proof for (\ref{full-observability}), by means of compactness-uniqueness argument.
Since this part of proof is rather analogous to the analysis developed in \cite[Section 4.1]{ADS-16}, we place it in Appendix~\ref{Appendix-obs}.

\subsection{High-frequency dissipation and contractibility}\label{Section-Contractibility}

With these results established in Sections \ref{Section-duality} and \ref{Section-observability}, we are able to demonstrate the high-frequency dissipation and hence the contractibility for (\ref{Problem-linearized}) (see Step 3 of Section \ref{Section-outline}). The first assertion of Proposition \ref{Prop-linear} can be then obtained. 

Let us begin with the following result relevant to (\ref{LF-controllability}), which 
will imply the strong dissipation in the high frequency.

\begin{proposition}\label{theorem-duality-1}
Let $T>T''$, $R>0$ and $m\in\N^+$ be arbitrarily given, and set $N=N(T,R,m)\in\N^+$ to be established in Proposition {\rm\ref{prop-obs}(2)}. Then there exists a constant $C=C(T,R)>0$, not depending on $m,N$, such that for every $\hat u\in B_R$ and $v^0\in \mathcal H$, there is a control $\zeta\in L^2_tH^{\scriptscriptstyle1/5}_x$ satisfying 
\begin{equation}\label{Null-controllability-2}
\mathbf P_mv[T]=\mathbf P_mU(T)v^0\quad {\it and}\quad \int_0^T \|\zeta(t)\|_{_{H^{\scriptscriptstyle1/5}}}^2dt\leq C\|v^0\|^2_{_{\mathcal H}},
\end{equation} 
where $v=\mathcal V_{\hat u}(v^0,\chi\mathscr P^{\scriptscriptstyle T}_{\scriptscriptstyle N}\zeta)$.
\end{proposition}

\begin{proof}[{\bf Proof}]
For $\hat u\in B_R$ and $v^0\in \mathcal H$, we consider a controlled system for the difference 
$$
w=v-z\quad {\rm with\ }v=\mathcal V_{\hat u}(v^0,\chi\mathscr P^{\scriptscriptstyle T}_{\scriptscriptstyle N}\zeta),\ z[\cdot]=U(\cdot)v^0,
$$
where $\zeta$ stands for the control to be determined.
Then, it follows that
\begin{equation}\label{Problem-w}
w=\mathcal V_{\hat u}(0,0,-3\hat u^2z+\chi\mathscr P^{\scriptscriptstyle T}_{\scriptscriptstyle N}\zeta)\footnote{Notice that the solution $w$ could possess higher regularity than $v$, while $z[T]\rightarrow 0$ in $\mathcal H$ as $T\rightarrow+\infty$ (see Lemma \ref{Linear-semigroup}). Such a simple but crucial observation justifies the use of the terminology ``asymptotic regularity''}.
\end{equation} 
We next introduce a function
\begin{equation*}
\tilde w=\mathcal V_{\hat u}^T(0,0,-3\hat u^2z).
\end{equation*}
Recall that for every $g\in H^{\scriptscriptstyle11/7}$, the mapping $f\mapsto gf$ is a bounded linear operator from $H^1$ into itself, and its operator norm can be bounded by $C\|g\|_{_{H^{\scriptscriptstyle11/7}}}$; this is mainly due to the fact that $H^{\scriptscriptstyle11/7}$ is a Banach algebra with respect to pointwise multiplication. This together with Lemma~\ref{Linear-semigroup} implies  that 
\begin{equation}\label{bound-0}
	\int_0^T\|(\hat u^2z)(t)\|^2_{_{H^{1}}}dt\leq C\|v^0\|^2_{_{\mathcal H}},
\end{equation}
where we have also used the setting $\hat u\in B_R$.
This means, in view of the estimate of type (\ref{energy-1}) for $\mathcal V^T_{\hat u}$, that 
\begin{equation}\label{bound-1}
\|\tilde w[t]\|^2_{_{\mathcal  H^{1}}}\leq C\|v^0\|^2_{_{\mathcal H}}
\end{equation}
for all $t\in[0,T]$, where the constant $C$ depends on $T,R$. Letting
$
\tilde v=w-\tilde w,
$
it then follows that 
$$
\tilde v=\mathcal V_{\hat u}(-\tilde w[0],\chi\mathscr P^{\scriptscriptstyle T}_{\scriptscriptstyle N}\zeta).
$$

Now, making use of Propositions \ref{theorem-duality} and \ref{prop-obs}, it follows that there exists a control $\zeta\in L^2_tH^{\scriptscriptstyle1/5}_x$ such that 
\begin{equation}\label{bound-2}
\mathbf P_m\tilde v[T]=0\quad \text{and}\quad \int_0^T \|\zeta(t)\|_{_{H^{\scriptscriptstyle1/5}}}^2dt\leq C\|\tilde w[0]\|^2_{_{\mathcal H^{\scriptscriptstyle1/5}}},
\end{equation} 
where the constant $C$ depends on $T,R$.
Finally, putting (\ref{Problem-w}),(\ref{bound-1}),(\ref{bound-2}) all together, we conclude that
\begin{equation}\label{LFC-w}
\mathbf P_mw[T]=0\quad \text{and}\quad \int_0^T \|\zeta(t)\|_{_{H^{\scriptscriptstyle1/5}}}^2dt\leq C\|v^0\|^2_{_{\mathcal H}},
\end{equation}
which leads to (\ref{Null-controllability-2}), as desired.
\end{proof}

We conclude this subsection with a proof of Proposition \ref{Prop-linear}(1). 

\begin{proof}[{\bf Proof of Proposition \ref{Prop-linear}(1)}]
We first notice that 
\begin{equation}\label{projection-0}
\int_0^T\|\mathscr P^{\scriptscriptstyle T}_{\scriptscriptstyle N}\phi(t)\|^2_{_{H^{\scriptscriptstyle1/5}}}dt=\sum_{j,k=1}^N\lambda_j^{1/5}\phi_{jk}^2\leq \sum_{j,k=1}^\infty\lambda_j^{1/5}\phi_{jk}^2=\int_0^T\|\phi(t)\|^2_{_{H^{\scriptscriptstyle1/5}}}dt
\end{equation}
for any $\phi\in L^2_tH^{\scriptscriptstyle1/5}_x$, where $\phi_{jk}=\int_{D_T}\phi(t,x)\alpha^{\scriptscriptstyle T}_k(t)e_j(x)dtdx$.
To continue,
recall the constant $T_\varepsilon$ established in (\ref{largetime-2}). 
Let us continue to use the setting in the proof of Proposition \ref{theorem-duality-1}, where the time spread $T$ is specified as $T>T_\varepsilon$ and $m\in\N^+$ will be determined later. Recall that $v$ is decomposed as
$$
v=w+z\quad {\rm with\ }w=\mathcal V_{\hat u}(0,0,-3\hat u^2z+\chi\mathscr P^{\scriptscriptstyle T}_{\scriptscriptstyle N}\zeta),\ z[\cdot]=U(\cdot)v^0.
$$ 
Taking (\ref{energy-1}),(\ref{bound-0}),(\ref{LFC-w}),(\ref{projection-0}) into account, it  follows that 
\begin{equation*}
\begin{array}{ll}
\displaystyle\|w[T]\|_{_{\mathcal H^{\scriptscriptstyle1/5}}}^2\leq C\int_0^T\|-3\hat u^2z+\chi\mathscr P^{\scriptscriptstyle T}_{\scriptscriptstyle N}\zeta(t)\|^2_{_{H^{\scriptscriptstyle1/5}}}dt\leq C\|v^0\|_{_{\mathcal H}}^2,
\end{array}
\end{equation*}
and hence
\begin{equation*}
\|(I-\mathbf P_m)w[T]\|_{_{\mathcal H}}^2\leq C\lambda^{-1/5}_m\|v^0\|_{_{\mathcal H}}^2.
\end{equation*}
As a consequence,  for a sufficiently large $m\in\N^+$ there holds
$$
\|(I-\mathbf P_m)w[T]\|_{_{\mathcal H}}\leq \frac{\varepsilon}{2}\|v^0\|_{_{\mathcal H}},
$$
which leads to 
$$
\|w[T]\|_{_{\mathcal H}}\leq \frac{\varepsilon}{2}\|v^0\|_{_{\mathcal H}}.
$$
This together with (\ref{asymptotic-1}) gives rise to the desired.
\end{proof}

\subsection{Structure of the control}\label{Section-controlstructure}

In the previous  subsection, we have obtained the existence of the controls contracting system (\ref{Problem-linearized}). In order to complete the proof of Proposition \ref{Prop-linear}, it remains to investigate the structure of control. As stated in Step 4 of Section \ref{Section-outline}, the HUM-based argument of optimal control will come into play below.

Let us introduce a functional 
$
J\colon\mathbf H_m\rightarrow\R
$
by setting
$$
J(q)=\frac{1}{2}\int_0^T\|\mathscr P^{\scriptscriptstyle T}_{\scriptscriptstyle N}(\chi\varphi)(t)\|^2_{_{H^{\scriptscriptstyle-1/5}}}dt+\langle (v_0,v_1+av_0),\varphi^\bot[0] \rangle_{_{\mathcal H^{\scriptscriptstyle1/5},\mathcal H_*^{\scriptscriptstyle1/5}}},\quad q=(q_1,q_2),
$$
where $\varphi=\mathcal W_{\hat u}^T(\varphi^T)$ with $\varphi^T=(q_2,-q_1+aq_2)$.
Our characterizations of the functional $J$ is collected in the following result, which contributes to the last ingredient of the proof for Proposition \ref{Prop-linear}(2). 

\begin{proposition}\label{solvability-minimize}
Let $T>T''$, $R>0$ and $m\in\N^+$ be arbitrarily given, and set $N=N(T,R,m)\in\N^+$ to be established in Proposition {\rm\ref{prop-obs}(2)}. Then the following assertions hold.
\begin{enumerate}[leftmargin=2em]
\item[$(1)$] For every  $\hat u\in B_R$ and $v^0\in \mathcal H^{\scriptscriptstyle1/5}$, the functional $J$ has a unique global minimizer $\hat q=(\hat q_1,\hat q_2)\in\mathbf H_m$.
		
\item[$(2)$] There exists a constant $C=C(T,R)>0$ such that
\begin{equation*}
\mathbf P_m\hat v[T]=0\quad {\it and}\quad \int_0^T \|\hat\zeta(t)\|_{_{H^{\scriptscriptstyle1/5}}}^2dt\leq C\|v^0\|^2_{_{\mathcal H^{\scriptscriptstyle1/5}}}
\end{equation*} 
for any $\hat u\in B_R$ and $v^0\in \mathcal H^{\scriptscriptstyle1/5}$, where 
$\hat v=\mathcal V_{\hat u}(v^0,\chi\mathscr P^{\scriptscriptstyle T}_{\scriptscriptstyle N}\hat\zeta)$ with 
\begin{equation*}
\hat\zeta=(-\Delta)^{-1/5}\mathscr P^{\scriptscriptstyle T}_{\scriptscriptstyle N}(\chi\hat\varphi), \quad \hat\varphi=\mathcal W_{\hat u}^T(\hat\varphi^T),\quad \hat \varphi^T=(\hat q_2,-\hat q_1+a\hat q_2).
\end{equation*}
				
\item[$(3)$] For every $\hat u\in B_R$, the mapping $\Upsilon_{\hat u}$, defined by 
$$
\Upsilon_{\hat u}\colon\mathcal H^{\scriptscriptstyle1/5} \rightarrow \mathbf H_m\subset\mathcal H_*^{1/5},\quad \Upsilon_{\hat u}(v^0)=\hat q,
$$ 
is a bounded linear operator.
		
\item[$(4)$] The mapping 
$
B_R\ni\hat u\mapsto\Upsilon_{\hat u}\in \mathcal L(\mathcal H^{\scriptscriptstyle1/5};\mathcal H_*^{\scriptscriptstyle1/5})
$ 
is Lipschitz and continuously differentiable.
\end{enumerate}
\end{proposition}

\begin{proof}[{\bf Proof}]
We begin with verifying assertion (1). It is easy to check that for any given $v^0\in \mathcal H^{\scriptscriptstyle1/5}$, the functional $J$ is bounded below on $\mathbf H_m$, i.e.,
$$
r_0:=\inf_{q\in \mathbf H_m}J(q)>-\infty,
$$
which enables us to assure the existence of a global minimizer $\hat q=(\hat q_1,\hat q_2)$.
To verify the uniqueness, let $\tilde q=(\tilde q_1,\tilde q_2)\in \mathbf H_m$ be another minimizer, i.e., $J(\tilde q)=r_0$. Then, one has
\begin{equation*}
\begin{aligned}
&\int_0^T\left\|
\mathscr P^{\scriptscriptstyle T}_{\scriptscriptstyle N}\left[\chi\left(
\frac{\hat\varphi-\tilde\varphi}{2}
\right)\right]
\right\|^2_{_{H^{\scriptscriptstyle-1/5}}}dt+\int_0^T\left\|
\mathscr P^{\scriptscriptstyle T}_{\scriptscriptstyle N}\left[\chi\left(
\frac{\hat\varphi+\tilde\varphi}{2}
\right)\right]
\right\|^2_{_{H^{\scriptscriptstyle-1/5}}}dt\\
&=\frac{1}{2}\int_0^T\|\mathscr P^{\scriptscriptstyle T}_{\scriptscriptstyle N}(\chi\hat\varphi)(t)\|^2_{_{H^{\scriptscriptstyle-1/5}}}dt+\frac{1}{2}\int_0^T\|\mathscr P^{\scriptscriptstyle T}_{\scriptscriptstyle N}(\chi\tilde\varphi)(t)\|^2_{_{H^{\scriptscriptstyle-1/5}}}dt,
\end{aligned}
\end{equation*}
in view of the parallelogram law, where $\hat\varphi=\mathcal W_{\hat u}^T(\hat\varphi^T)$ and $\tilde\varphi=\mathcal W_{\hat u}^T(\tilde\varphi^T)$ with $\hat \varphi^T=(\hat q_2,-\hat q_1+a\hat q_2)$ and $\tilde \varphi^T=(\tilde q_2,-\tilde q_1+a\tilde q_2)$.
Accordingly, 
\begin{equation}\label{bound-32}
\begin{aligned}
&\int_0^T\left\|
\mathscr P^{\scriptscriptstyle T}_{\scriptscriptstyle N}\left[\chi\left(
\frac{\hat\varphi-\tilde\varphi}{2}
\right)\right]
\right\|^2_{_{H^{\scriptscriptstyle-1/5}}}dt+2J(\tfrac{1}{2}(\hat q+\tilde q))\\
&=J(\hat q_1,\hat q_2)+J(\tilde q_1,\tilde q_2).
\end{aligned}
\end{equation}
The RHS of (\ref{bound-32}) equals to $2r_0$, while
$$
{\rm LHS\ of\ }(\ref{bound-32})\geq \int_0^T\left\|
\mathscr P^{\scriptscriptstyle T}_{\scriptscriptstyle N}\left[\chi\left(
\frac{\hat\varphi-\tilde\varphi}{2}
\right)\right]
\right\|^2_{_{H^{\scriptscriptstyle-1/5}}}dt+2r_0.
$$
This implies that 
$$
\int_0^T\left\|
\mathscr P^{\scriptscriptstyle T}_{\scriptscriptstyle N}\left[\chi\left(
\frac{\hat\varphi-\tilde\varphi}{2}
\right)\right]
\right\|^2_{_{H^{\scriptscriptstyle-1/5}}}dt=0,
$$
which combined with the observability (\ref{truncated-observability}) leads to $\hat q=\tilde q$. Therefore, we conclude assertion (1).
	
To prove assertion (2), we first notice the dual identity 
\begin{equation}\label{dual-identity-3}
\begin{aligned}
&\langle \hat v[T], (q_1-aq_2,q_2)\rangle_{_{\mathcal H^{\scriptscriptstyle1/5},\mathcal H_*^{\scriptscriptstyle1/5}}}-\langle (v_0,v_1+av_0), \varphi^\bot[0]\rangle_{_{\mathcal H^{\scriptscriptstyle1/5},\mathcal H_*^{1/5}}}\\
&=-\langle 
\hat v(T),aq_2
\rangle_{_{H^{\scriptscriptstyle6/5},H^{\scriptscriptstyle-6/5}}}+\int_0^T\langle 
\hat\zeta,\mathscr P^{\scriptscriptstyle T}_{\scriptscriptstyle N}(\chi\varphi)
\rangle_{_{H^{\scriptscriptstyle1/5},H^{\scriptscriptstyle-1/5}}} dt,
\end{aligned}
\end{equation}
for any $q=(q_1,q_2)\in \mathbf H_m$, where $
\varphi=\mathcal W_{\hat u}^T(\varphi^T)$ with $\varphi^T=(q_2,- q_1+a q_2)$.
At the same time, since $\hat q$ is the minimizer of $J$, the G\^{a}teaux derivative at $\hat q$ equals to zero. Therefore, it follows that 
\begin{equation}\label{zero-derivation}
\int_0^T\langle \hat\zeta, \mathscr P^{\scriptscriptstyle T}_{\scriptscriptstyle N}(\chi\varphi) \rangle_{_{H^{\scriptscriptstyle1/5},H^{\scriptscriptstyle-1/5}}}dt+\langle (v_0,v_1+av_0),\varphi^\bot[0] \rangle_{_{\mathcal H^{\scriptscriptstyle1/5},\mathcal H_*^{\scriptscriptstyle1/5}}}=0.
\end{equation}
This together with (\ref{dual-identity-3}) gives rise to 
$
\langle \hat v[T], q\rangle_{_{\mathcal H^{\scriptscriptstyle1/5},\mathcal H_*^{\scriptscriptstyle1/5}}}=0.
$
By the arbitrariness of $q\in\mathbf H_m$, it can be derived that 
\begin{equation}\label{Null-controllability-9}
\mathbf P_m\hat v[T]=0.
\end{equation}
	
On the other hand, it can be derived from (\ref{zero-derivation}) with  $q=\hat q$ that
$$
\int_0^T\|\mathscr P^{\scriptscriptstyle T}_{\scriptscriptstyle N}(\chi\hat\varphi)(t) \|_{_{H^{\scriptscriptstyle-1/5}}}^2dt+\langle (v_0,v_1+av_0),\hat\varphi^\bot[0] \rangle_{_{\mathcal H^{\scriptscriptstyle1/5},\mathcal H_*^{\scriptscriptstyle1/5}}}=0.
$$
Moreover, notice that 
\begin{equation}\label{bound-33}
\int_0^T\|\mathscr P^{\scriptscriptstyle T}_{\scriptscriptstyle N}(\chi\hat\varphi)(t) \|_{_{H^{\scriptscriptstyle-1/5}}}^2dt=\int_0^T\|\hat\zeta(t) \|_{_{H^{\scriptscriptstyle1/5}}}^2dt,
\end{equation}
while, taking (\ref{energy-2}),(\ref{truncated-observability}) into account,
\begin{equation*}
\begin{array}{ll}
\displaystyle\left|
\langle (v_0,v_1+av_0),\hat\varphi^\bot[0] \rangle_{_{\mathcal H^{\scriptscriptstyle1/5},\mathcal H_*^{\scriptscriptstyle1/5}}}
\right| \leq  C \|v^0\|_{_{\mathcal H^{\scriptscriptstyle1/5}}}^2+\frac{1}{2}\int_0^T\|\mathscr P^{\scriptscriptstyle T}_{\scriptscriptstyle N}(\chi\hat\varphi)(t) \|_{_{H^{\scriptscriptstyle-1/5}}}^2dt.
\end{array}
\end{equation*}
In summary, we conclude that
\begin{equation}\label{bound-34}
\int_0^T\|\hat\zeta(t) \|_{_{H^{\scriptscriptstyle1/5}}}^2dt\leq C\|v^0\|_{_{\mathcal H^{\scriptscriptstyle1/5}}}^2,
\end{equation}
which together with (\ref{Null-controllability-9}) completes the proof of assertion (2).
	
We proceed to establish the linearity described in assertion (3). For $v^0,w^0\in \mathcal H^{\scriptscriptstyle1/5}$ and $\alpha,\beta\in\R$, 
let us denote 
$$
\hat q=\Upsilon_{\hat u}(v^0),\quad
\hat q'=\Upsilon_{\hat u}(w^0),\quad
\hat q''=\Upsilon_{\hat u}(\alpha v^0+\beta w^0)
$$
and define $\hat\varphi,\hat\varphi',\hat\varphi''$ to be $\mathcal W^T_{\hat u}(q_2,- q_1+a q_2)$ with
$
(q_1,q_2)=\hat q,\,\hat q'$ and $\hat q'',
$
respectively. Then, we repeat the deduction that gave (\ref{zero-derivation}) for the solutions $\alpha\hat\varphi,\beta\hat\varphi',-\hat\varphi''$ and add the resulting identity. It thus follows that
\begin{equation}\label{identity-3}
\int_0^T(\mathscr P^{\scriptscriptstyle T}_{\scriptscriptstyle N}[\chi(\alpha\hat\varphi+\beta\hat\varphi'-\hat\varphi'')],\mathscr P^{\scriptscriptstyle T}_{\scriptscriptstyle N}(\chi\varphi))_{_{H^{\scriptscriptstyle-1/5}}}dt
=0
\end{equation}
for any $q=(q_1,q_2)\in \mathbf H_m$, where $
\varphi=\mathcal W_{\hat u}^T(q_2,- q_1+a q_2)$. Letting 
$
q=\alpha\hat q+\beta\hat q'-\hat q''
$
in (\ref{identity-3}), one derives that
$$
\int_0^T\|\mathscr P^{\scriptscriptstyle T}_{\scriptscriptstyle N}[\chi(\alpha\hat\varphi+\beta\hat\varphi'-\hat\varphi'')]\|^2_{_{H^{\scriptscriptstyle-1/5}}}dt
=0,
$$
which together with (\ref{truncated-observability})  implies that
$
\alpha\hat q+\beta\hat q'-\hat q''=0.
$
That is, $$\Upsilon_{\hat u}(\alpha v^0+\beta w^0)=\alpha\Upsilon_{\hat u}(v^0)+\beta\Upsilon_{\hat u}(w^0),$$ as desired.
In addition, recalling (\ref{bound-33}),(\ref{bound-34}) and using the observability (\ref{truncated-observability}) again, one sees readily that
\begin{equation}\label{bound-36}
\|\hat q\|_{_{\mathcal H^{\scriptscriptstyle1/5}_*}}^2\leq C\|v^0\|_{_{\mathcal H^{\scriptscriptstyle1/5}}}^2,
\end{equation}
where $C>0$ depends only on $T,R$.
This combined with the linearity of $\Upsilon$ as just verified leads to assertion (3).

It remains to demonstrate assertion (4), which will be done by adapting the argument involved in \cite[Proposition 5.5]{Shi-15}. For convenience we denote  by $\|\cdot\|_{\rm sup}$ the supremum norm on $C([0,T];H^{\scriptscriptstyle11/7})$, and write
$$
\Psi_{\hat u}(q_1,q_2)=\mathcal W_{\hat u}^T(q_2,-q_1+a q_2)
$$
for $\hat u\in B_R$ and $(q_1,q_2)\in \mathbf H_m$. Then, from Proposition \ref{Energy-estimate}(3) it follows that
\begin{equation}\label{Lipschitz-continuity}
\|\Psi_{\hat u_1}(q_1,q_2)[t]-\Psi_{\hat u_2}(q_1,q_2)[t]\|_{_{\mathcal H^{\scriptscriptstyle-6/5}}}\leq C\|\hat u_1-\hat u_2\|_{\rm sup}\|(q_1,q_2)\|_{_{\mathcal H_*^{\scriptscriptstyle1/5}}}
\end{equation}
for any $\hat u_1,\hat u_2\in B_R,(q_1,q_2)\in\mathbf H_m$ and $t\in[0,T]$, where the constant $C>0$ depends on $T,R$. To continue, letting 
$$
\Upsilon_i=\Upsilon_{\hat u_i}(v_0,v_1),\quad i=1,2
$$
with $(v_0,v_1)\in \mathcal H^{\scriptscriptstyle1/5}$, it follows from (\ref{zero-derivation}) that
\begin{equation*}
\begin{aligned}
&\int_0^T(\mathscr P^{\scriptscriptstyle T}_{\scriptscriptstyle N}[\chi\Psi_{\hat u_1}(\Upsilon_1)],\mathscr P^{\scriptscriptstyle T}_{\scriptscriptstyle N}[\chi\Psi_{\hat u_1}(q_1,q_2)])_{_{H^{\scriptscriptstyle-1/5}}}dt\\
&-\int_0^T(\mathscr P^{\scriptscriptstyle T}_{\scriptscriptstyle N}[\chi\Psi_{\hat u_2}(\Upsilon_2)],\mathscr P^{\scriptscriptstyle T}_{\scriptscriptstyle N}[\chi\Psi_{\hat u_2}(q_1,q_2)])_{_{H^{\scriptscriptstyle-1/5}}}dt\\ 
&+\langle (v_0,v_1+av_0),[\Psi_{\hat u_1}(q_1,q_2)]^\bot[0]-[\Psi_{\hat u_2}(q_1,q_2)]^\bot[0] \rangle_{_{\mathcal H^{\scriptscriptstyle1/5},\mathcal H_*^{\scriptscriptstyle1/5}}}=0
\end{aligned}
\end{equation*}
for any $(q_1,q_2)\in \mathbf H_m$.
Accordingly, by taking $(q_1,q_2)=\Upsilon_1-\Upsilon_2$, one derives that
\begin{align*}
& \int_0^T
\|\mathscr P^{\scriptscriptstyle T}_{\scriptscriptstyle N}[\chi\Psi_{\hat u_1}(\Upsilon_1-\Upsilon_2)]\|^2_{_{H^{\scriptscriptstyle-1/5}}}dt \\
& +\int_0^T(\mathscr P^{\scriptscriptstyle T}_{\scriptscriptstyle N}[\chi(\Psi_{\hat u_1}-\Psi_{\hat u_2})(\Upsilon_2)],\mathscr P^{\scriptscriptstyle T}_{\scriptscriptstyle N}[\chi\Psi_{\hat u_1}(\Upsilon_1-\Upsilon_2)])_{_{H^{\scriptscriptstyle-1/5}}}dt \\
&
+\int_0^T(\mathscr P^{\scriptscriptstyle T}_{\scriptscriptstyle N}[\chi\Psi_{\hat u_2}(\Upsilon_2)],\mathscr P^{\scriptscriptstyle T}_{\scriptscriptstyle N}[\chi(\Psi_{\hat u_1}-\Psi_{\hat u_2})(\Upsilon_1-\Upsilon_2)])_{_{H^{\scriptscriptstyle-1/5}}}dt\\
& +\langle (v_0,v_1+av_0),[\Psi_{\hat u_1}(\Upsilon_1-\Upsilon_2)]^\bot[0]-[\Psi_{\hat u_2}(\Upsilon_1-\Upsilon_2)]^\bot[0] \rangle_{_{\mathcal H^{\scriptscriptstyle1/5},\mathcal H_*^{\scriptscriptstyle1/5}}}=0.
\end{align*}
Making use of (\ref{truncated-observability}), we have
\begin{equation*}
\int_0^T\|\mathscr P^{\scriptscriptstyle T}_{\scriptscriptstyle N}[\chi\Psi_{\hat u_1}(\Upsilon_1-\Upsilon_2)]\|^2_{_{H^{\scriptscriptstyle-1/5}}}dt\geq C\|\Upsilon_1-\Upsilon_2\|_{_{\mathcal H_*^{1/5}}}^2.
\end{equation*}
At the same time, one can deduce, in view of  (\ref{bound-36}),(\ref{Lipschitz-continuity}), that 
\begin{equation*}
\begin{aligned}
&\left|\int_0^T(\mathscr P^{\scriptscriptstyle T}_{\scriptscriptstyle N}[\chi(\Psi_{\hat u_1}-\Psi_{\hat u_2})(\Upsilon_2)],\mathscr P^{\scriptscriptstyle T}_{\scriptscriptstyle N}[\chi\Psi_{\hat u_1}(\Upsilon_1-\Upsilon_2)])_{_{H^{\scriptscriptstyle-1/5}}}dt \right|\\
&+\left|\int_0^T(\mathscr P^{\scriptscriptstyle T}_{\scriptscriptstyle N}[\chi\Psi_{\hat u_2}(\Upsilon_2)],\mathscr P^{\scriptscriptstyle T}_{\scriptscriptstyle N}[\chi(\Psi_{\hat u_1}-\Psi_{\hat u_2})(\Upsilon_1-\Upsilon_2)])_{_{H^{\scriptscriptstyle-1/5}}}dt\right|\\
&+\left|\langle (v_0,v_1+av_0),[\Psi_{\hat u_1}(\Upsilon_1-\Upsilon_2)]^\bot[0]-[\Psi_{\hat u_2}(\Upsilon_1-\Upsilon_2)]^\bot[0] \rangle_{_{\mathcal H^{\scriptscriptstyle1/5},\mathcal H_*^{\scriptscriptstyle1/5}}}\right|
\\
&\leq C\|\Upsilon_1-\Upsilon_2\|_{_{\mathcal H_*^{\scriptscriptstyle1/5}}}\|\hat u_1-\hat u_2\|_{\rm sup}\|(v_0,v_1)\|_{_{\mathcal H^{\scriptscriptstyle1/5}}}.
\end{aligned}
\end{equation*}
In summary, we conclude that
$$
\|\Upsilon_1-\Upsilon_2\|_{_{\mathcal H_*^{\scriptscriptstyle1/5}}} \leq C\|\hat u_1-\hat u_2\|_{\rm sup}\|(v_0,v_1)\|_{_{\mathcal H^{\scriptscriptstyle1/5}}},
$$
which means the Lipschitz continuity of the mapping $\hat u\mapsto\Upsilon_{\hat u}$.
Finally, the $C^1$-smoothness of $\hat u\mapsto\Upsilon_{\hat u}$ can be directly verified by putting the identity (\ref{zero-derivation}), Proposition \ref{Energy-estimate}(3) (with $s=1/5$) and the implicit
function theorem.
The proof is then complete.
\end{proof}

We conclude this section with a proof of Proposition \ref{Prop-linear}(2).

\begin{proof}[{\bf Proof of Proposition \ref{Prop-linear}(2)}]
We first claim that for arbitrarily given $\hat u\in B_R$ and $v^0=(v_0,v_1)\in\mathcal H^{\scriptscriptstyle1/5}$, the control $\zeta$, constructed by the implication $(2)\Rightarrow (1)$ in Proposition \ref{theorem-duality}, coincides with $\hat\zeta$ constructed by Proposition \ref{solvability-minimize}(1)(2); statement (2) of Proposition \ref{theorem-duality} is by now verified by Proposition \ref{prop-obs}(2). To see this, let $\mathcal Z$ be a subspace of $L^2_tH^{\scriptscriptstyle1/5}_x$, consisting of the functions in the form 
$$
(-\Delta)^{-1/5}\mathscr P^{\scriptscriptstyle T}_{\scriptscriptstyle N}(\chi\varphi),\quad \varphi=\mathcal W_{\hat u}^T(q_2,- q_1+aq_2)
$$
with any $(q_1,q_2)\in\mathbf H_m$. Due to (\ref{truncated-observability}), it is not difficult to check that $\mathcal Z$ is closed in $L^2_tH^{\scriptscriptstyle1/5}_x$. In addition, following the argument that gave (\ref{dual-identity-3}), one finds that 
\begin{equation*}
\begin{aligned}
&\langle v[T], (q_1-aq_2,q_2)\rangle_{_{\mathcal H^{\scriptscriptstyle1/5},\mathcal H_*^{\scriptscriptstyle1/5}}}-\langle (v_0,v_1+av_0), \varphi^\bot[0]\rangle_{_{\mathcal H^{\scriptscriptstyle1/5},\mathcal H_*^{\scriptscriptstyle1/5}}}\\
&=-\langle 
v(T),aq_2
\rangle_{_{H^{\scriptscriptstyle6/5},H^{\scriptscriptstyle-6/5}}}+\int_0^T\langle 
\zeta,\mathscr P^{\scriptscriptstyle T}_{\scriptscriptstyle N}(\chi\varphi)
\rangle_{_{H^{\scriptscriptstyle1/5},H^{\scriptscriptstyle-1/5}}} dt,
\end{aligned}
\end{equation*}
where $v=\mathcal V_{\hat u}(v^0,\chi\mathscr P^{\scriptscriptstyle T}_{\scriptscriptstyle N}\zeta)$.
This together with (\ref{dual-identity-3}) implies that 
$$
\int_0^T(
\zeta,\tilde\zeta
)_{_{H^{\scriptscriptstyle1/5}}} dt=\int_0^T(
\hat\zeta,\tilde\zeta
)_{_{H^{\scriptscriptstyle1/5}}} dt
$$
for any $\tilde\zeta\in\mathcal Z$. Accordingly, $\hat\zeta$ is the orthogonal projection of $\zeta$ on the space $\mathcal Z$. This implies that 
$$
\int_0^T\|\hat\zeta(t)\|^2_{_{H^{\scriptscriptstyle1/5}}} dt\leq \int_0^T\|\zeta(t)\|^2_{_{H^{\scriptscriptstyle1/5}}} dt.
$$
At the same time, one can recall (\ref{minimizer-2}) to deduce that 
$
\int_0^T\|\zeta(t)\|^2_{_{H^{\scriptscriptstyle1/5}}} dt\leq \int_0^T\|\hat\zeta(t)\|^2_{_{H^{\scriptscriptstyle1/5}}} dt,
$
which gives rise to $\hat\zeta=\zeta$ immediately. In what follows, we shall identify $\zeta$ with $\hat\zeta$.

Thanks to Proposition \ref{solvability-minimize}(3), the mapping
$
\mathcal H^{\scriptscriptstyle1/5}\ni v^0\mapsto \hat\zeta\in L^2_tH^{\scriptscriptstyle1/5}_x
$
is a bounded linear operator for every $\hat u\in B_R$. We denote by $\Phi_0(\hat u)$ this operator. Moreover, Proposition \ref{solvability-minimize}(4) implies that the mapping
$$
\Phi_0\colon
B_R\rightarrow \mathcal L(\mathcal H^{\scriptscriptstyle1/5};L^2_tH^{\scriptscriptstyle1/5}_x)
$$
is Lipschitz and continuously differentiable.
Furthermore,  recalling the proof of Proposition \ref{Prop-linear}(1), one can notice that the  control $\zeta$ verifying (\ref{binding-property}) can be expressed as 
\begin{equation*}
\zeta=\Phi(\hat u)v^0:=\Phi_0(\hat u)(-\mathcal V^T_{\hat u}(0,0,-3\hat u^2z)[0])\quad {\rm with\ } z[\cdot]=U(\cdot)v^0,
\end{equation*} 
for every $\hat u\in B_R$ and $v^0\in\mathcal H$. Finally, the second assertion of Proposition \ref{Prop-linear}, i.e. the Lipschitz property and $C^1$-smoothness of the mapping
$
\Phi\colon B_R\rightarrow \mathcal L(\mathcal H;L^2_tH^{\scriptscriptstyle1/5}_x),
$
is a direct consequence of those properties of $\Phi_0$. The proof is then complete.
\end{proof}

Eventually, our main result, i.e. Theorem \ref{Th1}, can be derived from the conclusions of Proposition \ref{Prop-linear} in a very direct way. See Appendix \ref{Appendix-squeezing} for the details.

\begin{remark}
Inspired by {\rm\cite{ADS-16}}, the similar conclusions as in Theorem {\rm\ref{Th1}} and Proposition {\rm\ref{Prop-linear}} can also be established in a more general setting. In particular, when the potential term $3\hat u^2v$ in system {\rm(\ref{Problem-linearized})} is replaced by a general one $p(t,x)v$ with
$$
p\in L^\infty(D_T)\cap L^\infty_tH^r_x\quad {\rm for\ some\  }r>0,
$$ 
the contractibility presented in Proposition {\rm\ref{Prop-linear}} remains true. The proof in this situation follows the same idea, except that the space which we work with for improving the regularity of the control is taken to be $H^{\sigma}$ for some $\sigma=\sigma(r)>0$, instead of $H^{\scriptscriptstyle1/5}$. As a consequence, it is possible to verify the squeezing property for system {\rm(\ref{controlproblem-0})} in the case where the source term $u^3$ replaced with a general one $f(u)$. 
In the present paper, the emphasis is not to seek for the ``sharp'' conditions on $f$ which could guarantee the squeezing property.
\end{remark}

\begin{remark}\label{remark-stabilization}
Our result of contractibility {\rm(}i.e.  Proposition {\rm\ref{Prop-linear})} takes account of controlled solutions on $[0,T]$. Nevertheless, under suitable conditions, applying the contractibility properties on the intervals $[nT,(n+1)T]\ ( n\in\N)$ could enable one to deduce the exponential stabilization to the origin for system {\rm(\ref{Problem-linearized})}. That is, for every $v^0\in\mathcal H$, there exists a control $\zeta \in L^2_{loc}(\R^+;H)$ such that 
$$
v[t]\rightarrow 0\quad\text{in }\mathcal H
$$
at an exponential rate.
A similar situation could arise in
the squeezing property {\rm(}i.e. Theorem {\rm\ref{Th1})},
which could also indicate the exponential stabilization to an uncontrolled {\rm(}global{\rm)} solution $\hat u$ for system {\rm(\ref{controlproblem-0})}. This is roughly illustrated as 
$$
u[t]-\hat u[t]\rightarrow 0\quad\text{in }\mathcal H
$$
at an exponential rate.
\end{remark}

\section{Exponential mixing for random nonlinear wave equations}\label{Section-conclusion}

With the preparations from Sections \ref{Section-Probalistic part}-\ref{Section-control}, we are now able to establish exponential mixing for the random wave equation (\ref{Problem-0}), i.e. Theorem~\ref{Thm-informal-model}. The verification of abstract hypotheses, i.e. $(\mathbf{AC})$, $(\mathbf{I})$ and $(\mathbf{C})$ in Theorem \ref{Thm A}, contributes to the main content of the proof. More precisely, this will be done by  the 
following technical route
\vspace{0.3em}
\begin{itemize}[leftmargin=2em]
\item ``$(\mathcal H,\mathcal H^{\scriptscriptstyle4/7})$-asymptotic compactness in Theorem \ref{Thm-dynamicalsystem}'' implies  hypothesis~$(\mathbf{AC})$ (see Section~\ref{Section-6-1});

\item ``Global stability of the unforced problem in Proposition \ref{prop-global-stability}''  implies  hypothesis~$(\mathbf{I})$ (see Section \ref{Section-6-2});

\item ``Squeezing property in Theorem \ref{Th1}'' implies hypothesis~$(\mathbf{C})$ (see Section \ref{Section-6-3}).
\end{itemize}

\vspace{0.3em}
We mention that the parameters $R_0$ in Theorem~\ref{Thm-dynamicalsystem} and $R$ in Theorem~\ref{Th1} will be involved in the proof. Both of them  are directly determined by $T$ and $B_0$ below. In addition, some basic facts from the measure theory are useful in the verification of  hypothesis~$(\mathbf{C})$. For the reader's convenience, these necessary results are collected in Appendix~\ref{Section-6-5}.

\vspace{0.6em}

Below is to summarize the structure of  $\eta(t,x)$ involved in Theorem \ref{Thm-informal-model}. Under the setting $(\mathbf{S1})$ on  $a(x),\,\chi(x)$, we specify the quantity $\mathbf{T}$ as $\mathbf{T}=T_{\varepsilon}$, by means of  Theorem~\ref{Th1} with  $\varepsilon= 1/4$.  Letting $T>\mathbf{T}$ and $B_0>0$ be arbitrarily given, the random noise $\eta(t,x)$ in {\rm(\ref{Problem-0})} is of the form
\begin{equation*}
\begin{aligned}
&\eta(t,x)= \eta_n(t-nT,x),\quad t\in [nT,(n+1)T),\,n\in\N,\\ 
&\displaystyle\eta_n(t,x)=\chi(x)\sum_{j,k\in\N^+}b_{jk}\theta^{n}_{jk}\alpha^{\scriptscriptstyle T}_k(t)e_j(x),\quad t\in[0,T).
\end{aligned}
\end{equation*}
Here, the sequence $\{b_{jk};j,k\in\N^+\}$ of nonnegative numbers verifies
\begin{equation}\label{bounded-noise-0}
\sum_{j,k\in\N^+}b_{jk}\lambda_j^{2/7}\|\alpha_k\|_{_{L^\infty(0,1)}} \leq {B}_0T^{1/2},
\end{equation}
while $\{\theta_{jk}^{n};n\in\N \}$ is a sequence of i.i.d.~random variables with density $\rho_{jk}$ satisfying $(\mathbf{S2})$. We emphasize here that an integer $N$ will be appropriately chosen in Step 2 of Section {\rm\ref{Section-6-3} (}depending on $T,B_0${\rm)}, so that the conclusion of exponential mixing in Theorem {\rm\ref{Thm-informal-model}} is assured, provided that 
\begin{equation}\label{degenerate-noise}
b_{jk}\neq0,\quad\forall\, 1\leq j,k\leq N.
\end{equation}

\medskip

Recalling that $\alpha_k^{\scriptscriptstyle T}(t)=\frac{1}{\sqrt{T}}\alpha_k(\frac{t}{T})$, it follows from (\ref{bounded-noise-0}) that there exists a constant $B_1=B_1(\chi,B_0)>0$ such that
\begin{equation}\label{B1-bound1}
\sum_{j,k\in\N^+}b_{jk}\|\chi e_{j}\alpha_{k}^{\scriptscriptstyle T}\|_{_{L^\infty(0,T;H^{4/7})}}\leq B_1.
\end{equation}
Noticing that $\{\eta_n;n\in \N\}$ are i.i.d.~$L^2(D_T)$-valued random variables, we denote its common law by $\ell$, and the support by $\mathcal{E}$. In view of (\ref{B1-bound1}),  $\mathcal{E}$ is compact in $L^2(D_T)$ and bounded in $L^\infty(0,T;H^{\scriptscriptstyle4/7})$.

Let $\{u^n;n\in\N\}$ be the Markov process defined via \eqref{setting-wave}.
The corresponding Markov transition functions and Markov semigroups are written as $\{P_n(u,A);u\in\mathcal{H},\,A\in\mathcal{B}(\mathcal{H}),\,n\in\N\}$, $P_n$, $P_n^*$ as in Section~\ref{Section-Probalistic part}, respectively. In particular, for any $\mathcal{H}$-valued random initial condition $u^0$ (independent of $\{\eta_n;n\in\N\}$) with law $\nu\in\mathcal{P}(\mathcal{H})$,  one has
\begin{equation*}
    \mathscr{D}(u^n)=P_n^*\nu,\quad\forall\, n\in\N,
\end{equation*}
see, e.g., \cite[Section 1.3]{KS-12}.

\subsection{Asymptotic compactness}\label{Section-6-1}

Taking (\ref{structure-noise}),(\ref{B1-bound1}) into account, 
we observe that the sample paths of $\eta$ are contained in a bounded subset of $L^\infty(\R^+;H^{\scriptscriptstyle4/7})$. That is, 
$$
\eta\in\overline{B}_{L^\infty(\R^+;H^{\scriptscriptstyle4/7})}(R_0)\quad\text{almost surely with }R_0=B_1.
$$
This means that for every $\boldsymbol{\zeta}=\{\zeta_n;n\in\N
\}\in \mathcal{E}^{\N}$, the concatenation $f\colon\R^+\rightarrow H$ of $\boldsymbol{\zeta}$, i.e., 
$$
f(t,x)=\zeta_n(t-nT,x),\quad t\in[nT,(n+1)T),\,n\in\N,
$$ 
belongs to $\overline{B}_{L^\infty(\R^+;H^{\scriptscriptstyle4/7})}(R_0)$. This together with Theorem \ref{Thm-dynamicalsystem} implies that there exists a bounded subset $\mathscr B_{\scriptscriptstyle4/7}$ of $\mathcal H^{\scriptscriptstyle4/7}$ and  constants $C$, $\kappa>0$, all determined by ${B}_1$, such that
\begin{equation*}
\text{dist}_{_\mathcal H}(S_n(u;\boldsymbol{\zeta}),\mathscr B_{\scriptscriptstyle4/7})\leq C(1+E(u))e^{-\kappa Tn}
\end{equation*}
for any $u\in\mathcal{H},\boldsymbol{\zeta}\in \mathcal{E}^{\N}$ and $n\in\N$. Therefore, we conclude that hypothesis $(\mathbf{AC})$ holds with $\mathcal Y=\overline{\mathscr B_{\scriptscriptstyle4/7}}$ and $V(u)=C(1+E(u))$.

\vspace{0.3em}
\subsection{Irreducibility}\label{Section-6-2}
Let $\mathcal Y_\infty$ be the attainable set from $\mathcal Y=\overline{\mathscr B_{\scriptscriptstyle4/7}}$ (see Definition \ref{Def-attainable}). It then follows from  Corollary~\ref{Corollary-S2bound} that there exists $R_1=R_1({B}_1)>0$, such that 
$\mathcal Y_\infty\subset \overline{B}_{\mathcal{H}^{\scriptscriptstyle4/7}}(R_1)$. Making use of Proposition \ref{prop-global-stability}, one can derive that for any $\varepsilon>0$, there exists an integer $m=m(T,{B_1},\varepsilon)$ such that
\begin{equation*}
\|S_m(u;\boldsymbol{0})\|_{_\mathcal{H}}<\frac{\varepsilon}{2}
\end{equation*} 
for any $u\in\overline{B}_{\mathcal{H}^{\scriptscriptstyle4/7}}(R_1)$, where $\boldsymbol{0}$ stands for a sequence of zeros. Combined with
the compactness of $\overline{B}_{\mathcal{H}^{\scriptscriptstyle4/7}}(R_1)\times\mathcal{E}^m$ and the continuity of the map $(u,\boldsymbol{\zeta})\mapsto S_m(u;\boldsymbol{\zeta})$ (see Proposition \ref{prop-wellposedness}(1)), we then obtain that 
there exists a constant $\delta>0$ such that 
\begin{equation*}
\|S_m(u;\boldsymbol{\zeta})\|_{_\mathcal{H}}<\varepsilon
\end{equation*}
for any $u\in\overline{B}_{\mathcal{H}^{\scriptscriptstyle4/7}}(R_1)$ and $\boldsymbol{\zeta}=\{\zeta_n;n\in\N\}$ with $\zeta_n\in \mathcal E\cap B_{L^2(D_T)}(\delta)$. As a consequence, 
\begin{equation*}
\begin{aligned}
P_m(u,B_{\mathcal{H}}(\varepsilon)) & \geq \Pb (\|\eta_n\|_{L^2(D_T)}<\delta,\quad\forall\,0\leq n\leq m-1) \\
& = \ell (B_{L^2(D_T)}(\delta))^{m}\\
& >0
\end{aligned}
\end{equation*}
for any $u\in\overline{B}_{\mathcal{H}^{\scriptscriptstyle4/7}}(R_1)$.
Here, the last step is due to the fact $0\in \mathcal E$, which is assured by $\rho_{jk}(0)>0$. Hypothesis $(\mathbf{I})$ is then verified.

\subsection{Coupling condition}\label{Section-6-3} 

In order to verify hypothesis $(\mathbf{C})$, we need some preliminaries regarding the optimal coupling (see Appendix~\ref{Section-6-5}).
For $\boldsymbol{\varepsilon}=(\varepsilon_1,\varepsilon_2)$ with $0\leq \varepsilon_2\leq \varepsilon_1<\infty$, we define a functional $\rho_{\boldsymbol{\varepsilon}}\colon\mathcal H\times \mathcal H\rightarrow [0,1]$ by 
$$
\rho_{\boldsymbol{\varepsilon}}(z)=\varphi_{\boldsymbol{\varepsilon}}(\|u-v\|_{_{\mathcal H}}),\quad z=(u,v)\in\mathcal H\times\mathcal H,
$$
where $\varphi_{\boldsymbol{\varepsilon}}\colon\R^+\rightarrow[0,1]$ is given by
\begin{equation}\label{def-varphi}
\varphi_{\boldsymbol{\varepsilon}}(s)=\begin{cases}
1 &\text{ for } s>\varepsilon_1,\\
\frac{s-\varepsilon_2}{\varepsilon_1-\varepsilon_2} &\text{ for } \varepsilon_2<s\leq \varepsilon_1,\\
0 &\text{ for } 0\leq s\leq \varepsilon_2.
\end{cases}
\end{equation}
Let us also set
\begin{equation*}
\|\mu-\nu\|_{\boldsymbol{\varepsilon}}=\inf_{(\xi,\eta)\in\mathscr{C}(\mu,\nu)} \E \rho_{\boldsymbol{\varepsilon}}(\xi,\eta),\quad \mu,\nu\in\mathcal P(\mathcal H),
\end{equation*}
where $\mathscr{C}(\mu,\nu)$ stands for the set of all couplings for $\mu$ and $\nu$ (see Section \ref{Section-Probalistic part}). 
\medskip

We now begin the analysis of coupling condition, which will be divided into four steps.

\noindent {\bf Step 1}. Let us introduce a measurable space
\begin{equation*}
Z=\left\{z=(u,v)\in\boldsymbol{Y}_{\infty};\|u-v\|_{_{\mathcal{H}}}\leq d\right\}
\end{equation*}
with $ \boldsymbol{Y}_{\infty}=\mathcal{Y}_{\infty}\times\mathcal{Y}_{\infty}$ and $d>0$ that will be chosen below, and a nonnegative measurable function on $Z$, i.e., 
\begin{equation*}
\lambda(z)=\frac{1}{2}\|u-v\|_{_{\mathcal{H}}},\quad z=(u,v)\in Z.
\end{equation*}
With the above settings, an application of Proposition \ref{Coupling lemma 1} with $(\theta_1,\theta_2)=(1/2,1)$ yields that there exists a probability space $(\Omega,\mathcal{F},\Pb)$ and  measurable mappings $\mathcal{R},\mathcal{R}'\colon Z\times\Omega\rightarrow \mathcal{H}$ such that 
$(\mathcal{R}(z),\mathcal{R}'(z))\in\mathscr{C}(P_1(u,\cdot),P_1(v,\cdot))$ and 
\begin{equation*}
\E \rho_{(\lambda(z),\lambda(z))}(\mathcal{R}(z),\mathcal{R}'(z))\leq \|P_1(u,\cdot)-P_1(v,\cdot)\|_{(\lambda(z),\lambda(z)/2)}
\end{equation*}
for any $z=(u,v)\in Z$.
Accordingly, using the definitions of $\rho$ and $\lambda$, 
\begin{equation}\label{Control LCC squeeze2}
\Pb(\|\mathcal{R}(z)-\mathcal{R}'(z)\|_{_\mathcal{H}}>\frac{1}{2}\|u-v\|_{_{\mathcal{H}}})\leq \|P_1(u,\cdot)-P_1(v,\cdot)\|_{(\lambda(z),\lambda(z)/2)}.
\end{equation} 

\medskip

\noindent{\bf Step 2}. In view of \eqref{B1-bound1}, $\mathcal E$ is a bounded subset of $L^2_tH^{\scriptscriptstyle4/7}_x$. Recall also that $\mathcal Y_\infty$ is bounded in $\mathcal H^{\scriptscriptstyle4/7}$ and choose $R_2=R_2(T,{B}_1)$  satisfying
$$
\mathcal Y_\infty\subset \overline{B}_{\mathcal H^{\scriptscriptstyle4/7}}(R_1),\quad \mathcal E\subset\overline{B}_{L^2_tH^{\scriptscriptstyle4/7}_x}(R_2).
$$
Then, taking Proposition~\ref{prop-wellposedness} into account, there exists a constant $R=R(T,{B}_1)>0$ such that 
$$
\hat u\in B_R\quad \text{with }\hat u[\cdot]=\mathcal S(\hat u^0,h)
$$
for any $\hat u^0\in \overline{B}_{\mathcal H^{\scriptscriptstyle4/7}}(R_1)$ and $h\in\overline{B}_{L^2_tH^{\scriptscriptstyle4/7}_x}(R_2+1)$, where $B_R$ is defined by (\ref{potential-space}).

Therefore, invoking Theorem \ref{Th1} (with $\varepsilon=1/4$), it allows to fix the constants $d>0$, $N\in\N^+$ depending only on $T,\,B_1$, and a mapping $$
\Phi'\colon \overline{B}_{\mathcal H^{\scriptscriptstyle4/7}}(R_1)\times \overline{B}_{L^2_tH^{\scriptscriptstyle4/7}_x}(R_2+1)\rightarrow \mathcal L(\mathcal H;L^2(D_T))
$$ 
such that 
\begin{equation}\label{Control LCC squeeze1}
\|S(u,\zeta)-S(v,\zeta+\chi\mathscr P^{\scriptscriptstyle T}_{\scriptscriptstyle N}\Phi'(u,\zeta)(u-v))\|_{_{\mathcal{H}}}\leq \frac{1}{4}\|u-v\|_{_{\mathcal{H}}}
\end{equation}
for any $\zeta\in \overline{B}_{L^2_tH^{\scriptscriptstyle4/7}_x}(R_2+1)$ and $u,v\in\overline{B}_{\mathcal H^{\scriptscriptstyle4/7}}(R_1)$ with $\|u-v\|_{_{\mathcal H}}\leq d$. Moreover, the mapping $\Phi'$ is Lipschitz and continuously differentiable. Now we assume (\ref{degenerate-noise}) with $N$ just established.

\medskip

\noindent{\bf Step 3}. Let $z=(u,v)\in Z$ be fixed and define a transformation $\Psi^{z}$ on $L^2_tH^{\scriptscriptstyle4/7}_x$ by
\begin{equation*}
\Psi^{z}(\zeta)=\zeta+
\Phi^{z}(\zeta),\quad 
\Phi^{z}(\zeta)=
\phi\left(\|\zeta\|_{_{L^2_tH^{\scriptscriptstyle4/7}_x}}^2\right)\chi\mathscr P^{\scriptscriptstyle T}_{\scriptscriptstyle N}\Phi'(u,\zeta)(u-v),
\end{equation*}
where $\phi\colon\R^+\rightarrow\R^+$ is a smooth function such that $\phi(s)=1$ for $s\leq R_2^2$ and $\phi(s)=0$ for $s\geq (R_2+1)^2$. Inequality (\ref{Control LCC squeeze1}) then gives rise to 
$$
\|S(u,\zeta)-S(v,\Psi^{z}(\zeta)\|_{_{\mathcal{H}}}\leq \frac{1}{4}\|u-v\|_{_{\mathcal{H}}}
$$
for $\ell$-almost every $\zeta\in L^2(D_T)$; notice that $\ell(\overline{B}_{L^2_tH^{\scriptscriptstyle4/7}_x}(R_2))=1$. Then, thanks to Lemma~\ref{Coupling lemma 2} with $\boldsymbol{\varepsilon}=(\lambda(z),\lambda(z)/2)$, this implies that
\begin{equation}\label{Control LCC squeeze3}
\|P_1(u,\cdot)-P_1(v,\cdot)\|_{(\lambda(z),\lambda(z)/2)}\leq 2\|\ell-\Psi^z_*\ell\|_{\rm TV},
\end{equation}
where $\|\ell-\Psi^z_*\ell\|_{\rm TV}$ denotes the total variation distance between two probability measures $\ell$ and $\Psi^z_*\ell$ (see \cite[Section 1.2.3]{KS-12}). 

To estimate the RHS of (\ref{Control LCC squeeze3}), we observe that the mapping $\Phi^z$ is Lipschitz and continuously differentiable on $L^2_tH^{\scriptscriptstyle4/7}_x$, while its range is contained in  $$\mathcal{Z}_1:={\rm span}\{\chi e_j\alpha^{\scriptscriptstyle T}_k;1\leq j,k\leq N\}.$$  
We further take $\mathcal{Z}_2=\overline{{\rm span}}\{\chi e_j\alpha^{\scriptscriptstyle T}_k;j>N\ \text{or }k>N\}$ and 
$
\mathcal{Z}=\mathcal{Z}_1\oplus\mathcal{Z}_2$. All these spaces are endowed with the $L^2_tH^{\scriptscriptstyle4/7}_x$-norm.
Using the noise structure \eqref{structure-noise} and \eqref{B1-bound1}, the probability measure $\ell$ on $(\mathcal{Z},\mathcal{B}(\mathcal{Z}))$ can be represented as the tensor product of its projections $\ell_1=(\mathsf{P}_{\mathcal{Z}_1})_*\ell$ and $\ell_2=(\mathsf{P}_{\mathcal{Z}_2})_*\ell$ as in Appendix~\ref{Appendix-1-2}. Moreover, by \eqref{degenerate-noise}, the sequence $\{b_{jk};j,k\in\N^+\}$ satisfies 
\begin{equation*}
    b_{jk}\neq 0\quad \text{for }1\leq j,k\leq N.
\end{equation*}
As a consequence, it allows one to employ Lemma~\ref{Coupling lemma 3} with $\varkappa$ being a proportion of $\|u-v\|_{_{\mathcal H}}$. Here, we also have used the fact that $\theta^{n}_{jk}$ admits the $C^1$-density $\rho_{jk}$. Thus there exists a constant $C>0$, depending on $b_{jk}$, such that
\begin{equation}\label{Control LCC squeeze4}
\|\ell-\Psi^z_*\ell\|_{\rm TV}\leq C\|u-v\|_{_{\mathcal H}}.
\end{equation}

Putting (\ref{Control LCC squeeze2}), (\ref{Control LCC squeeze3}) and (\ref{Control LCC squeeze4}) all together, we conclude that
\begin{equation}\label{coupling-binding}
\Pb(\|\mathcal{R}(z)-\mathcal{R}'(z)\|_{_{\mathcal{H}}}>\frac{1}{2}\|u-v\|_{_{\mathcal{H}}})\leq C_1\|u-v\|_{_\mathcal{H}},
\end{equation}
with a constant $C_1>0$. So, conditions (\ref{squeezing}),(\ref{g-def}) are verified for $(x,x')=(u,v)\in Z$, by taking $r=1/2$ and $g(s)=C_1s$. 

\medskip

\noindent{\bf Step 4}. Finally, the case of $z=(u,v)\in \boldsymbol{Y}_{\infty}\setminus Z$ is trivial. Indeed, without loss of generality, we can take  $\zeta,\zeta'$ to be independent random variables on $(\Omega,\mathcal{F},\Pb)$ with law  $\ell$. Then one can reach (\ref{coupling-binding}) by replacing $C_1$ with $d^{-1}$ and taking 
$$
\mathcal R(z)=S(u,\zeta),\quad \mathcal R'(z)=S(v,\zeta').
$$

Combining these analyses in Sections \ref{Section-6-1}-\ref{Section-6-3}, we verify the hypotheses $(\mathbf{AC})$, $(\mathbf{I})$ and $(\mathbf{C})$ laid out in Section \ref{Section-generalresult}. Therefore, an application of Theorem \ref{Thm A} leads to the conclusions of Theorem \ref{Thm-informal-model}.

\begin{appendices}

\section{Supplementary ingredients in probability}\label{Appendix-Probability}

    In this appendix, we summarize some useful supplementary probabilistic materials and the coupling method, as well as the proofs of Proposition~\ref{local exponential mixing} and Proposition~\ref{Limit theorems}.  

    \subsection{Supplementary materials}

    \subsubsection{Criterion for mixing on compact spaces}\label{K-S framework} In this subsection, we recall some results on exponential mixing of discrete-time Markov processes on compact spaces.  Let $(X,d)$ be a compact metric space and $\{x_n;n\in\N\}$ with $x_0=x$ be a Feller family of discrete-time Markov processes in $X$. We denote by $P_n(x,A)$ the corresponding Markov transition function, $P_n$ and  $P^*_n$ the Markov semigroups. 
    
    Let $ \mathbf{X}=X\times X$ and define the natural projections 
    $$\Pi,\Pi'\colon\mathbf{X}\rightarrow X,\quad \Pi(\vec{\boldsymbol{x}})=x,\,\Pi'(\vec{\boldsymbol{x}})=x'$$ 
    for $\vec{\boldsymbol{x}}=(x,x')$. A Markov process $\{\vec{\boldsymbol{x}}_n;n\in\N\}$ with phase space $\mathbf{X}$ is called an \textit{extension} for $\{x_n;n\in\N\}$ if, for every $n\in\N$ and $\vec{\boldsymbol{x}}=(x,x')\in\mathbf{X}$, we have 
    \begin{equation*}
    \Pi_* \boldsymbol{P}_n(\vec{\boldsymbol{x}}, \cdot)=P_n(x, \cdot), \quad \Pi_*^{\prime} \boldsymbol{P}_n(\vec{\boldsymbol{x}}, \cdot)=P_n\left(x^{\prime}, \cdot\right),
    \end{equation*}
    where $\boldsymbol{P}_n(\vec{\boldsymbol{x}},\cdot)$ stands for the transition function of $\{\vec{\boldsymbol{x}}_n;n\in\N\}$, and $\varphi_*\mu$ denotes the push-forward of the measure $\mu$ defined by $\varphi_*\mu(\cdot)=\mu(\varphi^{-1}(\cdot))$.  We also denote by  $\boldsymbol{P}_n,\boldsymbol{P}^*_n$ the corresponding Markov semigroups and by $\mathbf P_{\vec{\boldsymbol{x}}}$ the Markov family. By definition one has $$\boldsymbol{P}_n(\vec{\boldsymbol{x}},\cdot)\in\mathscr C(P_n(x,\cdot),P_n(x',\cdot))$$ for every $\vec{\boldsymbol{x}}=(x,x')\in\mathbf{X}$ and $n\in\N$. For clarity we also write $\vec{\boldsymbol{x}}_n=(x_n,x_n')$.

    We now recall the following theorem involving exponential mixing of discrete-time Markov processes on compact spaces.
       
    \begin{theorem}\label{Theorem-KS} {\rm(}Kuksin--Shirikyan {\rm\cite{KS-12})}
    Assume that the Markov process $\{x_n;n\in\N\}$ has an extension $\{\vec{\boldsymbol{x}}_n;n\in\N\}$ satisfying the following properties for some closed set $\boldsymbol{B}\subset\mathbf{X}:$ 
    \begin{itemize}[leftmargin=1em]
    \item  {\rm (}Recurrence{\rm )} The hitting time of  $\boldsymbol{B}$, defined by $$\boldsymbol{\tau}=\inf\{n\in\N;\vec{\boldsymbol{x}}_n\in \boldsymbol{B}\},$$ 
    is $\mathbf P_{\vec{\boldsymbol{x}}}$-almost surely finite for every $\vec{\boldsymbol{x}}\in\mathbf{X}$. Moreover, there exists a constant $\beta_1>0$ such that
    \begin{equation}\label{KS recurrence}
    \sup_{\vec{\boldsymbol{x}}\in\mathbf{X}}\mathbf E_{\vec{\boldsymbol{x}}}\exp(\beta_1\boldsymbol{\tau})<\infty.
    \end{equation}

    \item {\rm (}Squeezing{\rm )} There exist constants $c,\beta_2,\beta_3>0$ such that the stopping time $$\boldsymbol{\sigma}=\inf\{n\in\N;d(x_n,x_n')>ce^{-\beta_2n}\}$$ satisfies the following inequalities: 
    \begin{align}
    \inf_{\vec{\boldsymbol{x}}\in\boldsymbol{B}} \mathbf P_{\vec{\boldsymbol{x}}}(\boldsymbol{\sigma}=\infty)&>0,\label{KS squeezing1}\\
    \sup_{\vec{\boldsymbol{x}}\in\boldsymbol{B}} \mathbf E_{\vec{\boldsymbol{x}}}(\mathbf{1}_{\{\boldsymbol{\sigma}<\infty\}}\exp(\beta_3\boldsymbol{\sigma}))&<\infty,\label{KS squeezing2}
    \end{align}
    \end{itemize} 
    where $\mathbf{1}_A$ denotes the indicator function on set $A$. Then the Markov process $\{x_n;n\in\N\}$ has a unique invariant measure $\mu_*\in\mathcal{P}(X)$, which is exponentially mixing, i.e., there exist constants $C_0,\beta_0>0$ such that
    \begin{equation*}
    \|P_n^*\nu-\mu_*\|_L^*\leq C_0e^{-\beta_0 n}
    \end{equation*}
    for any $\nu\in\mathcal{P}(X)$ and  $n\in\N$.
    \end{theorem}

    \subsubsection{Transformations of measures under regular mappings}\label{Appendix-1-2}
        Let  $({\mathcal{Z}},\|\cdot\|_{\mathcal{Z}})$ be a separable Banach space that can be represented as the direct sum of two closed subspaces 
        \begin{equation*}
            {\mathcal{Z}}={\mathcal{Z}}_1\oplus {\mathcal{Z}}_2,
        \end{equation*}
        where ${\mathcal{Z}}_1$ is finite-dimensional, and we denote by $\mathsf{P}_{{\mathcal{Z}}_1}$ and $\mathsf{P}_{{\mathcal{Z}}_2}$ the corresponding projections. Assume further that $({\mathcal{Z}},\mathcal{B}({\mathcal{Z}}),\ell)$ is a probability space, where the probability measure 
        $\ell$  has a bounded support, and  can be written as the tensor product of its projections $\ell_{1}=(\mathsf{P}_{{\mathcal{Z}}_1})_*\ell$ and $\ell_{2}=(\mathsf{P}_{{\mathcal{Z}}_2})_*\ell$.  We assume that $\ell_{1}$ has a $C^1$-smooth density with respect to the Lebesgue measure on ${\mathcal{Z}}_1$.  The following result is due to\cite[Proposition 5.6]{Shi-15}. 
        
    \begin{lemma}{\rm(}Shirikyan {\rm\cite{Shi-15})} \label{Coupling lemma 3}
    In addition to the above settings, assume that $\Psi\colon {\mathcal{Z}}\rightarrow {\mathcal{Z}}$ is a mapping of the form $\Psi(\zeta)=\zeta+\Phi(\zeta)$, where $\Phi$ is a $C^1$-smooth mapping and the image of $\Phi$ is  contained in ${\mathcal{Z}}_1$. Suppose further that there is a constant $\varkappa>0$ such that 
    \begin{equation*}
    \|\Phi(\zeta_1)\|_{_{\mathcal{Z}}}\leq \varkappa,\quad \|\Phi(\zeta_1)-\Phi(\zeta_2)\|_{_{\mathcal{Z}}}\leq \varkappa\|\zeta_1-\zeta_2\|_{_{\mathcal{Z}}}
    \end{equation*}
    for any $\zeta_1,\zeta_2\in {\mathcal{Z}}$.
    Then there exists a constant $C>0$, not depending on $\varkappa$, such that
    \begin{equation*}
    \|\ell-\Psi_*\ell\|_{\rm TV}\leq C \varkappa.
    \end{equation*}
    \end{lemma}

    \subsubsection{Criterion for central limit theorems of stationary processes} In this appendix, we recall a central limit theorem criterion \cite[Corollary 1]{MW-00} for additive functionals of ergodic stationary Markov processes. For the reader's convenience, their key statements are summarized as follows. 

   \begin{theorem}{\rm(}Maxwell--Woodroofe {\rm\cite{MW-00})}
   Let $\{x_n;n\in\N\}$ be an ergodic stationary Markov process in a Polish space $\mathcal{X}$ with unique invariant measure $\mu_*$. Let $f\in B_b(\mathcal{X})$ be a function for which there exist constants $\beta<1$ and $C>0$ satisfying
   \begin{equation}\label{condition MW}
      \langle |\sum_{k=0}^{n-1}(P_kf-\langle f,\mu_*\rangle)|^2,\mu_*\rangle\leq Cn^{\beta}
   \end{equation}
   for any $n\in\N^+$. Then $\{f(x_n);n\in\N\}$ satisfies the central limit theorems in the following sense: 
    \begin{equation*}
        \frac{1}{\sqrt{n}}\sum_{k=0}^{n-1}(f(x_k)-\langle f,\mu_*\rangle)\rightarrow \mathcal{N}(0,\sigma_f^2)\quad\text{ as }n\rightarrow\infty,
    \end{equation*}
    where $\sigma_f^2\geq 0$ is given by     $\sigma_f^2=\lim\limits_{n\rightarrow\infty}\E\left(\frac{1}{\sqrt{n}}\sum_{k=0}^{n-1}(f(x_k)-\langle f,\mu_*\rangle)\right)^2$. 
  \end{theorem}

    \subsection{Optimal couplings}\label{Section-6-5}

    In this appendix, we summarize some basic notions and results surrounding the coupling approach. Let $(\mathcal{X},\|\cdot\|$) be a separable  Banach space and define a functional $\rho_{\boldsymbol{\varepsilon}}\colon\mathcal{X}\times\mathcal{X}\rightarrow[0,1]$, with $\boldsymbol{\varepsilon}=(\varepsilon_1,\varepsilon_2)$ and $0\leq \varepsilon_2\leq\varepsilon_1<\infty$, by the relation 
\begin{equation*}
\rho_{\boldsymbol{\varepsilon}}(x,x'):=  \varphi_{\boldsymbol{\varepsilon}}(\|x-x'\|),
\end{equation*}
where  $\varphi_{\boldsymbol{\varepsilon}}$ is defined as in (\ref{def-varphi}).
We further set
\begin{equation*}
\|\mu-\nu\|_{\boldsymbol{\varepsilon}}=\inf_{(\xi,\eta)\in\mathscr{C}(\mu,\nu)} \E \rho_{\boldsymbol{\varepsilon}}(\xi,\eta),\quad \mu,\nu\in\mathcal P(\mathcal X).
\end{equation*}
Kantorovich's theorem states that the infimum above can be always reached (see \cite[Theorem 5.10]{Villani-08}). That is, there exists a \textit{$\rho_{\boldsymbol{\varepsilon}}$-optimal coupling}  $(\xi_*,\eta_*)\in\mathscr{C}(\mu,\nu)$ such that
\begin{equation*}
\|\mu-\nu\|_{\boldsymbol{\varepsilon}}=\E \rho_{\boldsymbol{\varepsilon}}(\xi_*,\eta_*).
\end{equation*}

\begin{remark}
We list below some particular cases of $\rho_{\boldsymbol{\varepsilon}}$-optimal couplings.
\begin{itemize}[leftmargin=2em]
\item[\rm (1)] If $\varepsilon_1=\varepsilon_2=0$, then $\rho_{\boldsymbol{\varepsilon}}(x,x')=\mathbf{1}_{(0,\infty)}(\|x-x'\|)$. The $\rho_{\boldsymbol{\varepsilon}}$-optimal coupling is the usual maximal coupling of measures {\rm\cite{Thorisson-00}}.

\item[\rm (2)] If $\varepsilon_1=\varepsilon_2>0$, then $\rho_{\boldsymbol{\varepsilon}}(x,x')=\mathbf{1}_{(\varepsilon_1,\infty)}(\|x-x'\|)$. The $\rho_{\boldsymbol{\varepsilon}}$-optimal coupling is the concept of the $\varepsilon_1$-optimal coupling of measures {\rm\cite{Shi-15}}.

\item[\rm (3)] If $\varepsilon_1>\varepsilon_2=0$, then $\rho_{\boldsymbol{\varepsilon}}(x,x')=\min\{1,\|x-x'\|/\varepsilon_1\}$ is a continuous metric on $\mathcal{X}$. In this case, $\|\mu-\nu\|_{\boldsymbol{\varepsilon}}$ is the Wasserstein-1 distance between $\mu$ and $\nu$  associated with $\rho_{\boldsymbol{\varepsilon}}$ {\rm\cite{HM-06}}. In particular, it is equivalent to the dual-Lipschitz distance in the following sense
\begin{equation*}
\frac{\varepsilon_1}{1+\varepsilon_1} \|\mu-\nu\|_{\boldsymbol{\varepsilon}}\leq \|\mu-\nu\|_{L}^*\leq 2\|\mu-\nu\|_{\boldsymbol{\varepsilon}}.
\end{equation*}
\end{itemize}
\end{remark}
	
We now study the measurability of $\rho_{\boldsymbol{\varepsilon}}$-optimal couplings. Let $Z$ be a measurable space, and $\{\mu_i^z;z\in Z\},i=1,2$ be two families of probability measures on $\mathcal{X}$ such that the mappings $z\mapsto\mu_i^z$ are measurable from $Z$ to $\mathcal{P}(\mathcal{X})$. In addition, let $\lambda$ be a nonnegative measurable function on $Z$. 
\begin{proposition}\label{Coupling lemma 1}
Under the above settings, for every $0\leq\theta_1<\theta_2\leq 1$ there exists a probability space $(\Omega,\mathcal{F},\Pb)$ and measurable mappings $\mathcal{R},\mathcal{R}'\colon Z\times\Omega\rightarrow \mathcal{X}$ such that $(\mathcal{R}(z),\mathcal{R}'(z))\in\mathscr{C}(\mu_1^z,\mu_2^z)$ and 
\begin{equation}\label{optimal coupling measurablility}
\E \rho_{(\lambda(z),\theta_2\lambda(z))}(\mathcal{R}(z),\mathcal{R}'(z))\leq \|\mu_1^z-\mu_2^z\|_{(\lambda(z),\theta_1\lambda(z))}.
\end{equation}
\end{proposition}

\begin{proof}[{\bf Proof}] The proof of this proposition is analogous to that of \cite[Proposition 5.3]{Shi-15}. We split the measurable space $Z$ by $Z=\bigcup_{n\in\Z}Z_n$ with
\begin{align*}
    Z_n&=\left\{(n+1)^{-1}<\lambda(z)\leq n^{-1}\right\},\quad Z_{-n}=\left\{ n<\lambda(z)\leq n+1\right\}\quad\text{for }n\in\N^+,\\
    Z_0&=\{\lambda(z)=0\}.
\end{align*}

It suffices to construct the desired measurable couplings $\mathcal R,\mathcal R'$ on these disjoint sets, while the conclusion of this proposition will be obtained by a standard gluing procedure. 

For $z\in Z_0$, we can take $(\mathcal{R}_0(z),\mathcal{R}'_0(z))$ to be the usual maximal couplings on $(\Omega_0,\mathcal{F}_0,\Pb_0)$ for $\mu_1^z$ and $\mu_2^z$ for which 
(\ref{optimal coupling measurablility}) is satisfied; e.g., one can employ similar arguments as in \cite[Lemma 1]{KS-15}. For $z\in Z_n$ with $n\neq 0$, let us define the stretched measures $\tilde{\mu}_i^z$ by setting
$$
\tilde{\mu}_i^z(A)=\mu_i^z(\lambda(z)A),\quad A\in \mathcal{B}(\mathcal{X}).
$$
Then, an application of \cite[Corollary 5.22]{Villani-08} yields that there exists a 
$\rho_{(1,\theta_1)}$-optimal coupling $(\tilde{\xi}_*^z,\tilde{\eta}_*^z)$ for $\tilde{\mu}_1^z$ and $\tilde{\mu}_2^z$, defined on a common probability space $({\Omega}_n,\mathcal{F}_n,\Pb_n)$, such that the mapping  $z\mapsto (\tilde{\xi}_*^z,\tilde{\eta}_*^z)$  is measurable. In particular, it follows that
$$
\E \rho_{(1,\theta_2)}(\tilde{\xi}_*^z,\tilde{\eta}_*^z)\leq \E \rho_{(1,\theta_1)}(\tilde{\xi}_*^z,\tilde{\eta}_*^z)=
\|\tilde{\mu}_1^z-\tilde{\mu}_2^z\|_{(1,\theta_1)}.
$$
Thus, letting
$$
(\mathcal{R}_n(z)(\omega_n),\mathcal{R}'_n(z)(\omega_n))=\lambda(z)(\tilde{\xi}_*^z(\omega_n),\tilde{\eta}_*^z(\omega_n)),\quad z\in Z_n,\,\omega_n\in{\Omega}_n,
$$
it can be derived that $(\mathcal{R}_n(z),\mathcal{R}'_n(z))\in\mathscr{C}(\mu_1^z,\mu_2^z)$. Moreover, let us note that
\begin{align*}
    \|\mu_1^z-\mu_2^z\|_{(\lambda(z),\theta_1\lambda(z))}&=\inf_{(\xi,\eta)\in\mathscr C(\mu_1^z,\mu_2^z)}\E\rho_{(\lambda(z),\theta_1\lambda(z))}(\xi,\eta)\\
    &= \inf_{(\xi,\eta)\in\mathscr C(\mu_1^z,\mu_2^z)}\E\rho_{(1,\theta_1)}(\lambda(z)^{-1}\xi,\lambda(z)^{-1}\eta)\\
    &\geq \|\tilde{\mu}_1^z-\tilde{\mu}_2^z\|_{(1,\theta_1)}.
\end{align*}

Finally, let $(\Omega,\mathcal{F},\Pb)$ be the product space of $\{({\Omega}_n,\mathcal{F}_n,\Pb_n);n\in\Z\}$, and set 
$$
(\mathcal{R}(z)(\omega),\mathcal{R}'(z)(\omega))=(\mathcal{R}_n(z)(\omega_n),\mathcal{R}'_n(z)(\omega_n))\quad\text{ for } z\in Z_n,\,\omega=\{\omega_n;n\in\Z\}.
$$
By this construction, $(\mathcal{R}(z),\mathcal{R}'(z))\in\mathscr C(\mu_1^z,\mu_2^z)$, $\mathcal{R},\mathcal{R}'$ are measurable, and inequality (\ref{optimal coupling measurablility}) holds.
The proof is then complete.
\end{proof}

Next, we recall a lemma that could translate the issue of coupling hypothesis ({\bf C}) to a squeezing problem for controlled system.   Let $U_1, U_2$ be two $\mathcal{X}$-valued random variables defined on a probability space $(\mathcal{Z},\mathcal{B},\ell)$. Their laws are denoted by $\mu_1,\mu_2\in \mathcal{P}(\mathcal{X})$, respectively.  

\begin{lemma}\label{Coupling lemma 2}
Let $\boldsymbol{\varepsilon}=(\varepsilon_1,\varepsilon_2)$ with  $\varepsilon_1\geq \varepsilon_2\geq 0$. Assume that there exists a measurable mapping $\Psi\colon \mathcal{Z}\rightarrow \mathcal{Z}$ such that
\begin{equation*}
\|U_1(\zeta)-U_2(\Psi(\zeta))\|\leq \varepsilon_2
\end{equation*}
for almost every $\zeta\in \mathcal{Z}$. Then it follows that
\begin{equation*}
\|\mu_1-\mu_2\|_{\boldsymbol{\varepsilon}}\leq 2 \|\ell-\Psi_*\ell\|_{\rm TV}.
\end{equation*}
\end{lemma}

This lemma could be proved by following a
similar argument as in \cite[Proposition 5.2]{Shi-15}. So, we skip it.

\subsection{Proof of Proposition \ref{local exponential mixing}}\label{Proof of Prop mixing}
Below we present a detailed proof of Proposition~\ref{local exponential mixing}. The proof is based on an application of Theorem \ref{Theorem-KS}, which includes the verification of recurrence and squeezing properties for an appropriately constructed extension, consisting of three steps.

\medskip

\noindent {\bf Step 1 (Extension construction).}  Letting $\delta\in(0,1]$ be a small constant to be specified later, we introduce the diagonal set in $\boldsymbol{Y}_{\infty}$ by 
\begin{equation*}
\boldsymbol{\mathcal{D}}_\delta:=\{(x,x')\in\boldsymbol{Y}_{\infty};d(x,x')\leq \delta\}.
\end{equation*}
Then, let us define a coupling operator on $\boldsymbol{Y}_{\infty}$ by the relation
\begin{equation}\label{extension kernel}
\boldsymbol{{R}}(x,x'):=\begin{cases}
(\mathcal{R}(x,x'),\mathcal{R}'(x,x'))\quad&\text{for }(x,x')\in \boldsymbol{\mathcal{D}}_\delta,\\
(S(x,\xi),S(x',\xi'))&\text{otherwise},
\end{cases}
\end{equation}
where $\xi$ and $\xi'$ are independent copies of $\xi_0$. Without loss of generality, we may assume that $\xi,\xi',\mathcal{R},\mathcal{R}'$ are all defined on the same probability space. To emphasize the dependence on $\omega$, we will sometimes write $\boldsymbol{{R}}(x,x')$ as $\boldsymbol{{R}}(x,x',\omega)$.
		
Let $\{({\Omega}_n,{\mathcal{F}}_n,{\Pb}_n);n\in\N\}$ be a sequence of copies of the probability space on which $\boldsymbol{{R}}$ is defined. Let $(\boldsymbol{\Omega},\boldsymbol{\mathcal{F}},\mathbf{P})$ be the product of $\{({\Omega}_n,{\mathcal{F}}_n,{\Pb}_n);n\in\N\}$. For every $\vec{\boldsymbol{x}}=(x,x')\in \boldsymbol{Y}_\infty$ and $\boldsymbol{\omega}=\{\omega_n;n\in\N\}\in\boldsymbol{\Omega}$, we recursively define $\{\vec{\boldsymbol{x}}_n=(x_n,x_n');n\in\N\}$ by 
\begin{align*}
(x_{n+1}(\boldsymbol{\omega}),x_{n+1}'(\boldsymbol{\omega}))&=\boldsymbol{{R}}(x_{n},x_{n}',\omega_{n}),   \end{align*}
where $\vec{\boldsymbol{x}}_0=\vec{\boldsymbol{x}}=(x,x')$. By construction it follows that the laws of $x_n$ and $x'_n$ coincide with $P_n(x,\cdot)$ and  $P_n(x',\cdot)$, respectively. Thus, $\{\vec{\boldsymbol{x}}_n;n\in\N\}$ is an extension of $\{x_n;n\in\N\}$ with $x_0=x$.

\medskip

\noindent {\bf Step 2 (Verification of squeezing).} Without loss of generality, let us assume $g(\delta)<1$.
We proceed to show that the squeezing property (\ref{KS squeezing1}),(\ref{KS squeezing2}) holds for  $\boldsymbol{B}=\boldsymbol{\mathcal{D}}_\delta$ and $\boldsymbol{\sigma}=\boldsymbol{\sigma}_\delta$, where
\begin{equation*}
\boldsymbol{\sigma}_\delta:=\inf\{n\in\N;d(x_n,x_n')> r^n \delta\}.
\end{equation*}
Here, the constant $r\in[0,1)$ is given by (\ref{squeezing}). 

Let us fix any $\vec{\boldsymbol{x}}=(x,x')\in\boldsymbol{\mathcal{D}}_\delta$. In view of (\ref{extension kernel}), it follows that
\begin{equation}\label{one step squeezing}
\mathbf P_{\vec{\boldsymbol{x}}}(d(x_1,x_1')\leq rd(x,x'))\geq 1-g(d(x,x')).
\end{equation}
Then, let us define a sequence of decreasing sets 
\begin{equation*}
\boldsymbol{\Omega}_n=\left\{\boldsymbol{\omega}\in\boldsymbol{\Omega};d(x_{k+1},x_{k+1}')\leq rd(x_{k},x_{k}')\text{ for } 0\leq k\leq n\right\},\quad n\in\N.
\end{equation*}
Using inequality (\ref{one step squeezing}) and the Markov property, we obtain 
\begin{equation*}
\begin{aligned}
\mathbf P_{\vec{\boldsymbol{x}}}(\boldsymbol{\Omega}_{n+1})&=\mathbf E_{\vec{\boldsymbol{x}}}[\mathbf{1}_{\boldsymbol{\Omega}_{n}}(\mathbf P_{\vec{\boldsymbol{x}}}(d(x_{n+1},x_{n+1}')\leq rd(x_{n},x_{n}'))|\boldsymbol{\mathcal{F}}_{n})]\\
&=\mathbf E_{\vec{\boldsymbol{x}}}[\mathbf{1}_{\boldsymbol{\Omega}_{n}}\mathbf P_{\vec{\boldsymbol{x}}_{n}}(d(x_1,x_1')\leq rd(x_0,x_0'))]\\
&\geq \mathbf E_{\vec{\boldsymbol{x}}}[\mathbf{1}_{\boldsymbol{\Omega}_{n}}(1-g(d(x_{n},x_{n}')))]\\
&\geq (1-g(r^{n}d(x,x'))) \mathbf P_{\vec{\boldsymbol{x}}}(\boldsymbol{\Omega}_{n}),
\end{aligned}
\end{equation*}
where the last inequality is due to $d(x_{n},x_{n}')\leq r^{n}d(x,x')$ on $\boldsymbol{\Omega}_{n}$, as well as the increasing property of  $g$. Here,  $\boldsymbol{\mathcal{F}}_n$ denotes the natural filtration of the sequence $\{\vec{\boldsymbol{x}}_n;n\in\N\}$.  By iteration, we get that
\begin{equation*}
\mathbf P_{\vec{\boldsymbol{x}}}(\boldsymbol{\Omega}_{n})\geq \prod_{k=0}^{n}(1- g(r^kd(x,x'))) \geq \prod_{k\in\N}(1- g(r^kd(x,x'))) := G(d(x,x')).
\end{equation*}
Clearly, the function $G$ is decreasing and continuous on $[0,\delta]$ with $G(0)=1$. Moreover, one has
\begin{equation*}
\{\boldsymbol{\sigma}_\delta=\infty\}\supset \left\{d(x_{n+1},x_{n+1}')\leq rd(x_{n},x_{n}')\text{ for all } n\in\N\right\} =\bigcap_{n\in\N}\boldsymbol{\Omega}_n.
\end{equation*}
In conclusion, taking $0<\delta\leq 1$ sufficiently small so that $G(\delta)\geq 1/2$, there holds 
\begin{equation}\label{sigma=infty}
\mathbf P_{\vec{\boldsymbol{x}}}(\boldsymbol{\sigma}_\delta=\infty)\geq 1/2.
\end{equation}
Therefore, \eqref{KS squeezing1} is obtained.

At the same time, let us note that $\{\boldsymbol{\sigma}_\delta=n\}=\{\boldsymbol{\sigma}_\delta>n-1\}\cap\{d(x_n,x_n')>r^n\delta\}$, and $d(x_n,x_n')\leq r^n\delta$ on the set $\{\boldsymbol{\sigma}_\delta>n\}$.
Combined with the Markov property and  (\ref{one step squeezing}), these observations imply that for any $n\in\N^+$,
\begin{equation*}
\begin{aligned}
\mathbf P_{\vec{\boldsymbol{x}}}(\boldsymbol{\sigma}_\delta=n)&=\mathbf E_{\vec{\boldsymbol{x}}}[\mathbf{1}_{\{\boldsymbol{\sigma}_\delta>n-1\}} (\mathbf P_{\vec{\boldsymbol{x}}}(d(x_n,x_n')>r^n\delta)|\boldsymbol{\mathcal{F}}_{n-1})]\\
&=\mathbf E_{\vec{\boldsymbol{x}}}[\mathbf{1}_{\{\boldsymbol{\sigma}_\delta>n-1\}} \mathbf P_{\vec{\boldsymbol{x}}_{n-1}}(d(x_1,x_1')>r^n\delta)]\\
&\leq \mathbf E_{\vec{\boldsymbol{x}}}[\mathbf{1}_{\{\boldsymbol{\sigma}_\delta>n-1\}} \mathbf P_{\vec{\boldsymbol{x}}_{n-1}}(d(x_1,x_1')>rd(x_0,x_0'))]\\
&\leq g(r^{n-1}).
\end{aligned}
\end{equation*}

Then, taking (\ref{g-def}) into account, it follows that
\begin{align*}
\mathbf E_{\vec{\boldsymbol{x}}}(\mathbf{1}_{\{\boldsymbol{\sigma}_\delta<\infty\}}\exp(\beta_3\boldsymbol{\sigma}_\delta))&=\sum_{n\in\N}\exp(\beta_3 n)\mathbf P_{\vec{\boldsymbol{x}}}(\boldsymbol{\sigma}_\delta=n) \leq 1+\sum_{n\in\N^+}e^{\beta_3 n}g(r^{n-1})<\infty,
\end{align*}
where we take $\beta_3\in (0,-\limsup\limits_{n\rightarrow\infty}\tfrac{1}{n}\ln g(r^n))$. Inequality \eqref{KS squeezing2} thus follows.

\medskip

\noindent {\bf Step 3 (Verification of recurrence).}
It remains to verify the recurrence property (\ref{KS recurrence}) for the Markov process $\{\vec{\boldsymbol{x}}_n;n\in\N\}$, where the hitting time $\boldsymbol{\tau}$ is taken as
$$
\boldsymbol{\tau}_\delta=\inf\{n\in\N;\vec{\boldsymbol{x}}_n\in\boldsymbol{\mathcal{D}}_\delta\}.
$$ 
To this end, it suffices to show that there exists $m\in\N^+$ satisfying
\begin{equation}\label{p def}
p:=\inf_{\vec{\boldsymbol{x}}\in\boldsymbol{Y}_\infty}\mathbf P_{\vec{\boldsymbol{x}}}(\vec{\boldsymbol{x}}_m\in\boldsymbol{\mathcal{D}}_\delta)>0.
\end{equation}
Indeed, if (\ref{p def}) is true, the Markov property implies that 
\begin{equation*}
\begin{aligned}
\mathbf P_{\vec{\boldsymbol{x}}}(\boldsymbol{\tau}_\delta> km)&=\mathbf E_{\vec{\boldsymbol{x}}}[\mathbf E_{\vec{\boldsymbol{x}}}\mathbf{1}_{\{\boldsymbol{\tau}_\delta>km\}}|\boldsymbol{\mathcal{F}}_{(k-1)m}]\\
&=\mathbf E_{\vec{\boldsymbol{x}}}[\mathbf{1}_{\{\boldsymbol{\tau}_\delta>(k-1)m\}}\mathbf P_{\vec{\boldsymbol{x}}_{(k-1)m}}(\boldsymbol{\tau}_\delta>m)]\\
&\leq (1-p) \mathbf P_{\vec{\boldsymbol{x}}}(\boldsymbol{\tau}_\delta> (k-1)m)
\end{aligned}
\end{equation*}
for any $k\in\N^+$.
By iteration, it follows that 
$$        \sup_{\vec{\boldsymbol{x}}\in\boldsymbol{Y}_\infty}\mathbf P_{\vec{\boldsymbol{x}}}(\boldsymbol{\tau}_\delta> km)\leq (1-p)^k.
$$
This immediately implies that $\boldsymbol{\tau}_\delta<\infty$ almost surely by using the Borel--Cantelli lemma, and leads to (\ref{KS recurrence}) by taking $0<\beta_1<m^{-1}\ln(1-p)^{-1}$.  
		
To prove (\ref{p def}), denoting $\Delta_n=\{\boldsymbol{\omega}\in\boldsymbol{\Omega};\boldsymbol{\tau}_\delta\geq n\}$ for $n\in\N$, we have
\begin{equation}\label{delta_n partition}
\begin{aligned}
\mathbf P_{\vec{\boldsymbol{x}}}(\vec{\boldsymbol{x}}_n\in\boldsymbol{\mathcal{D}}_\delta)&= \mathbf P_{\vec{\boldsymbol{x}}}(\{\vec{\boldsymbol{x}}_n\in\boldsymbol{\mathcal{D}}_\delta\}\cap \Delta_n)+\mathbf P_{\vec{\boldsymbol{x}}}(\{\vec{\boldsymbol{x}}_n\in\boldsymbol{\mathcal{D}}_\delta\}\cap \Delta_n^c).
\end{aligned}
\end{equation}
For $\varepsilon>0$ and $n\in\N$, let us define 
$$\boldsymbol{A}^{n,\varepsilon}=\{\vec{\boldsymbol{x}}\in\boldsymbol{Y}_{\infty};\mathbf P_{\vec{\boldsymbol{x}}}(\Delta_n^c)>\varepsilon\}.$$
We consider first the case where $\vec{\boldsymbol{x}}\in\boldsymbol{A}^{n,\varepsilon}$. Recall (\ref{sigma=infty}) and observe that $\vec{\boldsymbol{x}}_{\boldsymbol{\tau}_\delta}\in\boldsymbol{\mathcal{D}}_{\delta}$ and
$$
\bigcap_{k\in\N}\{\vec{\boldsymbol{x}}_k\in\boldsymbol{\mathcal{D}}_{\delta}\}\supset\{\boldsymbol{\sigma}_\delta=\infty\}\quad\text{for }\vec{\boldsymbol{x}}_0=\vec{\boldsymbol{x}}\in\boldsymbol{\mathcal{D}}_{\delta}.
$$
Then,
one can employ the strong Markov property to infer that
\begin{equation}\label{delta_n^c}
\begin{aligned}
\mathbf P_{\vec{\boldsymbol{x}}}(\{\vec{\boldsymbol{x}}_n\in\boldsymbol{\mathcal{D}}_\delta\}\cap \Delta_n^c)&=\mathbf E_{\vec{\boldsymbol{x}}}[\mathbf E_{\vec{\boldsymbol{x}}}(\mathbf{1}_{\{\vec{\boldsymbol{x}}_n\in\boldsymbol{\mathcal{D}}_\delta\}}\mathbf{1}_{\{\boldsymbol{\tau}_\delta<n\}}|\boldsymbol{\mathcal{F}}_{\boldsymbol{\tau}_\delta})]\\
&=\mathbf E_{\vec{\boldsymbol{x}}}[\mathbf{1}_{\{\boldsymbol{\tau}_\delta<n\}}\mathbf P_{\vec{\boldsymbol{x}}_{\boldsymbol{\tau}_\delta}}(\vec{\boldsymbol{x}}_{k}\in\boldsymbol{\mathcal{D}}_\delta)|_{k=n-\boldsymbol{\tau}_\delta}]\\
&\geq \mathbf P_{\vec{\boldsymbol{x}}}(\boldsymbol{\tau}_{\delta}<n)\cdot \inf_{\boldsymbol{y}\in\boldsymbol{\mathcal{D}}_{\delta}}\mathbf P_{\boldsymbol{y}}(\boldsymbol{\sigma}_{\delta}=\infty)\\
&\geq \frac{1}{2} \mathbf P_{\vec{\boldsymbol{x}}}(\Delta_n^c).
\end{aligned}
\end{equation}
Thus plugging (\ref{delta_n^c}) into (\ref{delta_n partition}), it can be seen that 
\begin{equation}\label{S^n,kappa}
\mathbf P_{\vec{\boldsymbol{x}}}(\vec{\boldsymbol{x}}_n\in\boldsymbol{\mathcal{D}}_\delta)\geq \frac{\varepsilon}{2}.
\end{equation}

For the other case, i.e.,  $\vec{\boldsymbol{x}}\in\boldsymbol{Y}_\infty\setminus \boldsymbol{A}^{n,\varepsilon}$, we derive that 
\begin{equation}\label{x_n|delta_n1}
\begin{aligned}
\mathbf P_{\vec{\boldsymbol{x}}}(\{\vec{\boldsymbol{x}}_n\in\boldsymbol{\mathcal{D}}_\delta\}\cap\Delta_n)=\mathbf P_{\vec{\boldsymbol{x}}}(\vec{\boldsymbol{x}}_n\in\boldsymbol{\mathcal{D}}_\delta|\Delta_n)\mathbf P_{\vec{\boldsymbol{x}}}(\Delta_n)\geq (1-\varepsilon)\mathbf P_{\vec{\boldsymbol{x}}}(\vec{\boldsymbol{x}}_n\in\boldsymbol{\mathcal{D}}_\delta|\Delta_n).
\end{aligned}
\end{equation}
Below is to estimate $\mathbf P_{\vec{\boldsymbol{x}}}(\vec{\boldsymbol{x}}_n\in\boldsymbol{\mathcal{D}}_\delta|\Delta_n)$ for appropriately chosen $n$ and $\varepsilon$. In view of the construction of $\{\vec{\boldsymbol{x}}_n;n\in\N\}$, one can check that $x_n$ and $x'_n$ are independent on $\Delta_n$. This enables us to see that
\begin{equation}\label{x_n|delta_n2}
\begin{aligned}
\mathbf P_{\vec{\boldsymbol{x}}}(\vec{\boldsymbol{x}}_n\in\boldsymbol{\mathcal{D}}_\delta|\Delta_n)&\geq \mathbf P_{\vec{\boldsymbol{x}}}(\vec{\boldsymbol{x}}_n\in B(z,\delta/2)\times B(z,\delta/2)|\Delta_n)\\
&=\mathbf P_{\vec{\boldsymbol{x}}}(x_n\in B(z,\delta/2)|\Delta_n)\cdot\mathbf P_{\vec{\boldsymbol{x}}}({x}'_n\in B(z,\delta/2)|\Delta_n),
\end{aligned}
\end{equation}
where the point $z\in\mathcal X$ is given by hypothesis ({\bf I}). Making use of
(\ref{irreducibility}), there exists $m\in\N^+$ and $p'>0$ such that
\begin{equation*}
\mathbf P_{\vec{\boldsymbol{x}}}(x_m\in B(z,\delta/2))\geq p'
\end{equation*}
for any $x\in\mathcal{Y}_\infty$. As a consequence,
\begin{equation*}
p'\leq \mathbf P_{\vec{\boldsymbol{x}}}(x_m\in B(z,\delta/2))\leq \mathbf P_{\vec{\boldsymbol{x}}}(x_m\in B(z,\delta/2)|\Delta_m) + \mathbf P_{\vec{\boldsymbol{x}}}(\Delta_m^c). 
\end{equation*}
It then follows that
\begin{equation*}
\begin{aligned}
\mathbf P_{\vec{\boldsymbol{x}}}(x_m\in B(z,\delta/2)|\Delta_m)\geq \frac{p'}{2},
\end{aligned}
\end{equation*}
and similarly,
\begin{equation*}
\begin{aligned}
\mathbf P_{\vec{\boldsymbol{x}}}(x_m'\in B(z,\delta/2)|\Delta_m)\geq \frac{p'}{2}
\end{aligned}
\end{equation*}
for any $\vec{\boldsymbol{x}}\in\boldsymbol{Y}_\infty\setminus \boldsymbol{A}^{m,p'/2}$. Therefore, taking $n=m$ and $\varepsilon=p'/2$ in (\ref{x_n|delta_n1}),(\ref{x_n|delta_n2}), we conclude that
\begin{equation}\label{S^n,kappa-1}
\mathbf P_{\vec{\boldsymbol{x}}}(\vec{\boldsymbol{x}}_m\in\boldsymbol{\mathcal{D}}_\delta)\geq \left(1-\frac{p'}{2}\right)\frac{(p')^2}{4}
\end{equation}
for any $\vec{\boldsymbol{x}}\in\boldsymbol{Y}_\infty\setminus \boldsymbol{A}^{m,p'/2}$.

Finally, the claim (\ref{p def}) follows from the combination of (\ref{S^n,kappa-1}) and (\ref{S^n,kappa}) (with $n=m$ and $\varepsilon=p'/2$).
The proof is then complete.

\subsection{Proof of Proposition \ref{Limit theorems}} \label{Proof of Thm limits}  The proof consists of two parts, separately.

\medskip

\noindent {\bf Part 1 (Strong law of large numbers).} We use a martingale decomposition procedure developed in \cite{Shi-06,KW-12} to derive the strong law of large numbers. Let $f\in L_b(\mathcal{X})$ and $x\in\mathcal{X}$ be fixed. With no loss of generality, assume that $\langle f,\mu_*\rangle=0$. Let us define the corrector that will be used in the martingale approximation procedure by
            \begin{equation*}
                \phi(x)=\sum_{k\in\N}P_kf(x),
            \end{equation*}
            where the convergence of the series is ensured by (\ref{Exponential mixing}). Indeed, it follows that 
            \begin{equation*}
                |\phi(x)|\leq C\|f\|_{L}(1+V(x)).
            \end{equation*}
            for some constant $C>0$, not depending on $f$ and $x$.  In view of (\ref{xn asymptotic compactness}), $\{\phi(x_n);n\in\N\}$ is almost surely uniformly bounded.   We are now in a position to give the martingale approximation. For $n\in\N^+$, let
         \begin{equation*}
             \sum_{k=0}^{n-1}f(x_k)=M_n+N_n
         \end{equation*}
        with 
        \begin{equation*}
            M_n:=\phi(x_n)-\phi(x)+\sum_{k=0}^{n-1}f(x_k)\quad\text{and}\quad N_n:=\phi(x)-\phi(x_n).
        \end{equation*}
        
       Clearly, the uniform boundedness of $\phi(x_n)$ allows us to conclude that
        \begin{equation*}
           \lim\limits_{n\rightarrow\infty}{n^{-1}}N_n=0\quad\text{ almost surely}.
        \end{equation*}
        
        Thus, it remains to handle the martingale part. Indeed, one can easily check that $\{M_n;n\in\N^+\}$ is a zero-mean square-integrable martingale, and thus the standard strong law of large numbers for discrete-time martingales, see, e.g., \cite[Theorem A.12.1]{KS-12}, implies the desired results.
             
    \medskip
    
    \noindent {\bf Part 2 (Central limit theorems).}  The proof of the central limit theorems consists of two steps. We shall first prove it for the ergodic stationary Markov process $\{x_n^*;n\in\N\}$, where $x_n^*$ is defined by 
    $$
    x_{n+1}^*=S(x_{n}^*,\xi_{n}),\;n\in\N\quad\text{and}\quad x_0^*=x^*.
    $$
    Here $x^*$ is an $\mathcal{X}$-valued random variable with law $\mu_*$, and is independent of $(\xi_n;n\in\N)$. Then, in the next step, we extend to the general case.  Let  $f\in L_b(\mathcal{X})$ be arbitrarily fixed. 

    \medskip
     
    \noindent {\bf Step 2.1  (The stationary case).} Invoking exponential mixing \eqref{Exponential mixing}, one can calculate that for there exists constant $C>0$ such that
    \begin{align*}
        |\sum_{k=0}^{n-1}(P_kf(x)-\langle f,\mu_*)\rangle|\leq C(1+V(x))\|f\|_{L}
    \end{align*}   
    for any $x\in\mathcal{X}$. 
   In view of the fact that $\text{supp }\mu_*\subset\mathcal{Y}_{\infty}$, one gets
    \begin{align*}
        \langle |\sum_{k=0}^{n-1}(P_kf-\langle f,\mu_*\rangle)|^2,\mu_*\rangle\leq (C\sup_{x\in\mathcal{Y}_{\infty}}(1+V(x))\|f\|_L)^2
    \end{align*}
    for any $n\in\N^+$. As the above estimation is independent of $n$, condition (\ref{condition MW}) is satisfied with $\beta=0$. Thus, the central limit theorems for $\{f(x^*_n);n\in\N\}$ follows.

    \medskip

    \noindent {\bf Step 2.2 (The general case).} It remains to handle the general case with $\{x_n;n\in\N\}$ defined by \eqref{RDS},\eqref{initial condition}. For any $x\in\mathcal{X}$, to indicate the initial condition,  let us write 
    $$s_n^x(f)=\frac{1}{\sqrt{n}}\sum_{k=0}^{n-1}(f(S_k(x;\boldsymbol{\xi}))-\langle f,\mu_*\rangle).$$
    We also use the corresponding notation $ s_n^{*}(f)$ for the stationary process $\{x_n^*;n\in\N\}$. Form the previous step, we have known that
    \begin{equation*}
        s_n^{*}(f)\rightarrow\mathcal{N}(0,\sigma_f^2)\quad\text{ as }n\rightarrow\infty,
    \end{equation*}
    with
    \begin{equation*}
      \sigma_f^2=\lim\limits_{n\rightarrow\infty}\E_{\mu_*}\left(\frac{1}{\sqrt{n}}\sum_{k=0}^{n-1}(f(x_k)-\langle f,\mu_*\rangle)\right)^2.
    \end{equation*}      
    Here the notation $\E_{\mu_*}$ stands for the expectation corresponding to the invariant measure:
    $$
    \E_{\mu_*}(\cdot)=\int_{\mathcal{X}} \E_{x}(\cdot)\mu_*(dx).
    $$    
    
    Equivalently, it means that for any $F\in L_b(\R)$ with $\|F\|_{L}\leq 1$,
    \begin{equation}\label{Snmu-N}
       \lim\limits_{n\rightarrow\infty} \langle F,\mathscr{D}(s_n^{*}(f))\rangle=\lim\limits_{n\rightarrow\infty} \E_{\mu_*} F(s_n^{*}(f))= \langle F,\mathscr{D}(\mathcal{N}(0,\sigma_f^2))\rangle.
    \end{equation}\par
    On the other hand, again using exponential mixing \eqref{Exponential mixing}, one gets
        \begin{align*}
            |\langle F, \mathscr{D}(s_n^x(f))\rangle -\langle F, \mathscr{D}(s_n^{x'}(f))\rangle|\leq Cn^{-1/2}(1+V(x))\|f\|_{L}
        \end{align*}
    for any $x\in\mathcal{X}$ and $x'\in\mathcal{Y}_\infty$ with a universal constant $C>0$. Thus, it further yields that
     \begin{equation}\label{Snx-Snmu}
            |\langle F, \mathscr{D}(s_n^x(f))\rangle -\langle F, \mathscr{D}(s_n^{*}(f))\rangle|\leq Cn^{-1/2}(1+V(x))\|f\|_{L}.
    \end{equation}
    
   Consequently, collecting (\ref{Snmu-N}),(\ref{Snx-Snmu}), the proof is completed by 
    \begin{equation*}
        \lim\limits_{n\rightarrow\infty}\langle F, \mathscr{D}(s_n^x(f))\rangle =\langle F, \mathscr{D}(\mathcal{N}(0,\sigma_f^2)\rangle.
    \end{equation*}

\section{Auxiliary demonstrations for control problems}\label{Appendix-control}

In this appendix, we shall supplement the proofs of the intermediate result, i.e. Proposition \ref{prop-obs}, which has been taken for granted in establishing Proposition \ref{Prop-linear}. In addition, the deduction of squeezing property via contractibility will be presented in detail, so we complete rigorously the proof of Theorem \ref{Th1}.

\subsection{Proof of Proposition \ref{prop-obs}(1)}\label{Appendix-obs}

We argue by contradiction. 
Assume that for every $n\in\N^+$, there exists $\hat u^n\in B_R$ and $\varphi^T_n\in \mathcal H^{\scriptscriptstyle-6/5}$ such that
\begin{align}
	&  \|\varphi^T_n\|_{_{\mathcal  H^{\scriptscriptstyle-6/5}}}=1,\label{Hypothesis-2} \\
	& \int_0^T\|\chi\varphi^n(t)\|^2_{_{H^{\scriptscriptstyle-1/5}}}dt\leq \frac{1}{n}\quad {\rm with\ }\varphi^n=\mathcal W^T_{\hat u^n}(\varphi^T_n) \label{Hypothesis-3}.
\end{align}
In view of (\ref{Hypothesis-2}), one can use (\ref{energy-2}) to deduce that there exists a constant $C=C(T,R)>0$ such that
\begin{equation*}
\|\varphi^n[t]\|_{_{\mathcal H^{\scriptscriptstyle-6/5}}}\leq C
\end{equation*}
for all $n\in\N^+$ and $t\in[0,T]$. Accordingly, it follows that the sequence $\{\varphi^n;n\in\N^+\}$ is bounded in $L^\infty_tH^{\scriptscriptstyle-1/5}_x$, while $\{\partial_t\varphi^n;n\in\N^+\}$ is bounded in $L^\infty_tH_x^{\scriptscriptstyle-6/5}$. This together with the Aubin--Lions lemma implies that $\{\varphi^n;n\in\N^+\}$ is relatively compact in $C([0,T];H^{\scriptscriptstyle-6/5})$. Therefore, we conclude that up to a subsequence,
\begin{align}
	& \varphi^n\overset{\star}{\rightharpoonup}  \varphi^0\quad {\rm in\ }L^\infty_tH_x^{\scriptscriptstyle-1/5},\notag\\
	& \partial_t\varphi^n\overset{\star}{\rightharpoonup} \partial_t\varphi^0\quad {\rm in\ }L^\infty_tH^{\scriptscriptstyle-6/5}_x,\notag\\
	& \varphi^n\rightarrow \varphi^0\quad {\rm in\ }C([0,T];H^{\scriptscriptstyle-6/5}),\label{convergence-4}\\
	& \varphi^n[T]\rightharpoonup \psi=(\psi_0^T,\psi_1^T)\quad {\rm in\ }\mathcal H^{\scriptscriptstyle-6/5},\notag\\
	& 3(\hat u^n)^2\overset{\star}{\rightharpoonup} p \quad {\rm in\ }L^\infty(D_T)\cap L^\infty_tH_x^{\scriptscriptstyle11/7}\notag
\end{align}
as $n\rightarrow \infty$. The limiting function $\varphi^0$ is the solution of 
\begin{equation*}
\boxempty \varphi^0-a(x)\partial_t\varphi^0+p(t,x)\varphi^0=0,\quad
\varphi^0[T]=\psi.
\end{equation*}

Due to (\ref{convergence-4}), it follows that
$$
\chi\varphi^n\rightarrow \chi\varphi^0 \quad {\rm in\ }L^2_tH^{\scriptscriptstyle-6/5}_x,
$$
which together with (\ref{Hypothesis-3}) leads to $\chi\varphi^0\equiv 0$. What follows is to show that 
\begin{equation}\label{contradiction-1}
\varphi^0\equiv 0.
\end{equation}
For this purpose,
let $\vartheta\in C_0^\infty(\R)$ such that
$$
\vartheta(x)=1\quad {\rm for\ }|x|\leq 1,\quad \vartheta(x)=0\quad {\rm for\ }|x|\geq 2.
$$
We then introduce the cut-off operator 
$$
\vartheta(-\Delta)\phi=\sum_{j\in\N^+}\vartheta(\lambda_j)(\phi,e_j)e_j,\quad \phi\in H.
$$
It is not difficult to verify that the operator $\vartheta(-\Delta)$ is adjoint on each $H^s$. 

\begin{lemma}\label{Communicators-estimate}
Let $\vartheta\in C_0^\infty(\R)$. Then the following assertions hold.
\begin{enumerate}[leftmargin=2em]
\item[$(1)$] For every $f\in C^\infty(\overline{D})$\footnote{Given a $L^\infty$-function $f$, we use the same notation to denote the corresponding multiplication operator $\phi\mapsto f\phi$.}, there exists a constant $C_1=C_1(f)>0$ such that
\begin{equation}\label{communicator-1}
\|[\vartheta(-\varepsilon^2\Delta),f]\|_{_{\mathcal L(H^{\scriptscriptstyle-1/5};H)}}+ \|[\vartheta(-\varepsilon^2\Delta),f]\|_{_{\mathcal L(H^{\scriptscriptstyle-6/5};H^{-1})}}\leq C_1\varepsilon^{4/5}
\end{equation}
for any $\varepsilon\in(0,1)$.
	
\item[$(2)$] There exists a constant $C_2>0$ such that
\begin{equation}\label{communicator-2}
\|[\vartheta(-\varepsilon^2\Delta),f]\|_{_{\mathcal L(H^{\scriptscriptstyle-1/5};H^{-1})}}\leq C_2\varepsilon^{8/35}\|f\|_{_{H^{\scriptscriptstyle11/7}}}
\end{equation}
for any $f\in H^{\scriptscriptstyle11/7}$ and $\varepsilon\in(0,1)$.
\end{enumerate}
\end{lemma}

Taking this lemma for granted, we continue to prove (\ref{contradiction-1}). For $\varepsilon\in(0,1)$ we define $\varphi^{0,\varepsilon}$ to be the solution of 
\begin{equation*}
\boxempty \varphi^{0,\varepsilon}-a(x)\partial_t\varphi^{0,\varepsilon}+p(t,x)\varphi^{0,\varepsilon}=0,\quad 
\varphi^{0,\varepsilon}[T]=(\vartheta(-\varepsilon^2\Delta)\psi_0^T,\vartheta(-\varepsilon^2\Delta)\psi_1^T).
\end{equation*}
Making use of Lemma \ref{basic-observability}(1) (see also Remark \ref{remark-general potential}), it can be derived that
\begin{equation}\label{bound-25}
	\begin{aligned}
	\|\varphi^{0,\varepsilon}[T]\|_{_{\mathcal H^{-1}}}^2&\leq C\int_0^T\|\chi\varphi^{0,\varepsilon}(t)\|^2dt\\
	&\leq C\int_0^T\|\chi z^{0,\varepsilon}(t)\|^2dt+C\int_0^T\|\chi \vartheta(-\varepsilon^2\Delta)\varphi^0(t)\|^2dt,
	\end{aligned}
\end{equation}
where $z^{0,\varepsilon}=\varphi^{0,\varepsilon}-\vartheta(-\varepsilon^2\Delta)\varphi^0$. To deal with the first term in RHS of (\ref{bound-25}), let us note that
\begin{equation*}
\boxempty z^{0,\varepsilon}-a(x)\partial_tz^{0,\varepsilon}+p(t,x)z^{0,\varepsilon}=-[\vartheta,a]\partial_t\varphi^0+[\vartheta,p]\varphi^0,\quad z^{0,\varepsilon}[T]=(0,0).
\end{equation*}
This together with 
(\ref{communicator-1}),(\ref{communicator-2}) means that
\begin{align*}
	\|z^{0,\varepsilon}[t]\|_{_{\mathcal H^{-1}}}\leq & \ C\int_0^T\left[
	\|[\vartheta,a]\partial_t\varphi^0\|_{_{H^{-1}}}+\|[\vartheta,p]\varphi^0\|_{_{H^{-1}}}
	\right]dt \\
	\leq & \ C\int_0^T\left[
	\varepsilon^{4/5}\|\partial_t\varphi^0\|_{_{H^{\scriptscriptstyle-6/5}}}+\varepsilon^{8/35}\|\varphi^0\|_{_{H^{\scriptscriptstyle-1/5}}}
	\right]dt \\
	\leq & \ C\left(\varepsilon^{4/5}+\varepsilon^{8/35}\right)\|\psi\|_{_{\mathcal H^{\scriptscriptstyle-6/5}}}.
\end{align*}
Accordingly, 
\begin{equation}\label{bound-27}
	\int_0^T\|\chi z^{0,\varepsilon}(t)\|^2dt\leq C\left(\varepsilon^{4/5}+\varepsilon^{8/35}\right)^2\|\psi\|_{_{\mathcal H^{\scriptscriptstyle-6/5}}}^2.
\end{equation}		
At the same time, it follows from the fact $\chi\varphi^0=0$ that
$$
\chi \vartheta(-\varepsilon^2\Delta)\varphi^0=[\chi,\vartheta(-\varepsilon^2\Delta)]\varphi^0.
$$
Using (\ref{communicator-1}) and (\ref{energy-2}), we obtain that
\begin{align*}
\int_0^T\|\chi \vartheta(-\varepsilon^2\Delta)\varphi^0(t)\|^2dt\leq C\varepsilon^{8/5} \int_0^T\|\varphi^0(t)\|^2_{_{H^{\scriptscriptstyle-1/5}}}dt\leq C\varepsilon^{8/5}\|\psi\|_{_{\mathcal H^{\scriptscriptstyle-6/5}}}^2.
\end{align*}
Combined with (\ref{bound-25}) and (\ref{bound-27}), this yields that
\begin{equation*}
	\|\varphi^{0,\varepsilon}[T]\|_{_{\mathcal H^{-1}}}\leq C\left(\varepsilon^{4/5}+\varepsilon^{8/35}\right)\|\psi\|_{_{\mathcal H^{\scriptscriptstyle-6/5}}}.
\end{equation*}
As a consequence,
$$
\|(\vartheta(-\varepsilon^2\Delta)\psi_0^T,\vartheta(-\varepsilon^2\Delta)\psi_1^T)\|_{_{\mathcal H^{-1}}}\rightarrow 0
$$
as $\varepsilon\rightarrow 0^+$. In conclusion, $\psi_0^T=\psi_1^T=0$ which leads to (\ref{contradiction-1}).

In the sequel, we proceed to show that
\begin{equation}\label{contradiction-2}
	\int_0^T\|\varphi^0(t)\|^2dt>0,
\end{equation}
which contradicts (\ref{contradiction-1}).
To this end, let us mention that inequality (\ref{basic-case-2}) can be expressed via the adjoint group $U^*(t)$. In fact, when $\hat u\equiv 0$, any solution $\varphi$ of the adjoint system (\ref{adjoint-system-1}) satisfies $\varphi[t]=U^*(T-t)\varphi^T$. Therefore, we rewrite (\ref{basic-case-2}) as
$$
\int_0^T\|\chi (U^*_1(t)\varphi^T)\|^2_{_{H^{\scriptscriptstyle-1/5}}}dt\geq C\|\varphi^T\|^2_{_{\mathcal  H^{\scriptscriptstyle-6/5}}},
$$
where $(U^*_1(t),U^*_2(t))=U^*(t)$.
This together with the reversibility of $U^*(t)$ implies that
\begin{equation}\label{basic-case-3}
\|U^*(t)\varphi^T\|^2_{_{\mathcal  H^{\scriptscriptstyle-6/5}}}\leq C\int_0^T\|\chi (U^*_1(s)\varphi^T)\|^2_{_{H^{\scriptscriptstyle-1/5}}}ds,
\end{equation}
for any $t\in[0,T]$.

At the same time, notice by (\ref{variation-formula-2}) that
\begin{equation*}
U^*(T-t)\varphi^T_n=\varphi^n[t]+\int_t^TU^*(s-t)\left(
\begin{matrix}
0\\
-3(\hat u^n)^2(s)\varphi^n(s)
\end{matrix}
\right)ds.
\end{equation*}
Then, one can apply (\ref{basic-case-3}) to deduce that 
\begin{equation}\label{bound-28}
	\begin{array}{ll}
		\displaystyle\|U^*(T-t)\varphi^T_n\|^2_{_{\mathcal H^{\scriptscriptstyle-6/5}}}\leq C\int_0^T\|\chi
		\varphi^n(s)\|^2_{_{H^{\scriptscriptstyle-1/5}}}ds+C\int_0^T\|\varphi^n(s)\|^2_{_{H^{\scriptscriptstyle-6/5}}}ds.
	\end{array}
\end{equation}
Moreover, it can be seen that
$$
{\rm LHS\ of\ (\ref{bound-28})}\geq C\|\varphi^T_n\|^2_{_{\mathcal H^{\scriptscriptstyle-6/5}}}
$$ 
This together with (\ref{Hypothesis-2}) implies that
\begin{equation*}
\begin{array}{ll}
\displaystyle 1\leq  C\int_0^T\|\chi
\varphi^n(t)\|^2_{_{H^{\scriptscriptstyle-1/5}}}dt+C\int_0^T\|\varphi^n(t)\|^2_{_{H^{\scriptscriptstyle-6/5}}}dt.
\end{array}
\end{equation*}
Letting $n\rightarrow \infty$ and taking (\ref{Hypothesis-3}),(\ref{convergence-4}) into account, we conclude that
$$
1\leq C\int_0^T\|\varphi^0(t)\|^2_{_{H^{\scriptscriptstyle-6/5}}}dt,
$$
which gives rise to (\ref{contradiction-2}). The proof of (\ref{full-observability}) is therefore complete.

\medskip

\begin{proof}[{\bf Proof of Lemma \ref{Communicators-estimate}}]
We only provide a proof of the second assertion, as the first can be derived by following the same arguments as in \cite[Section 2.3]{ADS-16}.

Notice that when $f\in H^{\scriptscriptstyle11/7}$ the multiplication operator $\phi\mapsto f\phi$ is bounded from $H^\alpha$ into itself for every $\alpha\in[0,11/7]$. This implies that
\begin{equation}\label{commute-1}
\|[\vartheta(-\varepsilon^2\Delta),f]\|_{_{\mathcal L(H^{\alpha};H^{\alpha})}}\leq C\|f\|_{_{H^{\scriptscriptstyle11/7}}}
\end{equation}
for any $\varepsilon\in(0,1)$. To continue, we obtain that
\begin{equation}\label{commute-2}
[\vartheta(-\varepsilon^2\Delta),f]\phi=\frac{1}{2\pi}\int_{\R}w(\varepsilon^2s)\hat\vartheta(s)ds
\end{equation}
for any $\phi\in H$,
where $\hat\vartheta$ is  the Fourier transform of $\vartheta$ and $w(s)=[e^{-is\Delta},f]\phi$. It then follows that
$$
\partial_sw=-i\Delta e^{-is\Delta}(f\phi)+if\cdot(\Delta e^{-is\Delta}\phi).
$$
Accordingly,
\begin{equation*}
\begin{aligned}
\|\partial_sw\|_{_{H^{\scriptscriptstyle-11/7}}}\leq C\|f\|_{_{H^{\scriptscriptstyle11/7}}}\|\phi\|_{_{H^{\scriptscriptstyle3/7}}},
\end{aligned}
\end{equation*}
provided that $\phi\in H^{\scriptscriptstyle3/7}$.
One thus sees that 
$$
\|w(s)\|_{_{H^{\scriptscriptstyle-11/7}}}\leq C|s|\|f\|_{_{H^{\scriptscriptstyle11/7}}}\|\phi\|_{_{H^{\scriptscriptstyle3/7}}}.
$$
Inserted into (\ref{commute-2}), this implies 
\begin{equation*}
\|[\vartheta(-\varepsilon^2\Delta),f]\|_{_{\mathcal L(H^{\scriptscriptstyle3/7};H^{\scriptscriptstyle-11/7})}}\leq C\varepsilon^2\|f\|_{_{H^{\scriptscriptstyle11/7}}}.
\end{equation*}
Interpolating it and (\ref{commute-1}) (with $\alpha=3/7$), we infer that
$$
\|[\vartheta(-\varepsilon^2\Delta),f]\|_{_{\mathcal L(H^{\scriptscriptstyle3/7};H^{\scriptscriptstyle3/7-\beta})}}\leq C\varepsilon^{\beta}\|f\|_{_{H^{\scriptscriptstyle11/7}}}
$$
for any $\beta\in[0,2]$. Taking $\beta=8/35$ and using the embedding $H^{1}\hookrightarrow H^{\scriptscriptstyle3/7}$, it follows that
$$
\|[\vartheta(-\varepsilon^2\Delta),f]\|_{_{\mathcal L(H^{1};H^{\scriptscriptstyle1/5})}}\leq C\varepsilon^{8/35}\|f\|_{_{H^{\scriptscriptstyle11/7}}}.
$$
Finally, the desired result is obtained by duality.
\end{proof}

\subsection{Proof of Theorem \ref{Th1}}\label{Appendix-squeezing}

For arbitrarily given  $\varepsilon\in(0,1)$ and $R>0$, we assume that $T=T_\varepsilon>0$ and $N=N(\varepsilon,T,R)\in\N^+$ are established in Proposition \ref{Prop-linear}.

Let $\hat u^0\in \mathcal H^{\scriptscriptstyle4/7}$ and $h\in L^2_tH_x^{\scriptscriptstyle4/7}$ such that $\hat u\in B_R$ with $\hat u[\cdot]=\mathcal S(\hat u^0,h)$. Next, we introduce the difference 
$
w=u-\hat u,
$
where $u[\cdot]=\mathcal S(u^0,h+\chi\mathscr P^{\scriptscriptstyle T}_{\scriptscriptstyle N}\zeta)$ with the initial state $u^0\in \mathcal H$ satisfying 
\begin{equation}\label{initial-value}
\|u^0-\hat u^0\|_{_{\mathcal H}}\leq 1,
\end{equation}
and the control $\zeta$ to be specified within the range of 
\begin{equation}\label{force-1}
\zeta\in \overline{B}_{L^2(D_T)}(1).
\end{equation}
Obviously, the controlled system for $w$ reads
\begin{equation}\label{equation-w}
    \begin{cases}
        \boxempty w+a(x)\partial_t w+(\hat u+w)^3-\hat u^3=\chi\mathscr P^{\scriptscriptstyle T}_{\scriptscriptstyle N}\zeta,\quad x\in D,\\
        w[0]=v^0:=u^0-\hat u^0.
    \end{cases}
\end{equation}
In addition, noticing (\ref{initial-value}),(\ref{force-1}), it can derived that there exists a constant $C=C(T,R)>0$ such that $\|u(t)\|_{_{H^1}}\leq C$ for any $t\in[0,T]$. This leads to
$$
\|(\hat u+w)^3(t)-\hat u^3(t)\|\leq C\|w(t)\|_{_{H^1}}.
$$
Therefore, one can multiply (\ref{equation-w}) by $\partial_t w$ and integrate over $D$ to deduce that
\begin{equation}\label{bound-w}
\|w[t]\|^2_{_{\mathcal H}}\leq C\left[
\|v^0\|_{_{\mathcal H}}^2+\int_0^T\|\zeta(t)\|^2dt
\right]
\end{equation}
for any $t\in[0,T]$.

On the other hand, an application of Proposition \ref{Prop-linear} yields that there exists a control $\zeta\in L^2(D_T)$ having the structure (\ref{structure-control}) and satisfying 
\begin{equation}\label{decay-0}
	\|v[T]\|_{_{\mathcal H}}\leq \frac{\varepsilon}{2}\|v^0\|_{_{\mathcal H}},\quad \int_0^T\|\zeta(t)\|^2dt\leq C\|v^0\|_{_{\mathcal H}}^2,
\end{equation}
where the constant $C$ depends on $T,R$, and $v=\mathcal{V}_{\hat u}(v^0,\chi\mathscr P^{\scriptscriptstyle T}_{\scriptscriptstyle N}\zeta)\in \mathcal X_T$ stands for the solution of (\ref{Problem-linearized}) with $v^0=u^0-\hat u^0$. In particular, due to the second inequality in (\ref{decay-0}), there exists a sufficiently small $d_0=d_0(T,R)>0$ such that if $\|v^0\|_{_{\mathcal H}}\leq d$ with $d\in(0,d_0)$,  the conditions
(\ref{initial-value}),(\ref{force-1}) are satisfied. It also follows that
the difference $z:=w-v$ satisfies
\begin{equation}\label{Problem-1}
\begin{cases}
\boxempty z+a(x)\partial_t z+w^3+3\hat u w^2=-3\hat u^2z,\quad x\in D,\\
z[0]=(0,0).
\end{cases}
\end{equation}
Using (\ref{bound-w}) and the second inequality in (\ref{decay-0}), one gets
$$
\begin{array}{ll}
\displaystyle\|w^3+3\hat u w^2\|^2\leq  C \left[
\|v^0\|_{_{\mathcal H}}^6+\|v^0\|_{_{\mathcal H}}^4
\right].
\end{array}
$$
Therefore,
by multiplying (\ref{Problem-1}) by $\partial_t z$ and integrating over $D$, we obtain
$$
\frac{d}{dt}
\left[
\|z\|^2_{_{H^1}}+\|\partial_tz\|^2
\right]\leq C\left[
\|z\|^2+\|\partial_tz\|^2+\|v^0\|_{_{\mathcal H}}^6+\|v^0\|_{_{\mathcal H}}^4
\right].
$$
This together with the Gronwall inequality implies that  
$$
\|z[T]\|_{_{\mathcal H}}\leq C(d^2+d)\|v^0\|_{_{\mathcal H}},
$$
which means
\begin{equation}\label{decay-1}
\|z[T]\|_{_{\mathcal H}}\leq \frac{\varepsilon}{2}\|v^0\|_{_{\mathcal H}},
\end{equation}
provided that $d=d(\varepsilon,T,R)\in(0,d_0)$ is sufficiently small. Finally, the combination of the first inequality in (\ref{decay-0}) and (\ref{decay-1}) gives rise to 
$$
\|w[T]\|_{_{\mathcal H}}\leq \varepsilon\|v^0\|_{_{\mathcal H}}.
$$
Theorem \ref{Th1} is then proved.

\section{Symbolic index}\label{Appendix-index}

In this appendix, we collect the most used symbols of the article, together with their meaning.

\vspace{0.3em}

{\footnotesize

\begin{longtable}{l|l}
                 
             \hline
             Functional analysis  &  Meaning \\
			 \hline 
            
            $D$, $\partial D$ & bounded domain in $\R^3$ with smooth boundary $\partial D$ \\  
            $\|\cdot\|$, $(\cdot,\cdot)$ & $(u,v)=\int_D uv$, $\|u\|=(u,u)^{1/2}$ for $u,v\in L^2(D)$\\  
            $H^s$, $H$ & domain of  $(-\Delta)^{s/2}$ with dual space $H^{-s}$ for $s\geq 0$; $H=L^2(D)$\\ 
            $\mathcal{H}^s$, $\mathcal{H}$ &  $\mathcal{H}^s=H^{1+s}\times H^s$, $s\in \R$; $\mathcal{H}=\mathcal{H}^0$ \\    
            $\mathcal{X}_T^s$, $\mathcal{X}_T$ & $\mathcal{X}_T^s=C([0,T];H^{1+s})\cap C^1([0,T];H^{s})$ with $T>0$, $s\in\R$; $\mathcal{X}_T=\mathcal{X}_T^0$\\ 
            $D_T$ & space-time domain, $D_T=(0,T)\times D$ with $T>0$\\                
            $\{e_j;j\in\N^+\}$, $\{\lambda_j;j\in\N^+\}$ & eigenvectors of $-\Delta$ with eigenvalues $\lambda_j$, forming an  orthonormal basis of $H$ \\  
            $\{\alpha^{\scriptscriptstyle T}_k;k\in\N^+\}$, $\{\alpha_k;k\in\N^+\}$ & smooth orthonormal basis of $L^2(0,T)$/$L^2(0,1)$; $\alpha^{\scriptscriptstyle T}_k(t)=T^{-1/2}\alpha_k(t/T)$ \\  
            $L^q_tL_x^r$, $L^q_tH_x^s$ & $L^q_tL_x^r=L^q(\tau,\tau+T;L^r(D))$, $L^q_tH_x^s=L^q(\tau,\tau+T;H^s)$ with $\tau\geq0$, $T>0$ \\    
            $\Gamma(x_0)$ &  portion of $\partial D$ satisfying $\Gamma(x_0)=\{x\in\partial D;(x-x_0)\cdot n(x)>0\}$ \\  
             $N_\delta(x_0)$ & $\delta$-neighborhood of boundary $\Gamma(x_0)$, $\{x\in D;|x-y|<\delta\ {\rm for\ some\ }y\in\Gamma(x_0)\}$ \\ 
            $a(x)$ &  nonnegative  $C^\infty(\overline{D})$ function supported by a $\Gamma$-type domain \\                    
            $\chi(x)$ & $C^\infty(\overline{D})$ cut-off function supported by a $\Gamma$-type domain \\                
            $u[t]$ & $u[t]=(u,\partial_tu)(t)$,  $t\geq 0$\\  
            $E(\psi)$,  $E_u(t)$ & $E(\psi_0,\psi_1)=\tfrac{1}{2}\| \psi_0\|^2_{{H^1}}+\tfrac{1}{2}\|\psi_1\|_{L^2}^2+\tfrac{1}{4} \|\psi_0\|_{L^4}^4$; $E_u(t)=E(u[t])$ \\  
            $C$ &  generic constant that may change from line to line \\
            $R$, $R_0$, $R_1$, $R_2$ & positive numbers; $R$/$R_0$ used in Theorem~\ref{Th1}/\ref{Thm-dynamicalsystem}, $R_1$,$R_2$ defined in Section~\ref{Section-conclusion}\\
            \hline
             Random wave equation &  \\
			 \hline  
          $\Box= \partial_{tt}^2- \Delta$ &  d'Alembert operator \\
            $\mathbf{T}$ & $\mathbf T= T_{1/4}$ with $T_{1/4}$ determined by \eqref{asymptotic-1},\eqref{largetime-2}; see also Section~\ref{Section-conclusion} \\
            $b_{jk}$ & nonnegative real numbers \\  
            $\theta^{n}_{jk}$, $\rho_{jk}$ & independent random variables $|\theta^{n}_{jk}|\leq 1$, $\theta^{n}_{jk}$ with $C^1$-density $\rho_{jk}$, $\rho_{jk}(0)>0$ \\ 
            $\eta_n(t,x)$ & i.i.d.~$L^2(D_T)$-valued random variables, $\eta_n(t,x)=\chi(x)\sum_{j\in\N^+}b_{jk}\theta^{n}_{jk}\alpha^{\scriptscriptstyle T}_k(t)e_j(x)$\\ 
            $\eta(t,x)$ & colored random noise $\eta(t,x)=\eta_n(t-nT,x)$ for $t\in[nT,(n+1)T)$, $n\in\N$ \\ 

            \hline
                Random dynamical system &  \\
			 \hline          
		   $(\mathcal{X},d),$ $\mathcal{Z}$ &  Polish spaces, i.e. complete separable metric spaces \\  
            
            $\boldsymbol{\xi}=(\xi_n;n\in\N)$, $\ell$, $\mathcal E$  & $\mathcal{Z}$-valued i.i.d.~random variables with common law $\ell$ and compact support $\mathcal{E}$\\   
             $S\colon\mathcal{X}\times\mathcal{Z}\rightarrow\mathcal{X}$ & continuous mapping   \\  $S_n(x;\boldsymbol{\xi})$ & $n$-th iteration of $S$ with $x\in\mathcal{X}$, $\boldsymbol{\xi}=(\xi_n;n\in\N)\in \mathcal{Z}^\N$ \\  
             $\mathcal{Y}_n$, $\mathcal{Y}_\infty$ & attainable sets of $\mathcal{Y}$, $\mathcal{Y}_n=\{S_n(x,\boldsymbol{\zeta});x\in\mathcal{S},\boldsymbol{\zeta}\in\mathcal{E}^{\N}\}$, $\mathcal{Y}_\infty=\overline{\cup_{n\in\N}\mathcal{Y}_n}$ \\  
            $B_{\mathcal{X}}(x,r)/B(x,r)$, $B_{\mathcal{X}}(r)$ &  open ball in $\mathcal{X}$ centered at $x$ with radius $r$; $B_{\mathcal{X}}(r)=B_{\mathcal{X}}(0,r)$ \\  
            $\overline{B}_{\mathcal X}(r)$ & closed ball centered at $0$ in $\mathcal X$, i.e. $\overline{B}_{\mathcal X}(r)=\overline{B_{\mathcal X}(r)}$ \\
            $\text{dist}_{\mathcal{X}}(x,A)$ & distance between $x\in\mathcal{X}$ and $A\subset \mathcal{X}$ \\  
             $\mathcal{B}(\mathcal{X})$ & Borel $\sigma$-algebra of $\mathcal{X}$  \\  
             $\mathcal{P}(\mathcal{X})$ & probability measures on $\mathcal{X}$, endowed with dual-Lipschitz norm $\|\cdot\|_{L}^*$  \\  
             $\text{supp }\mu$ & support of $\mu\in\mathcal{P}(\mathcal{X})$, $\text{supp }\mu=\{x\in\mathcal{X};\mu(B(x,r))>0 \text{ for any } r>0\}$\\  
             $\mathscr{D}(\xi)$ & law of random variable $\xi$\\  
            $\mathscr{C}(\mu,\nu)$ & couplings between $\mu,\nu\in \mathcal{P}(\mathcal{X})$\\  
			$B_b(\mathcal{X})$, $C_b(\mathcal{X})$, $L_b(\mathcal{X})$ & bounded Borel/continuous/Lipschitz functions on $\mathcal{X}$  \\  
             $\|f\|_{\infty}$  & supremum norm of $f\in B_b(\mathcal{X})$  \\  
            $\|f\|_{L}$  & Lipschitz norm of $f\in L_b(\mathcal{X})$, $\|f\|_{L}=\|f\|_{\infty} + \sup_{x\neq y}\frac{|f(x)-f(y)|}{d(x,y)}$ \\  
            $\langle f,\mu\rangle$& $\langle f,\mu\rangle=\int_{\mathcal{X}}f(x)\mu(dx)$ for $f\in B_b(\mathcal{X})$, $\mu\in \mathcal{P}(\mathcal{X})$  \\  
            $\|\cdot\|_L^*$ & $\|\mu-\nu\|_L^*=\sup\{|\langle f,\mu\rangle-\langle f,\nu\rangle|;f\in L_b(\mathcal{X}),\|f\|_L\leq 1\}$\\  
            $\Pb_x, \E_x$ & Markov family with $x\in\mathcal{X}$ and the corresponding expected value \\  
            $P_n(x,A)$ &  Markov transition functions with $x\in\mathcal{X}$, $A\in\mathcal{B}(\mathcal{X})$, $n\in\N$\\  
            $P_n$, $P_n^*$ & Markov semigroups on $B_b(\mathcal{X})$, $\mathcal{P}(\mathcal{X})$, respectively \\

            \hline
    	    Dynamical system &   \\
			 \hline
                 
             $U(t)$ & $C_0$-group generated by $\boxempty  v+a(x)\partial_t v=0$ \\   
             $B_R$ & $B_R=B_{C([0,T];H^{\scriptscriptstyle11/7})}(R)$ with $R>0$ \\  
            $F$ & $F=W^{1,\infty}(\mathbb R^+;H)\cap L^\infty(\mathbb R^+;H^{\scriptscriptstyle1/3})$ \\               
            $\mathcal U^f(t,\tau)(u_0,u_1)$ & solution of  (\ref{semilinear-problem-tau}) with $u[\tau]=(u_0,u_1)\in \mathcal{H}$, $t\geq \tau$ \\  
             $\mathscr{B}_0$, $\mathscr{B}_{\scriptscriptstyle4/7}$,  $\mathscr{B}_1$ & bounded sets of $\mathcal{H}$, $\mathcal{H}^{\scriptscriptstyle4/7}$, $\mathcal{H}^1$, respectively  \\
             \hline
             
	       Control theory &   \\
			 \hline
             $\mathcal L(\mathcal X;\mathcal Y)$, $\mathcal L(\mathcal X)$& bounded linear operators from $\mathcal X$ into $\mathcal Y$/$\mathcal X$ for Banach spaces $\mathcal{X},\mathcal{Y}$ \\   $\langle\cdot,\cdot\rangle_{\mathcal X,\mathcal X^*}$, $(\cdot,\cdot)_{_{\mathcal X}}$ & scalar product between $\mathcal X$, $\mathcal X^*$; inner product when $\mathcal X$ is a Hilbert space\\  
            $\mathcal H^s_*$, $\mathcal H_*$ &  $\mathcal H^s_*=H^{-1-s}\times H^{-s}$, $s\geq 0$; $\mathcal H_*=\mathcal H^0_*$ \\  
            $\mathscr P^{\scriptscriptstyle T}_{\scriptscriptstyle N}$ & projection of $L^2(D_T)$ onto  ${\rm span}\{ e_j\alpha^{\scriptscriptstyle T}_k;1\leq j,k\leq N\}$ \\              
            $\mathbf H_m$ & $\mathbf H_m=H_m\times H_m$ with $H_m={\rm span}\{e_j;1\leq j\leq m\}$ \\            
            ${\bf P}_m$ & projection of $\mathcal{H}$ onto  $\mathbf H_m$ \\ 
            $u^\bot[t]$ & $u^\bot[t]=(-\partial_t u,u)(t)$ with $u\in C^1([0,T];H^s)$, $t\geq 0$ \\       
            $\mathcal S(u_0,u_1,f)$ &   $\mathcal S(u_0,u_1,f)=u[\cdot]$ with $u\in\mathcal X_T$ being solution of (\ref{semilinear-problem})  \\       
            $\mathcal V_{\hat u}(v^0,f)$ & solution of (\ref{linear-problem}) with $b,p$ replaced by $a,3\hat u^2$, $\hat u\in B_R$ \\
            $\mathcal V^T_{\hat u}(v^T,f)$ & solution of (\ref{linear-problem}) with $b,p$ replaced by $a,3\hat u^2$ and terminal condition $v[T]=v^T$  \\
             $\mathcal W^T_{\hat u}(\varphi^T)$  &  solution of the adjoint system (\ref{adjoint-system-1})    \\          
            
            \hline

	\end{longtable}

    }

\end{appendices}

\noindent\textbf{Acknowledgments} \;
The authors would like to thank Yuxuan Chen for valuable discussions and suggestions during the preparation of the paper. 
Shengquan Xiang is partially  supported by  NSFC 12301562. Zhifei Zhang is  partially supported by  NSFC 12288101. Jia-Cheng Zhao is supported by China Postdoctoral Science Foundation 2024M750044.

	\normalem
      \bibliographystyle{abbrv}
      \bibliography{WaveMixing}

\begin{thebibliography}{100}

\bibitem{ADS-16}
K.~Ammari, T.~Duyckaerts, and A.~Shirikyan.
\newblock Local feedback stabilisation to a non-stationary solution for a
  damped non-linear wave equation.
\newblock {\em Math. Control Relat. Fields}, 6(1):1--25, 2016.

\bibitem{ALM-16}
N.~Anantharaman, M.~L\'eautaud, and F.~Maci\`a.
\newblock Wigner measures and observability for the {S}chr\"odinger equation on
  the disk.
\newblock {\em Invent. Math.}, 206(2):485--599, 2016.

\bibitem{BV-92}
A.~V. Babin and M.~I. Vishik.
\newblock {\em Attractors of evolution equations}.
\newblock North-Holland Publishing Co., Amsterdam, 1992.

\bibitem{BRS-11}
V.~Barbu, S.~S. Rodrigues, and A.~Shirikyan.
\newblock Internal exponential stabilization to a nonstationary solution for
  3{D} {N}avier-{S}tokes equations.
\newblock {\em SIAM J. Control Optim.}, 49(4):1454--1478, 2011.

\bibitem{Bardos-92}
C.~Bardos, G.~Lebeau, and J.~Rauch.
\newblock Sharp sufficient conditions for the observation, control, and
  stabilization of waves from the boundary.
\newblock {\em SIAM J. Control Optim.}, 30(5):1024--1065, 1992.

\bibitem{BBPS-21}
J.~Bedrossian, A.~Blumenthal, and S.~Punshon-Smith.
\newblock Almost-sure enhanced dissipation and uniform-in-diffusivity
  exponential mixing for advection-diffusion by stochastic {Navier-Stokes}.
\newblock {\em Probab. Theory Related Fields}, 179(3-4):777–834, 2021.

\bibitem{BBPS-22d}
J.~Bedrossian, A.~Blumenthal, and S.~Punshon-Smith.
\newblock Almost-sure exponential mixing of passive scalars by the stochastic
  {Navier-Stokes} equations.
\newblock {\em Ann. Probab.}, 50(1):241–303, 2022.

\bibitem{BBPS-22}
J.~Bedrossian, A.~Blumenthal, and S.~Punshon-Smith.
\newblock The {B}atchelor spectrum of passive scalar turbulence in stochastic
  fluid mechanics at fixed {R}eynolds number.
\newblock {\em Comm. Pure Appl. Math.}, 75(6):1237–1291, 2022.

\bibitem{BBPS-22b}
J.~Bedrossian, A.~Blumenthal, and S.~Punshon-Smith.
\newblock A regularity method for lower bounds on the {L}yapunov exponent for
  stochastic differential equations.
\newblock {\em Invent. Math.}, 227(2):429–516, 2022.

\bibitem{BSS-09}
M.~D. Blair, H.~F. Smith, and C.~D. Sogge.
\newblock Strichartz estimates for the wave equation on manifolds with
  boundary.
\newblock {\em Ann. Inst. H. Poincar\'{e} C Anal. Non Lin\'{e}aire},
  26(5):1817--1829, 2009.

\bibitem{BGN-23}
P.-M. Boulvard, P.~Gao, and V.~Nersesyan.
\newblock Controllability and ergodicity of three dimensional primitive
  equations driven by a finite-dimensional force.
\newblock {\em Arch. Ration. Mech. Anal.}, 247(1):Paper No. 2, 49 pp, 2023.

\bibitem{Bourgain-96}
J.~Bourgain.
\newblock Invariant measures for the {2D-defocusing nonlinear Schrödinger
  equation}.
\newblock {\em Comm. Math. Phys.}, 176(2):421--445, 1996.

\bibitem{Bourgain-14}
J.~Bourgain.
\newblock On the control problem for {S}chr\"odinger operators on tori.
\newblock In {\em Geometric aspects of functional analysis}, pages 97--105.
  Springer, Cham, 2014.

\bibitem{BBZ-13}
J.~Bourgain, N.~Burq, and M.~Zworski.
\newblock Control for {S}chr\"odinger operators on 2-tori: rough potentials.
\newblock {\em J. Eur. Math. Soc. (JEMS)}, 15(5):1597--1628, 2013.

\bibitem{BKL-02}
J.~Bricmont, A.~Kupiainen, and R.~Lefevere.
\newblock Exponential mixing of the 2{D} stochastic {N}avier-{S}tokes dynamics.
\newblock {\em Comm. Math. Phys.}, 230(1):87–132, 2002.

\bibitem{BDNY-24}
B.~Bringmann, Y.~Deng, A.~Nahmod, and H.~Yue.
\newblock Invariant {Gibbs} measures for the three dimensional cubic nonlinear
  wave equation.
\newblock {\em Invent. Math.}, 236(3):1133–1411, 2024.

\bibitem{BFZ-23}
Z.~Brze\'zniak, B.~Ferrario, and M.~Zanella.
\newblock Ergodic results for the stochastic nonlinear {S}chr\"odinger equation
  with large damping.
\newblock {\em J. Evol. Equ.}, 23(1):Paper No. 19, 31, 2023.

\bibitem{BGHS-21}
T.~Buckmaster, P.~Germain, Z.~Hani, and J.~Shatah.
\newblock Onset of the wave turbulence description of the longtime behavior of
  the nonlinear {Schrödinger} equation.
\newblock {\em Invent. Math.}, 225(3):787–855, 2021.

\bibitem{BG-97}
N.~Burq and P.~G\'erard.
\newblock Condition n\'ecessaire et suffisante pour la contr\^olabilit\'e{}
  exacte des ondes.
\newblock {\em C. R. Acad. Sci. Paris S\'er. I Math.}, 325(7):749--752, 1997.

\bibitem{BLP-08}
N.~Burq, G.~Lebeau, and F.~Planchon.
\newblock Global existence for energy critical waves in 3-{D} domains.
\newblock {\em J. Amer. Math. Soc.}, 21(3):831--845, 2008.

\bibitem{BT-08}
N.~Burq and N.~Tzvetkov.
\newblock Random data {Cauchy} theory for supercritical wave equations. {I.}
  local theory.
\newblock {\em Invent. Math.}, 173(3):449–475, 2008.

\bibitem{BT-08b}
N.~Burq and N.~Tzvetkov.
\newblock Random data {Cauchy} theory for supercritical wave equations. {II. A}
  global existence result.
\newblock {\em Invent. Math.}, 173(3):477–496, 2008.

\bibitem{BKS-20}
O.~Butkovsky, A.~Kulik, and M.~Scheutzow.
\newblock Generalized couplings and ergodic rates for {SPDE}s and other
  {M}arkov models.
\newblock {\em Ann. Appl. Probab.}, 30(1):41--55, 2020.

\bibitem{CV-02}
V.~V. Chepyzhov and M.~I. Vishik.
\newblock {\em Attractors for equations of mathematical physics}.
\newblock American Mathematical Society, Providence, RI, 2002.

\bibitem{CLT-08}
I.~Chueshov, I.~Lasiecka, and D.~Toundykov.
\newblock Long-term dynamics of semilinear wave equation with nonlinear
  localized interior damping and a source term of critical exponent.
\newblock {\em Discrete Contin. Dyn. Syst.}, 20(3):459--509, 2008.

\bibitem{Coron-07}
J.-M. Coron.
\newblock {\em Control and nonlinearity}.
\newblock American Mathematical Society, Providence, RI, 2007.

\bibitem{CC-2004}
J.-M. Coron and E.~Cr\'epeau.
\newblock Exact boundary controllability of a nonlinear {K}d{V} equation with
  critical lengths.
\newblock {\em J. Eur. Math. Soc. (JEMS)}, 6(3):367--398, 2004.

\bibitem{CKN-2024}
J.-M. Coron, A.~Koenig, and H.-M. Nguyen.
\newblock On the small-time local controllability of a {K}d{V} system for
  critical lengths.
\newblock {\em J. Eur. Math. Soc. (JEMS)}, 26(4):1193--1253, 2024.

\bibitem{Coron-Krieger-Xiang-1}
J.-M. Coron, J.~Krieger, and S.~Xiang.
\newblock Global controllability of a geometric wave equation.
\newblock {\em arxiv:2307.08329}, 2023.

\bibitem{CRX-2017}
J.-M. Coron, I.~Rivas, and S.~Xiang.
\newblock Local exponential stabilization for a class of {K}orteweg--de {V}ries
  equations by means of time-varying feedback laws.
\newblock {\em Anal. PDE}, 10(5):1089--1122, 2017.

\bibitem{CXZ-2023}
J.-M. Coron, S.~Xiang, and P.~Zhang.
\newblock On the global approximate controllability in small time of
  semiclassical 1-{D} {S}chr\"odinger equations between two states with
  positive quantum densities.
\newblock {\em J. Differential Equations}, 345:1--44, 2023.

\bibitem{DPZ-96}
G.~Da~Prato and J.~Zabczyk.
\newblock {\em {E}rgodicity for infinite dimensional systems}.
\newblock Cambridge University Press, Cambridge, 1996.

\bibitem{DO-05}
A.~Debussche and C.~Odasso.
\newblock Ergodicity for a weakly damped stochastic non-linear
  {S}chr{\"{o}}dinger equation.
\newblock {\em J. Evol. Equ.}, 5(3):317–356, 2005.

\bibitem{DGL-06}
B.~Dehman, P.~G\'{e}rard, and G.~Lebeau.
\newblock Stabilization and control for the nonlinear {S}chr\"{o}dinger
  equation on a compact surface.
\newblock {\em Math. Z.}, 254(4):729--749, 2006.

\bibitem{DL-09}
B.~Dehman and G.~Lebeau.
\newblock Analysis of the {HUM} control operator and exact controllability for
  semilinear waves in uniform time.
\newblock {\em SIAM J. Control Optim.}, 48(2):521--550, 2009.

\bibitem{DLZ-03}
B.~Dehman, G.~Lebeau, and E.~Zuazua.
\newblock Stabilization and control for the subcritical semilinear wave
  equation.
\newblock {\em Ann. Sci. \'{E}cole Norm. Sup.}, 36(4):525--551, 2003.

\bibitem{DZ-10}
A.~Dembo and O.~Zeitouni.
\newblock {\em Large deviations techniques and applications}.
\newblock Springer, Berlin, 2010.

\bibitem{DH-23}
Y.~Deng and Z.~Hani.
\newblock Full derivation of the wave kinetic equation.
\newblock {\em Invent. Math.}, 233(2):543--724, 2023.

\bibitem{DH-23a}
Y.~Deng and Z.~Hani.
\newblock Propagation of chaos and higher order statistics in wave kinetic
  theory.
\newblock {\em J. Eur. Math. Soc. (JEMS)}, to appear.

\bibitem{DNY-appear}
Y.~Deng, A.~Nahmod, and H.~Yue.
\newblock Invariant {Gibbs} measures and global strong solutions for nonlinear
  {Schrödinger} equations in dimension two.
\newblock {\em Ann. of Math.}, to appear.

\bibitem{Onofrio-13}
A.~d'Onofrio, editor.
\newblock {\em Bounded noises in physics, biology, and engineering}.
\newblock Modeling and Simulation in Science, Engineering and Technology.
  Birkh\"{a}user/Springer, New York, 2013.

\bibitem{EM-01}
W.~E and J.~C. Mattingly.
\newblock Ergodicity for the {N}avier-{S}tokes equation with degenerate random
  forcing: finite-dimensional approximation.
\newblock {\em Comm. Pure Appl. Math.}, 54(11):1386–1402, 2001.

\bibitem{EMS-01}
W.~E, J.~C. Mattingly, and Y.~Sinai.
\newblock Gibbsian dynamics and ergodicity for the stochastically forced
  {N}avier-{S}tokes equation.
\newblock {\em Comm. Math. Phys.}, 224(1):83–106, 2001.

\bibitem{EKZ-17}
I.~Ekren, I.~Kukavica, and M.~Ziane.
\newblock Existence of invariant measures for the stochastic damped
  {S}chrödinger equation.
\newblock {\em Stoch. Partial Differ. Equ. Anal.}, 5(3):343–367, 2017.

\bibitem{EKZ-18}
I.~Ekren, I.~Kukavica, and M.~Ziane.
\newblock Existence of invariant measures for the stochastic damped {KdV}
  equation.
\newblock {\em Indiana Univ. Math.}, 67(3):1221–1254, 2018.

\bibitem{ET-16}
M.~B. Erdo\u{g}an and N.~Tzirakis.
\newblock {\em Dispersive partial differential equations}.
\newblock Cambridge University Press, Cambridge, 2016.
\newblock Wellposedness and applications.

\bibitem{EZ-10}
S.~Ervedoza and E.~Zuazua.
\newblock A systematic method for building smooth controls for smooth data.
\newblock {\em Discrete Contin. Dyn. Syst. Ser. B}, 14(4):1375--1401, 2010.

\bibitem{FZ-93}
E.~Feireisl and E.~Zuazua.
\newblock Global attractors for semilinear wave equations with locally
  distributed nonlinear damping and critical exponent.
\newblock {\em Comm. Partial Differential Equations}, 18(9-10):1539--1555,
  1993.

\bibitem{FM-95}
F.~Flandoli and B.~Maslowski.
\newblock Ergodicity of the 2-{D} {N}avier-{S}tokes equation under random
  perturbations.
\newblock {\em Comm. Math. Phys.}, 172(1):119–141, 1995.

\bibitem{FP-67}
C.~Foia\c{s} and G.~Prodi.
\newblock Sur le comportement global des solutions non-stationnaires des
  \'equations de {N}avier-{S}tokes en dimension {$2$}.
\newblock {\em Rend. Sem. Mat. Univ. Padova}, 39:1--34, 1967.

\bibitem{FGRT-15}
J.~F\"{o}ldes, N.~E. Glatt-Holtz, G.~Richards, and E.~Thomann.
\newblock Ergodic and mixing properties of the {B}oussinesq equations with a
  degenerate random forcing.
\newblock {\em J. Funct. Anal.}, 269(8):2427–2504, 2015.

\bibitem{FT-24}
J.~Forlano and L.~Tolomeo.
\newblock On the unique ergodicity for a class of 2 dimensional stochastic wave
  equations.
\newblock {\em Trans. Amer. Math. Soc.}, 377(1):345–394, 2024.

\bibitem{GK-22}
P.~Gao and S.~Kuksin.
\newblock Weak and strong versions of the {K}olmogorov 4/5-law for stochastic
  {B}urgers equation.
\newblock {\em Arch. Ration. Mech. Anal.}, 247(6):Paper No. 109, 14, 2023.

\bibitem{GHM-18}
N.~E. Glatt-Holtz, D.~P. Herzog, and J.~C. Mattingly.
\newblock Scaling and saturation in infinite-dimensional control problems with
  applications to stochastic partial differential equations.
\newblock {\em Ann. PDE}, 4(2):16, 2018.

\bibitem{GHMR-23}
N.~E. Glatt-Holtz, V.~R. Martinez, and G.~H. Richards.
\newblock On the long-time statistical behavior of smooth solutions of the
  weakly damped, stochastically-driven {K}d{V} equation.
\newblock {\em arXiv:2103.12942}, 2023.

\bibitem{GM-05}
B.~Goldys and B.~Maslowski.
\newblock Exponential ergodicity for stochastic {B}urgers and 2{D}
  {N}avier-{S}tokes equations.
\newblock {\em J. Funct. Anal.}, 226(1):230–255, 2005.

\bibitem{Grillakis-90}
M.~G. Grillakis.
\newblock Regularity and asymptotic behaviour of the wave equation with a
  critical nonlinearity.
\newblock {\em Ann. of Math. (2)}, 132(3):485--509, 1990.

\bibitem{Hairer-02}
M.~Hairer.
\newblock Exponential mixing properties of stochastic {PDE}s through asymptotic
  coupling.
\newblock {\em Probab. Theory Related Fields}, 124(3):345–380, 2002.

\bibitem{HM-06}
M.~Hairer and J.~C. Mattingly.
\newblock Ergodicity of the 2{D} {N}avier-{S}tokes equations with degenerate
  stochastic forcing.
\newblock {\em Ann. of Math. (2)}, 164(3):993 -- 1032, 2006.

\bibitem{HM-08}
M.~Hairer and J.~C. Mattingly.
\newblock Spectral gaps in {W}asserstein distances and the 2{D} stochastic
  {N}avier-{S}tokes equations.
\newblock {\em Ann. Probab.}, 36(6):2050 -- 2091, 2008.

\bibitem{HM-11b}
M.~Hairer and J.~C. Mattingly.
\newblock A theory of hypoellipticity and unique ergodicity for semilinear
  stochastic {PDE}s.
\newblock {\em Electron. J. Probab.}, 16:658 -- 738, 2011.

\bibitem{HM-11}
M.~Hairer and J.~C. Mattingly.
\newblock Yet another look at {H}arris' ergodic theorem for {M}arkov chains.
\newblock In {\em Seminar on Stochastic Analysis, Random Fields and
  Applications VI}, pages 109 -- 117, Basel, 2011. Birkhäuser.

\bibitem{HMS-11}
M.~Hairer, J.~C. Mattingly, and M.~Scheutzow.
\newblock Asymptotic coupling and a general form of {H}arris’ theorem with
  applications to stochastic delay equations.
\newblock {\em Probab. Theory Related Fields}, 149(1):223--259, 2011.

\bibitem{Hale-88}
J.~K. Hale.
\newblock {\em Asymptotic Behavior of Dissipative Systems}.
\newblock American Mathematical Society, Providence, RI, 1988.

\bibitem{Harris-56}
T.~E. Harris.
\newblock The existence of stationary measures for certain {M}arkov processes.
\newblock In {\em Proceedings of the Third Berkeley Symposium on Mathematical
  Statistics and Probability 1954–1955 II}, pages 113--124, Berkeley, 1956.
  California Press.

\bibitem{JL-13}
R.~Joly and C.~Laurent.
\newblock Stabilization for the semilinear wave equation with geometric control
  condition.
\newblock {\em Anal. PDE}, 6(5):1089--1119, 2013.

\bibitem{KK-23}
B.~Keeler and P.~Kleinhenz.
\newblock Sharp exponential decay rates for anisotropically damped waves.
\newblock {\em Ann. Henri Poincar\'e}, 24(5):1561--1595, 2023.

\bibitem{KT-95}
H.~Koch and D.~Tataru.
\newblock On the spectrum of hyperbolic semigroups.
\newblock {\em Comm. Partial Differential Equations}, 20(5-6):901--937, 1995.

\bibitem{KW-12}
T.~Komorowski and A.~Walczuk.
\newblock Central limit theorem for {M}arkov processes with spectral gap in the
  {W}asserstein metric.
\newblock {\em Stochastic Process. Appl.}, 122(5):2155–2184, 2012.

\bibitem{Krieger-Xiang-2022}
J.~Krieger and S.~Xiang.
\newblock Semi-global controllability of a geometric wave equation.
\newblock {\em arXiv:2205.00915}, 2022.

\bibitem{KX-23}
J.~Krieger and S.~Xiang.
\newblock Boundary stabilization of the focusing {NLKG} equation near unstable
  equilibria: radial case.
\newblock {\em Pure Appl. Anal.}, 5(4):833--894, 2023.

\bibitem{KNS-20}
S.~Kuksin, V.~Nersesyan, and A.~Shirikyan.
\newblock Exponential mixing for a class of dissipative {PDE}s with bounded
  degenerate noise.
\newblock {\em Geom. Funct. Anal.}, 30(1):126--187, 2020.

\bibitem{KNS-20-1}
S.~Kuksin, V.~Nersesyan, and A.~Shirikyan.
\newblock Mixing via controllability for randomly forced nonlinear dissipative
  {PDE}s.
\newblock {\em J. \'{E}c. polytech. Math.}, 7:871--896, 2020.

\bibitem{KPS-02}
S.~Kuksin, A.~Piatnitski, and A.~Shirikyan.
\newblock A coupling approach to randomly forced nonlinear {PDEs. II.}
\newblock {\em Comm. Math. Phys.}, 230(1):81–85, 2002.

\bibitem{KS-00}
S.~Kuksin and A.~Shirikyan.
\newblock Stochastic dissipative {PDE}s and {G}ibbs measures.
\newblock {\em Comm. Math. Phys.}, 213(2):291–330, 2000.

\bibitem{KS-01}
S.~Kuksin and A.~Shirikyan.
\newblock A coupling approach to randomly forced nonlinear {PDE's I}.
\newblock {\em Comm. Math. Phys.}, 221(2):351--366, 2001.

\bibitem{KS-02}
S.~Kuksin and A.~Shirikyan.
\newblock Coupling approach to white-forced nonlinear {PDEs}.
\newblock {\em J. Math. Pures Appl. (9)}, 81(6):567–602, 2002.

\bibitem{KS-12}
S.~Kuksin and A.~Shirikyan.
\newblock {\em Mathematics of two-dimensional turbulence}.
\newblock Cambridge University Press, Cambridge, 2012.

\bibitem{KS-15}
A.~Kulik and M.~Scheutzow.
\newblock A coupling approach to {Doob's} theorem.
\newblock {\em Atti Accad. Naz. Lincei Rend. Lincei Mat. Appl.}, 26(1):83–92,
  2015.

\bibitem{Lady-72}
O.~A. Ladyzhenskaya.
\newblock The dynamical system that is generated by the {N}avier-{S}tokes
  equations.
\newblock {\em Zap. Nau\v cn. Sem. Leningrad. Otdel. Mat. Inst. Steklov.
  (LOMI)}, 27:91--115, 1972.

\bibitem{LT-93}
I.~Lasiecka and D.~Tataru.
\newblock Uniform decay rates for semilinear wave equations with nonlinear and
  nonmonotone boundary feedback, without geometric conditions.
\newblock In {\em Differential equations in {B}anach spaces ({B}ologna, 1991)},
  volume 148, pages 129--139. Dekker, New York, 1993.

\bibitem{Laurent-10}
C.~Laurent.
\newblock Global controllability and stabilization for the nonlinear
  {S}chr\"odinger equation on some compact manifolds of dimension 3.
\newblock {\em SIAM J. Math. Anal.}, 42(2):785--832, 2010.

\bibitem{LRZ-10}
C.~Laurent, L.~Rosier, and B.-Y. Zhang.
\newblock Control and stabilization of the {K}orteweg-de {V}ries equation on a
  periodic domain.
\newblock {\em Comm. Partial Differential Equations}, 35(4):707--744, 2010.

\bibitem{LT-12}
H.~Lindblad and T.~Tao.
\newblock Asymptotic decay for a one-dimensional nonlinear wave equation.
\newblock {\em Anal. PDE}, 5(2):411--422, 2012.

\bibitem{Lions-88}
J.-L. Lions.
\newblock {\em Contr\^{o}labilit\'{e} exacte, perturbations et stabilisation de
  syst\`{e}mes distribu\'{e}s. Tome 1}.
\newblock Masson, Paris, 1988.

\bibitem{LS-23}
J.~L\"uhrmann and W.~Schlag.
\newblock Asymptotic stability of the sine-{G}ordon kink under odd
  perturbations.
\newblock {\em Duke Math. J.}, 172(14):2715--2820, 2023.

\bibitem{Martirosyan-14}
D.~Martirosyan.
\newblock Exponential mixing for the white-forced damped nonlinear wave
  equation.
\newblock {\em Evol. Equ. Control Theory}, 3(4):645--670, 2014.

\bibitem{Martirosyan-17}
D.~Martirosyan.
\newblock Large deviations for stationary measures of stochastic nonlinear wave
  equations with smooth white noise.
\newblock {\em Comm. Pure Appl. Math.}, 70(9):1754--1797, 2017.

\bibitem{MMN-02}
A.~Marzocchi, J.~E.~M. noz Rivera, and M.~G. Naso.
\newblock Asymptotic behaviour and exponential stability for a transmission
  problem in thermoelasticity.
\newblock {\em Math. Methods Appl. Sci.}, 25(11):955--980, 2002.

\bibitem{MY-02}
N.~Masmoudi and L.-S. Young.
\newblock Ergodic theory of infinite dimensional systems with applications to
  dissipative parabolic {PDE}s.
\newblock {\em Comm. Math. Phys.}, 227(3):461–481, 2002.

\bibitem{Mattingly-02}
J.~C. Mattingly.
\newblock Exponential convergence for the stochastically forced
  {N}avier-{S}tokes equations and other partially dissipative dynamics.
\newblock {\em Comm. Math. Phys.}, 230(3):421--462, 2002.

\bibitem{MW-00}
M.~Maxwell and M.~Woodroofe.
\newblock Central limit theorems for additive functionals of {M}arkov chains.
\newblock {\em Ann. Probab.}, 28(2):713–724, 2000.

\bibitem{MT-09}
S.~P. Meyn and R.~L. Tweedie.
\newblock {\em Markov chains and stochastic stability}.
\newblock Cambridge University Press, Cambridge, second edition, 2009.

\bibitem{Ner-22}
V.~Nersesyan.
\newblock Ergodicity for the randomly forced {N}avier-{S}tokes system in a
  two-dimensional unbounded domain.
\newblock {\em Ann. Henri Poincar\'{e}}, 23(6):2277--2294, 2022.

\bibitem{Ner-24}
V.~Nersesyan.
\newblock The complex {Ginzburg-Landau} equation perturbed by a force localised
  both in physical and {Fourier} spaces.
\newblock {\em Ann. Sc. Norm. Super. Pisa Cl. Sci. (5)}, XXV:1203--1223, 2024.

\bibitem{RZ-98}
L.~Robbiano and C.~Zuily.
\newblock Uniqueness in the {C}auchy problem for operators with partially
  holomorphic coefficients.
\newblock {\em Invent. Math.}, 131(3):493--539, 1998.

\bibitem{LLTT-17}
J.~L. Rousseau, G.~Lebeau, P.~Terpolilli, and E.~Tr\'elat.
\newblock Geometric control condition for the wave equation with a
  time-dependent observation domain.
\newblock {\em Anal. PDE}, 10(4):983--1015, 2017.

\bibitem{Shi-06}
A.~Shirikyan.
\newblock Law of large numbers and central limit theorem for randomly forced
  {PDE}'s.
\newblock {\em Probab. Theory Related Fields}, 134(2):215–247, 2006.

\bibitem{Shi-08}
A.~Shirikyan.
\newblock Exponential mixing for randomly forced partial differential
  equations: {M}ethod of coupling.
\newblock In {\em Instability in models connected with fluid flows. II , Int.
  Math. Ser. (N. Y.)}, pages 155--188, New York, 2008. Springer.

\bibitem{Shi-15}
A.~Shirikyan.
\newblock Control and mixing for 2{D} {N}avier-{S}tokes equations with
  space-time localised noise.
\newblock {\em Ann. Sci. \'{E}c. Norm. Sup\'{e}r}, 48(2):253--280, 2015.

\bibitem{Shi-17}
A.~Shirikyan.
\newblock Controllability implies mixing {I}. {C}onvergence in the total
  variation metric.
\newblock {\em Russian Math. Surveys}, 72(5):1381--1422, 2017.

\bibitem{Shi-21}
A.~Shirikyan.
\newblock Controllability implies mixing {II}. {C}onvergence in the
  dual-{L}ipschitz metric.
\newblock {\em J. Eur. Math. Soc. (JEMS)}, 23(4):1381--1422, 2021.

\bibitem{Thorisson-00}
H.~Thorisson.
\newblock {\em Coupling, stationarity and regeneration}.
\newblock Springer, New York, 2000.

\bibitem{Tolomeo-20}
L.~Tolomeo.
\newblock Unique ergodicity for a class of stochastic hyperbolic equations with
  additive space-time white noise.
\newblock {\em Comm. Math. Phys.}, 377(2):1311–134, 2020.

\bibitem{Tolomeo-23}
L.~Tolomeo.
\newblock Ergodicity for the hyperbolic {$P(\Phi)_2$-model}.
\newblock {\em arXiv:2310.02190}, 2023.

\bibitem{TW-18}
P.~Tsatsoulis and H.~Weber.
\newblock Spectral gap for the stochastic quantization equation on the
  2-dimensional torus.
\newblock {\em Ann. Inst. Henri Poincar\'{e} Probab. Stat.}, 54(3):1204–1249,
  2018.

\bibitem{Villani-08}
C.~Villani.
\newblock {\em Optimal transport. {O}ld and new}.
\newblock Springer, Berlin, 2008.

\bibitem{Xiang-23}
S.~Xiang.
\newblock Small-time local stabilization of the two-dimensional incompressible
  {N}avier-{S}tokes equations.
\newblock {\em Ann. Inst. H. Poincar\'{e} C Anal. Non Lin\'{e}aire},
  40(6):1487--1511, 2023.

\bibitem{Xiang-24}
S.~Xiang.
\newblock Quantitative rapid and finite time stabilization of the heat
  equation.
\newblock {\em ESAIM Control Optim. Calc. Var.}, 30:Paper No. 40, 25, 2024.

\bibitem{Zelik-04}
S.~Zelik.
\newblock Asymptotic regularity of solutions of a nonautonomous damped wave
  equation with a critical growth exponent.
\newblock {\em Commun. Pure Appl. Anal.}, 3(4):921--934, 2004.

\bibitem{Zhang-PRSL}
X.~Zhang.
\newblock Explicit observability estimate for the wave equation with potential
  and its application.
\newblock {\em Proc. R. Soc. Lond. A}, 456:1101--1115, 2000.

\bibitem{Zhang-00}
X.~Zhang.
\newblock Explicit observability inequalities for the wave equation with lower
  order terms by means of {C}arleman inequalities.
\newblock {\em SIAM J. Control Optim.}, 39(3):812--834, 2000.

\bibitem{Zuazua-90}
E.~Zuazua.
\newblock Exponential decay for the semilinear wave equation with locally
  distributed damping.
\newblock {\em Comm. Partial Differential Equations}, 15:205--235, 1990.

\end{thebibliography}

\medskip	
	
\end{document}